\documentclass[12pt]{amsart}
\usepackage[T1]{fontenc}
\usepackage{amssymb}
\usepackage{a4wide}
\usepackage{graphicx}
\usepackage{verbatim}
\usepackage{amsmath}
\usepackage{version}
\usepackage{amsfonts}
\usepackage{color}
\providecommand{\U}[1]{\protect\rule{.1in}{.1in}}

\newtheorem{theorem}{Theorem}[section]

\newtheorem{ass}[theorem]{Assumption}

\newtheorem{corollary}[theorem]{Corollary}

\newtheorem{definition}[theorem]{Definition}
\newtheorem{example}[theorem]{Example}

\newtheorem{lemma}[theorem]{Lemma}

\newtheorem{proposition}[theorem]{Proposition}
\newtheorem{remark}[theorem]{Remark}

\begin{document}
\title[Semiclassical measures and integrable systems]{Semiclassical completely integrable systems : long-time dynamics and observability via two-microlocal Wigner measures}
\author[N. Anantharaman]{Nalini Anantharaman}
\address[N. Anantharaman]{Universit\'{e} Paris 11, Math\'{e}matiques, B\^{a}t. 425, 91405
ORSAY
Cedex, FRANCE} \email{Nalini.Anantharaman@math.u-psud.fr}
\author[C. Fermanian]{Clotilde~Fermanian-Kammerer}
\address[C. Fermanian]{LAMA UMR CNRS 8050,
Universit\'e Paris Est - Cr\'eteil \\
61, avenue du G\'en\'eral de Gaulle\\
94010 CR\'ETEIL Cedex\\ FRANCE}
\email{clotilde.fermanian@u-pec.fr}
\author[F. Maci\`a]{Fabricio Maci\`a}
\address[F. Maci\`a]{Universidad Polit\'{e}cnica de Madrid. DCAIN, ETSI Navales. Avda. Arco de la
Victoria s/n. 28040 MADRID, SPAIN} \email{Fabricio.Macia@upm.es}
\thanks{N. Anantharaman wishes to acknowledge the support of Agence
Nationale de la Recherche, under the grant ANR-09-JCJC-0099-01. F.
Maci{\`a} takes part into the visiting faculty program of ICMAT
and is partially supported by grants MTM2010-16467 (MEC), ERC
Starting Grant 277778}

\begin{abstract}
We look at the long-time behaviour of
solutions to a semi-classical Schr\"odinger equation on the
torus. We consider time scales which go to infinity when the
semi-classical parameter goes to zero and we associate with each
time-scale the set of semi-classical measures associated with all
possible choices of initial data. On each classical invariant torus, the structure of semi-classical measures is described in terms of two-microlocal measures, obeying explicit propagation laws.

We apply this construction in two directions. We first analyse the regularity of semi-classical measures, and
we emphasize the existence of a
threshold~: for  time-scales below this threshold,  the set of
semi-classical measures contains measures which are singular with
respect to Lebesgue measure in the ``position'' variable, while at
(and beyond) the threshold, all  the semi-classical measures are
absolutely continuous in the ``position'' variable, reflecting the dispersive properties of the equation.
Second, the techniques of two-microlocal analysis introduced in the paper are used to prove semiclassical observability estimates. The results apply as well to general quantum completely integrable systems.
\end{abstract}
\maketitle

\newcommand{\nwc}{\newcommand}
\nwc{\nwt}{\newtheorem}
\nwt{coro}{Corollary}
\nwt{ex}{Example}
\nwt{prop}{Proposition}
\nwt{defin}{Definition}

%font change

\nwc{\mf}{\mathbf} %Latex (as in \bf not tilted math letters)
\nwc{\blds}{\boldsymbol} %Latex 
\nwc{\ml}{\mathcal} %Latex

%greek letters

\nwc{\lam}{\lambda}
\nwc{\del}{\delta}
\nwc{\Del}{\Delta}
\nwc{\Lam}{\Lambda}
\nwc{\elll}{\ell}
%blackboard bold math

\nwc{\IA}{\mathbb{A}} %algebraic
\nwc{\IB}{\mathbb{B}} %ball
\nwc{\IC}{\mathbb{C}} %complex
\nwc{\ID}{\mathbb{D}} %Dedekind
\nwc{\IE}{\mathbb{E}} %Euklides
\nwc{\IF}{\mathbb{F}} %finite field
\nwc{\IG}{\mathbb{G}} %Gauss
\nwc{\IH}{\mathbb{H}} %Hilbert\N-subgroup
\nwc{\IN}{\mathbb{N}} %natural
\nwc{\IP}{\mathbb{P}} %prime
\nwc{\IQ}{\mathbb{Q}} %rational
\nwc{\IR}{\mathbb{R}} %real
\nwc{\IS}{\mathbb{S}} %sphere
\nwc{\IT}{\mathbb{T}} %torus
\nwc{\IZ}{\mathbb{Z}} %integers
\def\bbbone{{\mathchoice {1\mskip-4mu {\rm{l}}} {1\mskip-4mu {\rm{l}}}
{ 1\mskip-4.5mu {\rm{l}}} { 1\mskip-5mu {\rm{l}}}}}
\def\bbleft{{\mathchoice {[\mskip-3mu {[}} {[\mskip-3mu {[}}{[\mskip-4mu {[}}{[\mskip-5mu {[}}}}
\def\bbright{{\mathchoice {]\mskip-3mu {]}} {]\mskip-3mu {]}}{]\mskip-4mu {]}}{]\mskip-5mu {]}}}}
\nwc{\setK}{\bbleft 1,K \bbright}
\nwc{\setN}{\bbleft 1,\cN \bbright}
 \newcommand{\Lim}{\mathop{\longrightarrow}\limits}
%Straight (vector) bold letters

%lowercase

\nwc{\va}{{\bf a}}
\nwc{\vb}{{\bf b}}
\nwc{\vc}{{\bf c}}
\nwc{\vd}{{\bf d}}
\nwc{\ve}{{\bf e}}
\nwc{\vf}{{\bf f}}
\nwc{\vg}{{\bf g}}
\nwc{\vh}{{\bf h}}
\nwc{\vi}{{\bf i}}
\nwc{\vI}{{\bf I}}
\nwc{\vj}{{\bf j}}
\nwc{\vk}{{\bf k}}
\nwc{\vl}{{\bf l}}
\nwc{\vm}{{\bf m}}
\nwc{\vM}{{\bf M}}
\nwc{\vn}{{\bf n}}
\nwc{\vo}{{\it o}}
\nwc{\vp}{{\bf p}}
\nwc{\vq}{{\bf q}}
\nwc{\vr}{{\bf r}}
\nwc{\vs}{{\bf s}}
\nwc{\vt}{{\bf t}}
\nwc{\vu}{{\bf u}}
\nwc{\vv}{{\bf v}}
\nwc{\vw}{{\bf w}}
\nwc{\vx}{{\bf x}}
\nwc{\vy}{{\bf y}}
\nwc{\vz}{{\bf z}}
\nwc{\bal}{\blds{\alpha}}
\nwc{\bep}{\blds{\epsilon}}
\nwc{\barbep}{\overline{\blds{\epsilon}}}
\nwc{\bnu}{\blds{\nu}}
\nwc{\bmu}{\blds{\mu}}
\nwc{\bet}{\blds{\eta}}

%bold letters
%\b* letters are tilted in math mode and scale in equations. 
%but cannot be used in plain text format.

%I. lowercase

\nwc{\bk}{\blds{k}}
\nwc{\bm}{\blds{m}}
\nwc{\bM}{\blds{M}}
\nwc{\bp}{\blds{p}}
\nwc{\bq}{\blds{q}}
\nwc{\bn}{\blds{n}}
\nwc{\bv}{\blds{v}}
\nwc{\bw}{\blds{w}}
\nwc{\bx}{\blds{x}}
\nwc{\bxi}{\blds{\xi}}
\nwc{\by}{\blds{y}}
\nwc{\bz}{\blds{z}}

%caligraphic

\nwc{\cA}{\ml{A}}
\nwc{\cB}{\ml{B}}
\nwc{\cC}{\ml{C}}
\nwc{\cD}{\ml{D}}
\nwc{\cE}{\ml{E}}
\nwc{\cF}{\ml{F}}
\nwc{\cG}{\ml{G}}
\nwc{\cH}{\ml{H}}
\nwc{\cI}{\ml{I}}
\nwc{\cJ}{\ml{J}}
\nwc{\cK}{\ml{K}}
\nwc{\cL}{\ml{L}}
\nwc{\cM}{\ml{M}}
\nwc{\cN}{\ml{N}}
\nwc{\cO}{\ml{O}}
\nwc{\cP}{\ml{P}}
\nwc{\cQ}{\ml{Q}}
\nwc{\cR}{\ml{R}}
\nwc{\cS}{\ml{S}}
\nwc{\cT}{\ml{T}}
\nwc{\cU}{\ml{U}}
\nwc{\cV}{\ml{V}}
\nwc{\cW}{\ml{W}}
\nwc{\cX}{\ml{X}}
\nwc{\cY}{\ml{Y}}
\nwc{\cZ}{\ml{Z}}

\nwc{\fA}{\mathfrak{a}}
\nwc{\fB}{\mathfrak{b}}
\nwc{\fC}{\mathfrak{c}}
\nwc{\fD}{\mathfrak{d}}
\nwc{\fE}{\mathfrak{e}}
\nwc{\fF}{\mathfrak{f}}
\nwc{\fG}{\mathfrak{g}}
\nwc{\fH}{\mathfrak{h}}
\nwc{\fI}{\mathfrak{i}}
\nwc{\fJ}{\mathfrak{j}}
\nwc{\fK}{\mathfrak{k}}
\nwc{\fL}{\mathfrak{l}}
\nwc{\fM}{\mathfrak{m}}
\nwc{\fN}{\mathfrak{n}}
\nwc{\fO}{\mathfrak{o}}
\nwc{\fP}{\mathfrak{p}}
\nwc{\fQ}{\mathfrak{q}}
\nwc{\fR}{\mathfrak{r}}
\nwc{\fS}{\mathfrak{s}}
\nwc{\fT}{\mathfrak{t}}
\nwc{\fU}{\mathfrak{u}}
\nwc{\fV}{\mathfrak{v}}
\nwc{\fW}{\mathfrak{w}}
\nwc{\fX}{\mathfrak{x}}
\nwc{\fY}{\mathfrak{y}}
\nwc{\fZ}{\mathfrak{z}}

%% (wide)tilde letters

\nwc{\tA}{\widetilde{A}}
\nwc{\tB}{\widetilde{B}}
\nwc{\tE}{E^{\vareps}}
%\nwc{\tcO}{\widetilde{\mathcal{O}}}
\nwc{\tk}{\tilde k}
\nwc{\tN}{\tilde N}
\nwc{\tP}{\widetilde{P}}
\nwc{\tQ}{\widetilde{Q}}
\nwc{\tR}{\widetilde{R}}
\nwc{\tV}{\widetilde{V}}
\nwc{\tW}{\widetilde{W}}
\nwc{\ty}{\tilde y}
\nwc{\teta}{\tilde \eta}
\nwc{\tdelta}{\tilde \delta}
\nwc{\tlambda}{\tilde \lambda}
%\nwc{\tchi}{\tilde \chi}
\nwc{\ttheta}{\tilde \theta}
\nwc{\tvartheta}{\tilde \vartheta}
\nwc{\tPhi}{\widetilde \Phi}
\nwc{\tpsi}{\tilde \psi}
\nwc{\tmu}{\tilde \mu}

%miscellany
\nwc{\To}{\longrightarrow} %limits

\nwc{\ad}{\rm ad}
\nwc{\eps}{\epsilon}
\nwc{\ep}{\epsilon}
\nwc{\vareps}{\varepsilon}

\def\bom{\mathbf{\omega}}
\def\om{{\omega}}
\def\ep{\epsilon}
\def\tr{{\rm tr}}
\def\diag{{\rm diag}}
\def\Tr{{\rm Tr}}
\def\i{{\rm i}}
\def\mi{{\rm i}}
\def\e{{\rm e}}
\def\sq2{\sqrt{2}}
\def\sqn{\sqrt{N}}
\def\vol{\mathrm{vol}}
\def\defi{\stackrel{\rm def}{=}}
\def\t2{{\mathbb T}^2}
%\def\tt2{{\mathbb T}^2}
%\nwc{\t1}{{\mathbb T}^1}
\def\s2{{\mathbb S}^2}
\def\hn{\mathcal{H}_{N}}
\def\shbar{\sqrt{\hbar}}
\def\A{\mathcal{A}}
\def\N{\mathbb{N}}
\def\T{\mathbb{T}}
\def\R{\mathbb{R}}
\def\RR{\mathbb{R}}
\def\Z{\mathbb{Z}}
\def\C{\mathbb{C}}
\def\O{\mathcal{O}}
\def\Sp{\mathcal{S}_+}
\def\Lap{\triangle}
\nwc{\lap}{\bigtriangleup}
\nwc{\rest}{\restriction}
\nwc{\Diff}{\operatorname{Diff}}
\nwc{\diam}{\operatorname{diam}}
\nwc{\Res}{\operatorname{Res}}
\nwc{\Spec}{\operatorname{Spec}}
\nwc{\Vol}{\operatorname{Vol}}
\nwc{\Op}{\operatorname{Op}}
\nwc{\supp}{\operatorname{supp}}
\nwc{\Span}{\operatorname{span}}

\nwc{\dia}{\varepsilon}
\nwc{\cut}{f}
\nwc{\qm}{u_\hbar}

\def\hto0{\xrightarrow{\hbar\to 0}}
\def\htoo{\stackrel{h\to 0}{\longrightarrow}}
\def\rto0{\xrightarrow{r\to 0}}
\def\rtoo{\stackrel{r\to 0}{\longrightarrow}}
\def\ntoinf{\xrightarrow{n\to +\infty}}

\providecommand{\abs}[1]{\lvert#1\rvert}
\providecommand{\norm}[1]{\lVert#1\rVert}
\providecommand{\set}[1]{\left\{#1\right\}}

\nwc{\la}{\langle}
\nwc{\ra}{\rangle}
\nwc{\lp}{\left(}
\nwc{\rp}{\right)}

%\nwc{\bal}{\begin{align}}
\nwc{\bequ}{\begin{equation}}
\nwc{\be}{\begin{equation}}
\nwc{\ben}{\begin{equation*}}
\nwc{\bea}{\begin{eqnarray}}
\nwc{\bean}{\begin{eqnarray*}}
\nwc{\bit}{\begin{itemize}}
\nwc{\bver}{\begin{verbatim}}

%\nwc{\eal}{\end{align}}
\nwc{\eequ}{\end{equation}}
\nwc{\ee}{\end{equation}}
\nwc{\een}{\end{equation*}}
\nwc{\eea}{\end{eqnarray}}
\nwc{\eean}{\end{eqnarray*}}
\nwc{\eit}{\end{itemize}}
\nwc{\ever}{\end{verbatim}}

\newcommand{\defeq}{\stackrel{\rm{def}}{=}}

\section{Introduction}

\subsection{The Schr\"odinger equation in the large time and high frequency
r\'egime}

This article is concerned with the dynamics of the linear equation
\begin{equation}
\left\{
\begin{array}
[c]{l} ih\partial_{t}\psi_{h}\left(  t,x\right)
=\left(H(hD_{x})+h^2 \mathbf{V}_h(t) \right)\,\psi_{h}\left(
t,x\right)  ,\qquad(t,x)\in\R\times\T^{d},\\
{\psi_{h}}_{|t=0}=u_{h},
\end{array}
\right.  \label{e:eq}
\end{equation}
on the torus $\mathbb{T}^{d} :=\left(
\mathbb{R/}2\pi\mathbb{Z}\right)  ^{d}$, with $H$ a smooth,
real-valued function on $(\R^{d})^*$ (the dual of $\R^{d}$), and
$h>0$. In other words, $H$ is a function on the cotangent bundle
$T^*\T^d=\T^d\times (\R^{d})^*$ that does not depend on the $d$ first
variables, and thus gives rise to a completely integrable
Hamiltonian flow. For the sake of simplicity, we shall assume that
$H\in\mathcal{C}^{\infty}\left(  \mathbb{R}^{d}\right)  $. However
the smoothness assumption on $H$ can be relaxed to
$\mathcal{C}^{k}$, where $k$ large enough, in most results of this
article. The lower order term $\mathbf{V}_h(t)$ is a
bounded self-adjoint operator (possibly depending on $t$ and $h$).
We assume that the map $t\mapsto \| \mathbf{V}_h(t)\|_{{\mathcal
L}(L^2(\T^d))}$ is  in $L^1_{loc}(\R)\cap L^\infty(\R)$, uniformly with respect to
$h$. This condition ensures the existence of a semi-group
associated with the operator $H(hD_x)+h^2\mathbf{V}_h(t)$ (see
Appendice B in  \cite{D-G}, Proposition B.3.6).

We are interested in the simultaneous limits $h\rightarrow0^{+}$
(high frequency limit) and $t\rightarrow+ \infty$ (large time evolution). Our
results give a description of the limits of sequences of \textquotedblleft
position densities\textquotedblright\ $\left\vert \psi_{h}\left(
t_{h},x\right)  \right\vert ^{2}$ at times $t_{h}$ that tend to infinity as
$h\rightarrow0^{+}$.

\begin{remark}\label{r:Hh}In future applications, it will be interesting to note as of now that we may allow $H=H_h$ to depend on the parameter $h$, in such a way that
$H_h$ converges to some limit $H_0$ in the $\cC^k$ topology on
compact sets, for $k$ sufficiently large. For instance, we allow
$H_h(\xi)=H(\xi + h \omega)$, where $\omega\in (\R^{d})^*$ is a
fixed vector.\end{remark}

To be more specific, let us denote by $S_{h}(t,s)$ the semigroup
associated with the operator $H(hD_{x})+h^2\mathbf{V}_h(t)$ and set
$S^t_h=S_h(t,0)$.
Fix a $\emph{time}$ \emph{scale}, that is, a function
\begin{align*}
\tau:\IR_{+}^{\ast}  &  \longrightarrow\IR_{+}^{\ast}\\
h  &  \longmapsto\tau_{h},
\end{align*}
such that $\liminf_{h\rightarrow0^{+}}\tau
_{h}>0$ (actually, we shall be mainly concerned in scales that go to
$+\infty$ as $h\rightarrow0^{+}$). Consider a family of initial conditions
$(u_{h})$, normalised in $L^{2}(\mathbb{T}^{d})$: $\left\Vert u_{h}\right\Vert
_{L^{2}\left(  \mathbb{T}^{d}\right)  }=1$ for $h>0$, and $h$-oscillating in
the terminology of \cite{GerardMesuresSemi91, GerLeich93}, \emph{i.e.}:
\begin{equation}
\limsup_{h\rightarrow0^{+}}\left\Vert \mathbf{1}_{\left[  R,+\infty\right[  }\left(
-h^{2}\Delta\right)  u_{h}\right\Vert _{L^{2}\left(  \mathbb{T}^{d}\right)
}\Lim_{R\To\infty}0,
 \label{e:hosc}
\end{equation}
where $\mathbf{1}_{\left[  R,+\infty\right[  }$ is the characteristic function of
the interval $\left[  R,+\infty\right[  $. Our main object of interest is the
density $\left\vert S_{h}^{t}u_{h}\right\vert ^{2}$, and we introduce the
probability measures on $\mathbb{T}^{d}$:
\[
\mathbb{\nu}_{h}\left(  t,dx\right)  :=\left\vert S_{h}^{t}u_{h}(x)\right\vert
^{2}dx;
\]
the unitary character of $S_{h}^{t}$ implies that $\nu_{h}\in\mathcal{C}
\left(  \mathbb{R};\mathcal{P}\left(  \mathbb{T}^{d}\right)  \right)
$ (in what follows, $\mathcal{P}\left(  X\right)  $ stands for the
set of probability measures on a Polish space $X$).

To study the long-time behavior of the dynamics, we rescale time
by $\tau_{h}$ and look at the time-scaled probability densities:
\begin{equation}
\nu_{h}\left(  \tau_{h}t,dx\right)  . \label{e:nuht}
\end{equation}
When $t\not =0$ is fixed and $\tau_{h}$ grows too rapidly, it is a notoriously
difficult problem to obtain a description of the limit points (in the
weak-$\ast$ topology) of these probability measures as $h\rightarrow0^{+}$,
for rich enough families of initial data $u_{h}$. See for instance
\cite{Schub-largetimes, Paul11} in the case where the underlying classical
dynamics is chaotic, the $u_{h}$ are a family of lagrangian states, and
$\tau_{h}=h^{-2+\epsilon}$. In completely integrable situations, such as the
one we consider here, the problem is of a different nature, but rapidly leads
to intricate number theoretical issues \cite{MarklofPoisson, MarklofPoisson2,
MarklofSquares}.

We soften the problem by considering the family of probability measures
\eqref{e:nuht} as elements of $L^{\infty}\left(  \mathbb{R};\mathcal{P}\left(
\mathbb{T}^{d}\right)  \right)  $. Our goal will be to give a precise
description of the set $\mathcal{M}\left(  \tau\right)  $ of their
accumulation points in the weak-$\ast$ topology for $L^{\infty}\left(
\mathbb{R};\mathcal{P}\left(  \mathbb{T}^{d}\right)  \right)  $, obtained as
$\left(  u_{h}\right)  $ varies among all possible sequences of initial data
$h$-oscillating and normalised in $L^{2}\left(  \mathbb{T}^{d}\right)  $.

The compactness of $\mathbb{T}^{d}$ ensures that $\mathcal{M}\left(
\tau\right)  $ is non-empty. Having $\nu\in\mathcal{M}\left(  \tau\right)  $
is equivalent to the existence of a sequence $(h_{n})$ going to $0$ and of a
normalised, $h_{n}$-oscillating sequence $\left(  u_{h_{n}}\right)  $ in
$L^{2}\left(  \mathbb{T}^{d}\right)  $ such that:
\begin{equation}
\label{e:averagelim}\lim_{n\rightarrow+\infty}\frac{1}{\tau_{h_{n}}}
\int_{\mathbb{\tau}_{h_{n}}a}^{\tau_{h_{n}}b}\int_{\mathbb{T}^{d}}\chi\left(
x\right)  \left\vert S_{h_{n}}^{t}u_{h_{n}}\left(  x\right)  \right\vert
^{2}dxdt=\int_{a}^{b}\int_{\mathbb{T}^{d}}\chi\left(  x\right)  \nu\left(
t,dx\right)  dt,
\end{equation}
for every real numbers $a<b$ and every $\chi\in\mathcal{C}\left(
\mathbb{T}^{d}\right)  $. In other words, we are averaging the densities
$\left\vert S_{h}^{t}u_{h}(x)\right\vert ^{2}$ over time intervals of size
$\tau_{h}$. This averaging, as we shall see, makes the study more tractable.

If case \eqref{e:averagelim} occurs, we shall say that $\nu$ is obtained
through the sequence $\left(  u_{h_{n}}\right)  $. To simplify the notation, when
no confusion can arise, we shall simply write that $h\To0^{+}$ to mean that we
are considering a discrete sequence $h_{n}$ going to $0^{+}$, and we shall
denote by $\left(  u_{h}\right)  $ (instead of $\left(  u_{h_{n}}\right)  $)
the corresponding family of functions.

\begin{remark}
\label{r:boundedtime}When the function $\tau$ is bounded, the
convergence of $\nu_{h}\left(  \tau_{h}t,\cdot\right)  $ to an
accumulation point $\nu\left( t,\cdot\right)  $ is locally uniform
in $t$. According to Egorov's theorem (see, for instance,
\cite{EvansZworski}), $\nu$ can be completely described in terms
of semiclassical defect measures of the corresponding sequence of
initial data $\left(  u_{h}\right)  $, transported by the
classical Hamiltonian flow $\phi_{s}:T^{\ast}\mathbb{T}^{d}\To
T^{\ast }\mathbb{T}^{d}$ generated by $H$, which in this case is
completely integrable~:
\begin{equation}\label{def:phis}
\phi_{s}(x,\xi):=(x+sdH(\xi),\xi).
\end{equation}
As an example, take $\tau_h=1$ and consider the case where the initial data
$u_{h}$ are coherent states~: fix $\rho\in\mathcal{C}_{c}^{\infty}\left(
\mathbb{R}^{d}\right)  $ with $\left\Vert \rho\right\Vert _{L^{2}\left(
\mathbb{R}^{d}\right)  }=1$, fix $\left(  x_{0},\xi_{0}\right)  \in
\mathbb{R}^{d}\times\mathbb{R}^{d}$, and let $u_{h}\left(  x\right)  $ be the
$2\pi\mathbb{Z}^{d}$-periodization of the following coherent state:
\[
\frac{1}{h^{d/4}}\rho\left(  \frac{x-x_{0}}{\sqrt{h}}\right)  e^{i\frac
{\xi_{0}}{h}\cdot x}.
\]
Then $\nu_{h}\left(  t,\cdot\right)  $ converges, for every $t\in\mathbb{R}$,
to:
\[
\delta_{x_{0}+tdH\left(  \xi_{0}\right)  }\left(  x\right)  .
\]
\end{remark}

When the time scale $\tau_{h}$ is unbounded, the $t$-dependence of elements
$\nu\in\mathcal{M}\left(  \tau\right)  $ is not described by such a simple
propagation law. From now on we shall only consider the case where $\tau
_{h}\Lim_{h\To 0} +\infty$.

The problem of describing the elements in $\mathcal{M}\left(  \tau\right)  $
for some time scale $\left(  \tau_{h}\right)  $ is related to several aspects
of the dynamics of the flow $S_{h}^{t}$ such as dispersive effects and unique
continuation. In \cite{AnantharamanMaciaSurv, MaciaDispersion} the reader will
find a description of these issues in the case where the operator $S_{h}
^{t}$ is the semiclassical Schr\"{o}dinger propagator $e^{iht\Delta}$
corresponding to the Laplacian on an arbitrary compact Riemannian manifold.
 In that
setting, the time scale $\tau_{h}=1/h$ appears in a natural way, since it
transforms the semiclassical propagator into the non-scaled flow
$e^{ih\tau_{h}t\Delta}=e^{it\Delta}$. The possible accumulation points of
sequences of probability densities of the form $|e^{it\Delta}u_{h}|^{2}$
depend on the nature of the dynamics of the geodesic flow.
When
the geodesic flow has the Anosov property (a very strong form of chaos, which holds on negatively curved manifolds), the results in \cite{AnRiv} rule
out concentration on sets of small dimensions, by proving lower bounds on the
Kolmogorov-Sinai entropy of semiclassical defect measures.
Even in the apparently simpler case that the geodesic flow is completely
integrable, different type of concentration phenomena may occur, depending on
fine geometrical issues (compare the situation in Zoll manifolds
\cite{MaciaAv} and on flat tori \cite{MaciaTorus, AnantharamanMacia}).

\subsection{Semiclassical defect measures}

Our results are more naturally described in terms of \emph{Wigner
distributions }and \emph{semiclassical measures} (these are the semiclassical
version of the \emph{microlocal defect measures} \cite{GerardMDM91, TartarH},
and have also been called \emph{microlocal lifts} in the recent
literature about quantum unique ergodicity, see for instance the celebrated paper \cite{LindenQUE}). The \emph{Wigner distribution} associated
to $u_{h}$ (at scale $h$) is a distribution on the cotangent bundle $T^{\ast
}\mathbb{T}^{d}$, defined by
\begin{equation}
\int_{T^{\ast}\mathbb{T}^{d}}a(x,\xi)w_{u_{h}}^{h}(dx,d\xi)=\left\langle
u_{h},\Op_{h}(a)u_{h}\right\rangle _{L^{2}(\mathbb{T}^{d})},\qquad
\mbox{ for all }a\in\mathcal{C}_{c}^{\infty}(T^{\ast}\mathbb{T}^{d}),
\label{e:initialwigner}
\end{equation}
where $\Op_{h}(a)$ is the operator on $L^{2}(\mathbb{T}^{d})$ associated to
$a$ by the Weyl quantization. The reader not familiar with these objects can
consult the appendix of this article or the book \cite{EvansZworski}. For the moment, just recall that $w_{u_{h}}^{h}$ extends naturally to
smooth functions $\chi$ on $T^{\ast}\mathbb{T}^{d}=\mathbb{T}^{d}
\times(\R^{d})^*$ that depend only on the first coordinate, and in this case we have
\begin{equation}
\int_{T^{\ast}\mathbb{T}^{d}}\chi(x)w_{u_{h}}^{h}(dx,d\xi)=\int_{\mathbb{T}
^{d}}\chi(x)|u_{h}(x)|^{2}dx. \label{e:proj}
\end{equation}
The main object of our study will be the (time-scaled) Wigner distributions
corresponding to solutions to (\ref{e:eq}):
\[
w_{h}(t,\cdot):=w_{S_{h}^{\tau_{h}t}u_{h}}^{h}
\]
The map $t\longmapsto w_{h}(t,\cdot)$ belongs to $L^{\infty}(\R;\cD^{\prime
}\left(  T^{\ast}\mathbb{T}^{d}\right)  )$, and is uniformly bounded in that
space as $h\To0^{+}$ whenever $\left(  u_{h}\right)  $ is normalised in
$L^{2}\left(  \mathbb{T}^{d}\right)  $. Thus, one can extract subsequences
that converge in the weak-$\ast$ topology on $L^{\infty}(\R;\cD^{\prime
}\left(  T^{\ast}\mathbb{T}^{d}\right)  )$. In other words, after possibly
extracting a subsequence, we have
\[
\int_{\R}\int_{T^*\T^d}\varphi(t)a(x,\xi)w_{h}(t,dx,d\xi)dt\Lim_{h\To0}\int_{\R}\int_{T^*\T^d }
\varphi(t)a(x,\xi)\mu(t,dx,d\xi)dt
\]
for all $\varphi\in L^{1}(\R)$ and $a\in\mathcal{C}_{c}^{\infty}(T^{\ast
}\mathbb{T}^{d})$, and the limit $\mu$ belongs to $L^{\infty}\left(
\mathbb{R};\mathcal{M}_{+}\left(  T^{\ast}\mathbb{T}^{d}\right)  \right)
$ (here $\mathcal{M}_{+}\left(  X\right)  $ denotes the set of positive
Radon measures on a Polish space $X$).

The set of limit points thus obtained, as $\left(  u_{h}\right)  $ varies
among normalised sequences, will be denoted by $\mathcal{\widetilde{M}}\left(
\tau\right)  $. We shall refer to its elements as (time-dependent)
semiclassical measures.

Moreover, if $\left(  u_{h}\right)  $ is $h$-oscillating (see~(\ref{e:hosc})), it follows that
$\mu\in L^{\infty}\left(  \mathbb{R};\mathcal{P}\left(  T^{\ast}\mathbb{T}
^{d}\right)  \right)  $ and identity (\ref{e:proj}) is also verified in the
limit~:
\[
\int_{a}^{b}\int_{\mathbb{T}^{d}}\chi\left(  x\right)  |S_{h}^{\tau_{h}t}
u_{h}\left(  x\right)  |^{2}dxdt\Lim_{h\To0}\int_{a}^{b}\int_{T^{\ast
}\mathbb{T}^{d}}\chi\left(  x\right)  \mu\left(  t,dx,d\xi\right)  dt,
\]
for every $a<b$ and every $\chi\in\mathcal{C}^{\infty}\left(  \mathbb{T}
^{d}\right)  $. Therefore, $\mathcal{M}\left(  \tau\right)  $ coincides with
the set of projections onto $x$ of semiclassical measures in
$\mathcal{\widetilde{M}}\left(  \tau\right)  $ corresponding to $h$-oscillating sequences \cite{GerardMesuresSemi91,GerLeich93}.

It is also shown in the appendix that the elements of $\mathcal{\widetilde{M}
}\left(  \tau\right)  $ are measures that are {\it $H$-invariant}, by which we mean that they are
invariant
under the action of the hamiltonian flow $\phi_s$ defined in~(\ref{def:phis}).

\subsection{Results on the regularity of semiclassical measures}

The main results in this article are aimed at obtaining a precise description
of the elements in $\mathcal{\widetilde{M}}\left(  \tau\right)  $ (and, as a
consequence, of those of $\mathcal{M}\left(  \tau\right)  $). We first present
a regularity result which emphasises the critical character of the time scale
$\tau_{h}=1/h$ in situations in which the Hessian of $H$ is non-degenerate,
definite (positive or negative).

\begin{theorem}
\label{t:main}(1) If $\tau_{h}\ll1/h$ then $\mathcal{M}\left(
\tau\right)  $ contains elements that are singular with respect to
the Lebesgue measure $dtdx$. Actually,
$\mathcal{\widetilde{M}}\left(  \tau\right)  $ contains all
measures invariant by the flow~$\phi_s$ defined in~(\ref{def:phis}).

(2) Suppose $\tau_{h}\sim1/h$ or $\tau_{h}\gg1/h$. Assume that the Hessian
$d^{2}H(\xi)$ is definite for all~$\xi$. Then
\[
{\mathcal{M}}\left(  \tau\right)  \subseteq L^{\infty}\left(  \mathbb{R}
;L^{1}\left(  \mathbb{T}^{d}\right)  \right)  ,
\]
in other words
the elements of $\mathcal{M}\left(  \tau\right)  $ are absolutely
continuous with respect to $dtdx$.
\end{theorem}

The proof of (1) in Theorem~\ref{t:main} relies on the construction of examples, while the proof of~(2) is based on the forthcoming Theorem \ref{t:precise}, which contains a careful analysis of the case $\tau_h=1/h$ (see section~\ref{sec:1/h}). A comparison argument between different time-scales allows to treat the case $\tau_h \gg 1/h$ (see section~\ref{subsec:hierarchy}).

Note also that the construction leading to Theorem~\ref{t:main} (2) also yields observability results~: see section~\ref{sec:obs} below. Finally, we point out in Section~\ref{sec:gen} that Theorem~\ref{t:main} extends to general quantum completely integrable systems.
 An interesting and immediate by-product of Theorem~\ref{t:main} is the following corollary.

\begin{corollary}
Theorem \ref{t:main}(2) applies in particular when the data $(u_{h})$ are
eigenfunctions of $H(hD_{x})$, and shows (assuming the Hessian of $H$ is
definite) that the weak limits of the probability measures $|u_{h}(x)|^{2} dx$
are absolutely continuous.
\end{corollary}

Note that statement (2) of Theorem \ref{t:main} has already been proved in the
case $H(\xi)=|\xi|^{2}$ in~\cite{BourgainQL97} and~\cite{AnantharamanMacia}
with different proofs (the proof in the second reference extends to the
$x$-dependent Hamiltonian $\left\vert \xi\right\vert ^{2}+h^{2}V\left(
x\right)  $). However, the extension to more general $H$ of the method in
\cite{AnantharamanMacia} is not straightforward, even in the case where
$H(\xi)=\xi\cdot A\xi$, where $A$ is a symmetric linear map~: $(\IR^{d}
)^{*}\To \IR^{d}$ (i.e. the Hessian of $H$ is constant), the difficulty
arising when $A$ has irrational coefficients.

\medskip

Let us now comment on the assumptions of the theorem. We first want to emphasize that the conclusion of Theorem \ref{t:main}(2) may fail if the condition on the
Hessian of $H$ is not satisfied.

\medskip

\begin{example}
Fix $\omega\in\mathbb{R}^{d}$ and take $H\left(  \xi\right)
=\xi\cdot\omega$ and $\mathbf{V}_h(t)=0$. Let $\mu_{0}$ be an
accumulation point in $\cD^{\prime}\left(  T^{\ast
}\mathbb{T}^{d}\right)  $ of the Wigner distributions $\left(
w_{u_{h}} ^{h}\right)  $ defined in \eqref{e:initialwigner},
associated to the initial data $(u_{h})$. Let
$\mu\in\mathcal{\widetilde{M}}\left(  \tau\right) $ be the limit
of $w_{S_{h}^{\tau_{h}t}u_{h}}^{h}$ in $L^{\infty}(\R;\cD^{\prime
}\left(  T^{\ast}\mathbb{T}^{d}\right) )$. Then an application of
Egorov's theorem (actually, a particularly simple adaptation of
the proof of Theorem 4 in \cite{MaciaAv}) gives the relation,
valid for any time scale $\left( \tau_{h}\right)  $~:
\[
\int_{T^{\ast}\mathbb{T}^{d}}a\left(  x,\xi\right)  \mu\left(  t,dx,d\xi
\right)  =\int_{T^{\ast}\mathbb{T}^{d}}\left\langle a\right\rangle \left(
x,\xi\right)  \mu_{0}\left(  dx,d\xi\right)  ,
\]
for any $a\in\mathcal{C}_{c}^{\infty}\left(  T^{\ast}\mathbb{T}^{d}\right)  $
and a.e. $t\in\mathbb{R}$. Here $\left\langle a\right\rangle $ stands for
the average of $a$ along the Hamiltonian flow $\phi_{s}$, that is in our case
\[
\left\langle a\right\rangle \left(  x,\xi\right)  =\lim_{T\rightarrow\infty
}\frac{1}{T}\int_{0}^{T}a\left(  x+s\omega,\xi\right)  ds.
\]
Hence, as soon as $\omega$ is resonant (in the sense of \S \ref{s:decompo})
and $\mu_{0}=\delta_{x_{0}}\otimes\delta_{\xi_{0}}$ for some $\left(
x_{0},\xi_{0}\right)  \in T^{\ast}\mathbb{T}^{d}$, the measure $\mu$ will be
singular with respect to $dtdx$.
\end{example}

\noindent It is also easy to provide counter-examples where the Hessian of $H$ is
non-degenerate, but not definite.

\begin{example} On the two-dimensional torus $\IT^{2}$,
consider $H(\xi)=\xi_{1}^{2}-\xi_{2}^{2}$, where $\xi=(\xi
_{1},\xi_{2})$. Take for $(u_{h}(x_{1},x_{2}))$ the periodization of
\[
\frac{1}{\left(  2\pi h\right)  ^{1/2}}\rho\left(  \frac{x_{1}-x_{2}}
{h}\right)
\]
where $\rho\in\mathcal{C}_{c}^{\infty}\left(  \mathbb{R}\right)  $ satisfies
$\left\Vert \rho\right\Vert _{L^{2}\left(  \mathbb{R}\right)  }=1$. Then the
functions $u_{h}$ are eigenfunctions of $H(hD_{x})$ for the eigenvalue $0$ and the measures
$|u_{h}(x_{1},x_{2})|^{2}dx_{1}\,dx_{2}$ obviously concentrate on the diagonal
$\{x_{1}=x_{2}\}$.
\end{example}

\noindent Note however that in this example the system is {\em
isoenergetically degenerate} at $\xi=0$. Recall the definition of
{\em isoenergetic non-degeneracy}~: the Hamiltonian $H$ is
isoenergetically non-degenerate at $\xi$ if for all $\eta\in
(\R^d)^*$, and $\lambda\in\R$,
$$
 dH(\xi)\cdot \eta=0 \mbox{ and } d^2H(\xi)\cdot\eta=\lambda
dH(\xi) \Longrightarrow (\eta, \lambda)=(0, 0).
$$
Definiteness of the Hessian implies isoenergetic non-degeneracy at
all $\xi$ such that $dH(\xi)\not= 0$. In view of the previous
example, one may wonder whether isoenergetic non-degeneracy is a
sufficient assumption for our results. In Section \ref{s:suff} we
give a sufficient set of assumptions for our results which is
weaker than definiteness, but is not implied by isoenergetic
non-degeneracy except in dimension $d=2$. As a conclusion,
isoenergetic non-degeneracy is sufficient for all our results in
dimension $d=2$, but not in dimensions $d\geq 3$, as is finally
shown by the following counter-example~:

\begin{example}
Take $d=3$. On $(\R^3)^*$ consider $H(\xi)=\xi_{1}^{2}+\xi_{2}^2-
\xi_3^3$, and let $u_{h}(x_{1},x_{2}, x_3)$ be the periodization
of
\[
\frac{1}{\left(  2\pi \eps\right)  ^{1/2}}\rho\left(
\frac{x_{2}+x_{3}} {\eps}\right)e^{i\frac{\alpha x_1
+x_2+x_3}{h}},
\]
where $\rho\in\mathcal{C}_{c}^{\infty}\left(  \mathbb{R}\right)  $
, $\left\Vert \rho\right\Vert _{L^{2}\left(  \mathbb{R}\right)
}=1$, and $\eps=\eps(h)$ tends to $0$ with $\eps(h)\gg h$. Note
that $u_h$ is an eigenfunction of $H(h D_x)$. The Wigner measures
of $(u_h)$ concentrate on the set $\{\xi_1=\alpha,
\xi_2=\xi_3=1\}$ where the system is isoenergetically
non-degenerate if $\alpha\not= 0$. Its projection on $\T^3$ is
supported on the hyperplane $\{x_2+x_3=0\}$.
\end{example}

\noindent In Section \ref{s:halpha} we present an example communicated to us by J.
Wunsch showing that absolute continuity of the elements of $\mathcal{M}\left(
1/h\right)  $ may fail in the presence of a subprincipal symbol of order
$h^{\beta}$ with $\beta\in\left(  0,2\right)  $ even in the case $H\left(
\xi\right)  =\left\vert \xi\right\vert ^{2}$. We also show in Section~\ref{s:halpha} that absolute continuity may fail for the elements of
$\mathcal{M}\left(  1/h\right)  $ when $H\left(  \xi\right)  =\left\vert
\xi\right\vert ^{2k}$, $k\in\mathbb{N}$ and $k>1$; a situation where the
Hessian is degenerate at $\xi=0$.

\medskip
\medskip

We point out that Theorem \ref{t:main}(2) admits a microlocal refinement, which allows us to
deal with more general Hamiltonians $H$ whose Hessian is not necessarily
definite at every $\xi\in\mathbb{R}^{d}$. Given $\mu\in\mathcal{\widetilde{M}
}\left(  \tau\right)  $ we shall denote by $\bar{\mu}$ the image of $\mu$
under the map $\pi_{2}:(x,\xi)\longmapsto\xi$. For ${\mathbf V}_h(t)=\Op_h(V(t, x, \xi))$ with $V\in C^\infty(\R\times T^* \T^d)$, it is shown in the appendix
that $\bar{\mu}$ does not depend on $t$ if $\tau_h\ll h^{-2}$: in this case we have $\bar{\mu
}=\left(  \pi_{2}\right)  _{\ast}\mu_{0}$, where the measure $\mu_{0}$ is an
accumulation point in $\mathcal{D}^{\prime}\left(  T^{\ast}\mathbb{T}
^{d}\right)  $ of the sequence $\left(  w_{u_{h}}^{h}\right)  $. For simplicity we restrict our attention to that case in the following theorem~:

\begin{theorem}
\label{t:lagrangian} Assume that $\mathbf{V}_h(t)=\Op_h(V(t,
\cdot))$ with $V\in C^\infty(\R\times T^* \T^d)$ bounded.

Let $\mu\in\mathcal{\widetilde{M}}\left(  1/h\right)  $
and denote by $\mu_{\xi}(t,\cdot)$ the disintegration of $\mu(t,\cdot)$ with
respect to the variable $\xi$, \emph{i.e. }for every $\theta\in L^{1}\left(
\mathbb{R}\right)  $ and every bounded measurable function~$f$:
\[
\int_{\mathbb{R}}\theta(t)\int_{\mathbb{T}^{d}\times\mathbb{R}^{d}}f(x,\xi
)\mu(t,dx,d\xi)dt=\int_{\mathbb{R}}\theta\left(  t\right)  \int_{\mathbb{R}
^{d}}\left(  \int_{\mathbb{T}^{d}}f(x,\xi)\mu_{\xi}(t,dx)\right)  \bar{\mu
}(d\xi)dt.
\]
Then for $\bar{\mu}$-almost every $\xi$ where $d^{2}H(\xi)$ is definite, the
measure $\mu_{\xi}(t,\cdot)$ is absolutely continuous.
\end{theorem}

\noindent Let us introduce the closed set
\[
C_{H}:=\left\{  \xi\in\R^{d}:\, d^{2}H(\xi)\text{ is not definite}\right\}  .
\]
The following consequence of Theorem \ref{t:lagrangian} provides a refinement
on Theorem \ref{t:main}(2), in which the global hypothesis on the Hessian of
$H$ is replaced by a hypothesis on the sequence of initial data.

\begin{corollary}
\label{c:mainsharp}Suppose $\nu\in\mathcal{M}\left(  1/h\right)  $ is obtained
through an $h$-oscillating sequence $\left(  u_{h}\right)  $ having a
semiclassical measure $\mu_{0}$ such that $\mu_{0}\left(  \T^{d}\times
C_{H}\right)  =0$. Then $\nu$ is absolutely continuous with respect to $dtdx$.
\end{corollary}

\subsection{Second-microlocal structure of the semiclassical measures}\label{sec:1/h}

Theorem \ref{t:lagrangian} is a consequence of a more detailed result on the
structure of the elements of $\mathcal{\widetilde{M}}\left(  1/h\right)  $ on which we focus in this paragraph. We
follow here the strategy of~\cite{AnantharamanMacia} that we adapt to a
general Hamiltonian $H(\xi)$. The proof relies on a decomposition of the
measure associated with the primitive submodules of $(\Z^{d})^*$. Before stating
it, we must introduce some notation.

Recall that $(\R^d)^*$ is the dual of $\R^d$. Later in the paper, we will sometimes identify both by working in the canonical basis of $\R^d$. We will denote by $(\Z^d)^*$ the lattice in $(\R^d)^*$ defined by $(\Z^d)^*=\{\xi\in
(\R^d)^*, \xi.n\in\Z, \:\forall n\in \Z^d\}$.
We call a submodule $\Lambda\subset(\Z^{d})^*$ primitive if $\left\langle
\Lambda\right\rangle \cap(\mathbb{Z}^{d})^*=\Lambda$ (here $\left\langle
\Lambda\right\rangle $ denotes the linear subspace of $(\mathbb{R}^{d})^*$ spanned
by $\Lambda$). Given such a submodule we define:
\begin{equation}
\label{def:ILambda}I_{\Lambda}:=\left\{  \xi\in(\mathbb{R}^{d})^*: dH\left(
\xi\right)\cdot k  =0,\;\forall k\in\Lambda\right\}  .
\end{equation}
We note that $I_{\Lambda}\setminus C_{H}$ is a smooth submanifold.

We define also $L^{p}\left(  \mathbb{T}^{d},\Lambda\right)  $ for $p\in\left[
1,\infty\right]  $ to be the subspace of $L^{p}\left(  \mathbb{T}^{d}\right)
$ consisting of the functions~$u $ such that $\widehat{u}\left(  k\right)  =0$
if $k\in(\IZ^{d})^*\setminus\Lambda$ (here $\widehat{u}\left(  k\right)  $ stand for
the Fourier coefficients of~$u$). Given $a\in\mathcal{C}_{c}^{\infty}\left(
T^{\ast}\mathbb{T}^{d}\right)  $ and $\xi\in\mathbb{R}^{d}$, denote by
$\left\langle a\right\rangle _{\Lambda}\left(  \cdot,\xi\right)  $ the
orthogonal projection of $a\left(  \cdot,\xi\right)  $ on $L^{2}\left(
\mathbb{T}^{d},\Lambda\right)  $:
\begin{equation}
\label{def:aLambda}\langle a\rangle_{\Lambda}(x, \xi)= \sum_{k\in\Lambda} \widehat
a_{k}(\xi) {\frac{\mathrm{e}^{ikx}}{(2\pi)^{d}}}
\end{equation}
Note that if $a$ only has frequencies in $\Lambda$, then $\langle a\rangle_\Lambda=a$.

For $\omega $ in the torus $\la\Lambda %
\ra/\Lambda $, we denote by $L_{\omega }^{2}(\R^{d},\Lambda )$ the
subspace of $L_{\text{loc}}^{2}(\R^{d})\cap \mathcal{S}^{\prime
}(\mathbb{R}^{d})$ formed by the functions whose Fourier transform
is supported in $\Lambda -\omega $.  Each $L_{\omega }^{2}(\R^{d},\Lambda )$ has a natural Hilbert space structure.

We denote by $m_{\left\langle a\right\rangle _{\Lambda}}\left(
\xi\right)  $ the operator acting on each $L^{2}_\omega\left(
\mathbb{R}^{d} ,\Lambda\right)  $ by multiplication by
$\left\langle a\right\rangle _{\Lambda}\left(  \cdot,\xi\right)
$.

\begin{theorem}
\label{t:precise}(1) Let $\mu\in\mathcal{\widetilde{M}}\left(  1/h\right)  $. For
every primitive submodule $\Lambda\subset(\Z^{d})^*$ there exists a positive
measure $\mu_{\Lambda}^{{\rm final}}\in L^{\infty}\left(  \mathbb{R};\mathcal{M}_{+}\left(
T^{\ast}\mathbb{T}^{d}\right)  \right)  $ supported on $\mathbb{T}^{d}\times
I_{\Lambda}$ and invariant by the Hamiltonian flow $\phi_{s}$ such that~: for
every $a\in\mathcal{C}_{c}^{\infty}\left(  T^{\ast}\mathbb{T}^{d}\right)  $
that vanishes on $\T^{d}\times C_{H}$ and every $\theta\in L^{1}\left(
\mathbb{R}\right)  $:
\begin{equation}
\int_{\mathbb{R}}\theta\left(  t\right)  \int_{T^{\ast}\mathbb{T}^{d}}a\left(
x,\xi\right)  \mu\left(  t,dx,d\xi\right)  dt=\sum_{\Lambda\subseteq
\mathbb{Z}^{d}}\int_{\mathbb{R}}\theta\left(  t\right)  \int_{\mathbb{T}
^{d}\times I_{\Lambda}}a\left(  x,\xi\right)  \mu_{\Lambda}^{{\rm final}}\left(
t,dx,d\xi\right)  dt, \label{e:mh1}
\end{equation}
the sum being taken over all primitive submodules of $(\mathbb{Z}^{d})^*$.

\medskip
 In addition, there exists a measure $\bar{\mu}_{\Lambda}(t)$ on $(\la\Lambda\ra/\Lambda)\times I_{\Lambda}$ and a measurable family $\left\{ N_{\Lambda}(t, \omega, \xi) \right\}_{t\in\R, \omega\in\la\Lambda\ra/\Lambda,  \xi \in I_{\Lambda}}$ of non-negative, symmetric, trace-class operators acting on
$L^{2}_\omega\left(  \mathbb{R}^{d},\Lambda\right)  $, such that
the following holds:
\begin{equation}
\int_{\mathbb{T}^{d}\times I_{\Lambda}}a\left(  x,\xi\right)  \mu_{\Lambda
}^{{\rm final}}\left(  t,dx,d\xi\right)  =\int_{(\la\Lambda\ra/\Lambda)\times I_{\Lambda}}\operatorname{Tr}
\left(
m_{\left\langle a\right\rangle_{\Lambda}}\left(  \xi\right)
 N_{\Lambda}(t, \omega, \xi)
 \right)
\bar{\mu}_{\Lambda}(t, d\omega, d\xi)  .
\label{e:mh2}
\end{equation}

(2) If $\mathbf{V}_h(t)=\Op_h(V(t, \cdot))$ with $V\in
\cC^\infty(\R\times T^* \T^d)$, then $\bar{\mu}_{\Lambda}$ does not
depend on $t$, and $N_{\Lambda}(t,\omega, \xi)$ depends
continuously on $t$, and solves the Heisenberg equation labelled
below as (\emph{Heis}$_{\Lambda, \omega, \xi}$).

When the Hessian of $H$ is definite, formula (\ref{e:mh1}) holds
for every $a\in\mathcal{C}_{c}^{\infty}\left( T^{\ast}\mathbb{T}
^{d}\right)  $ and therefore completely describes $\mu$.
\end{theorem}

\begin{remark}
The arguments in Section 6.1 of \cite{AnantharamanMacia} show that Theorem
\ref{t:lagrangian} is a consequence of Theorem~\ref{t:precise}. Therefore, in
this article only the proof of Theorem \ref{t:precise} will be presented.
\end{remark}

Theorem \ref{t:precise} has been proved for $H\left(  \xi\right)  =\left\vert
\xi\right\vert ^{2}$ in \cite{MaciaTorus} for $d=2$ and in
\cite{AnantharamanMacia} for the $x$-dependent Hamiltonian $\left\vert
\xi\right\vert ^{2}+h^{2}V\left(  x\right)  $ in arbitrary dimension  (in these papers the parameter~$\omega$ does not appear, and all measures have Dirac masses at $\omega=0$).

The measures $\mu_{\Lambda}^{{\rm final}}$ in equation \eqref{e:mh1} are obtained as the final step of an interative procedure that involves a process of successive microlocalizations along nested
sequences of submanifolds in frequency space. 

Theorem \ref{t:precise}(2) allows to describe the dependence of
$\mu$ on the parameter $t$. This is a subtle issue since, as was noticed in \cite{MaciaAv,
MaciaTorus}, the semiclassical measures of the sequence of initial
data $\left(  u_{h}\right) $ do not determine uniquely the time
dependent semiclassical measure~$\mu$. Thus, when $\mathbf{V}_h(t)=\Op_h(V(t, \cdot))$ with
$V\in \cC^\infty(\R\times T^* \T^d)$, the measure $\mu_{\Lambda}^{{\rm final}}\left(
t,dx,d\xi\right)  $  is fully
determined by the measures $\bar{\mu}_{\Lambda}$ and the family of operators
$N_{\Lambda}(0,\omega, \xi)$, which are objects determined by
the initial data $(u_{h})$. The $N_{\Lambda}(t,\omega, \xi)$ are obtained
from $N_{\Lambda}(0,\omega, \xi)$ by propagation along a Heisenberg equation (Heis$_{\Lambda, \omega, \xi}$), written in Theorem \ref{prop:opvame}, which is the evolution equation of operators that
comes from the following Schr\"odinger equation in $L^2_\omega(\R^d,\Lambda)$~:
\begin{equation}\tag{{\rm S}$_{\Lambda, \omega, \xi}$}
i\partial _{t}v=\left( \frac{1}{2}d^{2}H(\xi )D_{y}\cdot D_{y}+\la
V(\cdot ,\xi )\ra_{\Lambda }\right) v.  \label{e:schrodLam}
\end{equation}%
This process gives
an explicit construction of $\mu$ in terms of the initial data.
Full details on the structure of these objects are provided in
Sections \ref{s:finitedistance} and \ref{s:successive}.

Theorem \ref{t:precise} is stated for the time scale $\tau_{h}=1/h$; if $\tau _{h}\ll1/h$, the elements of
$\mathcal{\widetilde{M}}\left(  \tau\right)  $ can also be described by
a similar result (see Section \ref{s:rec}) involving expression
(\ref{e:mh1}). However, in that case, the propagation law  involves classical transport
rather than propagation along a Schr\"{o}dinger flow, and as a
result Theorem \ref{t:main}(2) does not hold for $\tau_{h}\ll1/h$.

Second microlocalisation has been used in the 80's for studying
propagation of singularities (see \cite{B,BLer,De,Leb}). The
two-microlocal construction performed here is in the spirit of
that done in \cite{NierScat, Fermanian2micro, FermanianShocks} in
Euclidean space in the context of semi-classical measures. We also
refer the reader to the articles \cite{VasyWunsch09, VasyWunsch11,
Wunsch10} for related work regarding the study of the wave-front
set of solutions to semiclassical integrable systems.

When the Hessian of $H$ is constant Theorem \ref{t:precise} gives a complement
to the results announced in \cite{AnantharamanMacia} (where the argument was
only valid when the Hessian has rational coefficients).

\subsection{Hierarchy of time scales\label{subsec:hierarchy}}

In this section, we discuss the dependence of
the set $\mathcal{M}\left(  \tau\right)  $ on the time scale $\tau$. The following proposition allows to derive Theorem~\ref{t:main}(2) for $\tau_h\gg 1/h$ from the result about $\tau_h=1/h$. Denote by $\mathcal{M}_{\text{av}}\left(  \tau\right)  $ the subset of
$\mathcal{P}\left(  \mathbb{T}^{d}\right)  $ consisting of measures of the
form:
\[
\int_{0}^{1}\nu\left(  t,\cdot\right)  dt,\quad\text{where }\nu\in
\operatorname{Conv}\mathcal{M}\left(  \tau\right)  \text{.}
\]
where $\operatorname{Conv}X$ stands for the convex hull of a set
$X\subset L^{\infty}\left(  \mathbb{R};\mathcal{P}\left(
\mathbb{T}^{d}\right)  \right)  $ with respect to the weak-$\ast$
topology.
We have the following result.
\begin{proposition}
\label{prop:hierarchy1}
Suppose $\left(  \tau_{h}\right)  $ and $\left(  \tau'_{h}\right)   $ are time
scales tending to infinity and such that $\tau'_{h}\ll\tau_{h}$. Then:
\[
\mathcal{M}\left(  \tau\right)
\subseteq L^{\infty}\left(  \mathbb{R};\mathcal{M}_{\mathrm{av}}\left(
\tau'\right)  \right)  .
\]
\end{proposition}

\medskip

It is also important to clarify the link between the time-dependent Wigner distributions and
those associated with eigenfunctions. Eigenfunctions are the most commonly
studied objects in the field of quantum chaos, however, we shall see that they do
not necessarily give full information about the time-dependent Wigner distributions.
For the sake of simplicity, we state the results that follow in
the case $\mathbf{V}_h(t)=0$, although they easily generalise to the
case in which $\mathbf{V}_h(t)$ does not depend on $t$. Start noting
that the spectrum of $H\left( hD_{x}\right)$ coincides with
$H\left( h\mathbb{Z}^{d}\right) $; given $E_{h}\in
\operatorname{sp}(H\left( hD_{x}\right))$ the corresponding
normalised eigenfunctions are of the form:
\begin{equation}
u_{h}\left(  x\right)  =\sum_{H\left(  hk\right)  =E_{h}}c_{k}^{h}e^{ik\cdot
x},\quad\text{with}\quad\sum_{k\in\mathbb{Z}^{d}}\left\vert c_{k}
^{h}\right\vert ^{2}=\frac{1}{\left(  2\pi\right)  ^{d}}. \label{e:neigf}
\end{equation}
In addition, one has:
\[
\nu_{h}\left(  \tau_{h}t,\cdot\right)  =\left\vert S_{h}^{\tau_{h}t}
u_{h}\right\vert ^{2}=\left\vert u_{h}\right\vert ^{2},
\]
independently of $\left(  \tau_{h}\right)  $ and $t$. Let us denote by
$\mathcal{M}\left(  \infty\right)  $ the set of accumulation points in
$\mathcal{P}\left(  \mathbb{T}^{d}\right)  $ of sequences $\left\vert
u_{h}\right\vert ^{2}$ where $\left(  u_{h}\right)  $ varies among all
possible $h$-oscillating sequences of normalised eigenfunctions (\ref{e:neigf}), we have
$$
\mathcal{M}\left(  \infty\right)  \subseteq
\mathcal{M}\left(  \tau\right).$$
As a consequence of  Theorem~\ref{t:main}, we obtain the following result.

\begin{corollary}
All eigenfunction
limits $\mathcal{M}\left(  \infty\right)  $ are absolutely continuous under
the definiteness assumption on the Hessian of $H$.
\end{corollary}

A time scale of special importance is the one related to the minimal spacing of
eigenvalues~: define
\begin{equation}
\tau_{h}^{H}:=h\sup\left\{  \left\vert E_{h}^{1}-E_{h}^{2}\right\vert
^{-1}\;:\;E_{h}^{1}\neq E_{h}^{2},\;E_{h}^{1},E_{h}^{2}\in H\left(
h\mathbb{Z}^{d}\right)  \right\}  . \label{e:tauh}
\end{equation}
It is possible to have $\tau_{h}^{H}=\infty$: for instance, if $H\left(
\xi\right)  =\left\vert \xi\right\vert ^{\alpha}$ with $0<\alpha<1$ or
$H\left(  \xi\right)  =\xi\cdot A\xi$ with $A$ a real symmetric matrix that is
not proportional to a matrix with rational entries (this is the
content of the Oppenheim conjecture, settled by Margulis \cite{MargulisDani,
MargulisOpp}). In some other situations, such as $H\left(  \xi\right)
=\left\vert \xi\right\vert ^{\alpha}$ with $\alpha>1$, (\ref{e:tauh}) is
finite~:\ $\tau_{h}^{H}=h^{1-\alpha}$.

\begin{proposition}
\label{p:conv} If $\tau_{h}\gg\tau_{h}^{H}$ one has:
\[
\mathcal{M}\left(  \tau\right)  =\operatorname{Conv}\mathcal{M}
\left(  \infty\right)  .
\]

\end{proposition}

This result is a consequence of the more general results presented in Section
\ref{s:hierarchy}.

Note that Proposition \ref{p:conv} allows to complete the description of
$\mathcal{M}\left(  \tau\right)  $ in the case $H\left(  \xi\right)
=\left\vert \xi\right\vert ^{2}$ as the time scale varies.

\begin{remark}
Suppose $H\left(  \xi\right)  =\left\vert \xi\right\vert ^{2}$, or more
generally, that $\tau_{h}^{H}\sim1/h$ and the Hessian of $H$ is definite.
Then:
\[
\begin{array}
[c]{ll}
\text{if }\tau_{h}\ll1/h, & \exists\nu\in\mathcal{M}\left(  \tau\right)
\text{ such that }\nu\perp dtdx;\smallskip\\
\text{if }\tau_{h}\sim1/h, & \mathcal{M}\left(  \tau\right)  \subseteq
L^{\infty}\left(  \mathbb{R};L^{1}\left(  \mathbb{T}^{d}\right)  \right)
;\smallskip\\
\text{if }\tau_{h}\gg1/h & \mathcal{M}\left(  \tau\right) =
\operatorname{Conv}\mathcal{M}\left(  \infty\right)  .
\end{array}
\]

\end{remark}

Finally, we point out that in this case the regularity of semiclassical measures
can be precised. The elements in $\mathcal{M}\left( \infty\right)$
are trigonometric polynomials when $d=2$, as shown in \cite{JakobsonTori97}; and in general they are more regular than
merely absolutely continuous, see \cite{Aissiou, JakobsonTori97,
JakNadToth01}. The same phenomenon occurs with those elements in
$\mathcal{M}\left( 1/h\right)$ that are obtained through sequences
whose corresponding semiclassical measures do not charge $\{ \xi=0
\}$, see \cite{AJM11}.

\subsection{Application to semiclassical and non-semiclassical observability estimates}

As was already shown in \cite{AnantharamanMacia} for the case
$H\left( \xi \right) =\left\vert \xi \right\vert ^{2}$ the
characterization of the structure of the elements in
$\mathcal{M}\left( 1/h\right) $ implies quantitative, unique
continuation-type estimates for the solutions of the
Schr\"{o}dinger equation (\ref{e:eq}) known as observability
inequalities. This is the case again in this setting; here we
shall prove the following result.

\begin{theorem}
\label{t:semiclassicalObs}Let $U\subset \mathbb{T}^{d}$ open and nonempty, $%
T>0$ and $\chi \in \cC_{c}^{\infty }(\mathbb{R}^{d})$ such that
$\supp\chi \cap C_{H}=\emptyset $. Assume that
$\mathbf{V}_h(t)=\Op_h(V(t, \cdot))$ with $V\in \cC^\infty(\R\times
T^* \T^d)$ bounded. Then the following are equivalent:\smallskip

\noindent i) \emph{Semiclassical observability estimate}. There exists $C=C(U,T,\chi
)>0$ and $h_{0}>0$ such that:%
\begin{equation}
\left\Vert \chi \left( hD_{x}\right) u\right\Vert _{L^{2}(\mathbb{T}%
^{d})}^{2}\leq C\int_{0}^{T}\int_{U}\left\vert S_{h}^{t/h}\chi \left(
hD_{x}\right) u\left( x\right) \right\vert ^{2}dxdt,  \label{e:scobs}
\end{equation}%
for every $u\in L^{2}(\mathbb{T}^{d})$ and $h\in \left( 0,h_{0}\right] $.

\noindent ii) \emph{Unique continuation in }$\mathcal{\widetilde{M}}\left( 1/h\right) $%
. For every $\mu \in \mathcal{\widetilde{M}}\left( 1/h\right) $ with $%
\overline{\mu }(\supp\chi )\neq 0$ and $\overline{\mu }(C_{H})=0$ (recall that $\overline{\mu }$ is the image of $\mu$ under the projection $\pi_2$) one has:%
\begin{equation*}
\int_{0}^{T}\mu \left( t,U\times \supp\chi \right) dt\neq 0.
\end{equation*}

Besides, any of i) or ii) is implied by the following statement.

\medskip

\noindent iii) \emph{Unique continuation for the family of Schr\"{o}dinger equations (S$_{\Lambda, \omega, \xi}$)}. For
every $\Lambda \subset \mathbb{Z}^{d}$, every $\xi \in \supp\chi $ with $%
\Lambda \subseteq dH(\xi)^\perp$ and every $\omega\in \la\Lambda \ra/\Lambda $, one has the following unique continuation
property: if $v\in \cC\left( \mathbb{R};L_{\omega }^{2}(\mathbb{R}^{d},\Lambda
)\right) $ solves the Schr\"odinger equation
(S$_{\Lambda, \omega, \xi}$) and $v|_{\left( 0,T\right) \times U}=0$ then $v=0$.
\end{theorem}

This result will be proved as a consequence of the structure Theorem \ref%
{t:precise}.

\begin{remark}
\label{r:ucp}The unique continuation property for (\ref{e:schrodLam}) stated
in Theorem \ref{t:semiclassicalObs}, iii) is known to hold in any of the
following two cases:

i) $V(\cdot,\xi)$ is analytic in $\left( t,x\right) $ for every
$\xi$. This is a consequence of Holmgren's uniqueness theorem (see
\cite{Tataru})

ii) $V(\cdot,\xi)$ is smooth (or even continuous outside of a set
of null Lebesgue measure) for every $\xi$ and does not depend on
$t$ (see Theorem \ref{t:qobs} below).
\end{remark}

\begin{corollary}\label{c:obs}
Let $U$, $T$, $\chi $, and $\mathbf{V}_h(t)$ be as in Theorem
\ref{t:semiclassicalObs}, and suppose that $V$ satisfies any of
the two conditions in Remark \ref{r:ucp}. Then the semiclassical
observability estimate (\ref{e:scobs}) holds.
\end{corollary}

The nature of the observability estimate (\ref{e:scobs}) is better
appreciated when $H$ is itself quadratic. Suppose:
\begin{equation*}
H_{A,\theta }\left( \xi \right) =\frac{1}{2}A\left( \xi +\theta \right)
\cdot \left( \xi +\theta \right) ,
\end{equation*}%
where $\theta \in \mathbb{R}^{d}$, $A$ is a definite real matrix, and denote
by $\overline{S}^{t}$ the (non-semiclassical) propagator, starting at $t=0$,
associated to $H_{A,\theta }\left( D_{x}\right) +V\left( t,\cdot \right) $.
Clearly, the propagator $S_{h}^{t/h}$ associated to $H_{A,h\theta }\left(
hD_{x}\right) +h^{2}V\left( t,\cdot \right) $ coincides with $\overline{S}%
^{t}$ in this case.

\begin{corollary}
\label{c:quadobs}Let $U\subset \mathbb{T}^{d}$ be a nonempty open set, and $%
T>0$. Let $\mathbf{V}_h(t)=\Op_h(V(t, \cdot))$ with $V\in
\cC^\infty(\R\times T^* \T^d)$ bounded. Suppose that the following
unique continuation result holds:

\noindent For every $\Lambda \subset \mathbb{Z}^{d}$ and every
$\xi \in \mathbb{R}^{d}$ with $\Lambda \subseteq dH(\xi )^{\bot}$, if $v\in \cC\left( \mathbb{R};L_{\omega }^{2}(\mathbb{R}%
^{d},\Lambda )\right) $, $\omega \in \la\Lambda \ra/\Lambda $, solves:%
\begin{equation}
i\partial _{t}v=\left( H_{A,\theta }\left( D_{y}\right) +\la V(\cdot ,\xi )%
\ra_{\Lambda }\right) v  \label{e:schrodL}
\end{equation}%
and $v|_{\left( 0,T\right) \times U}=0$ then $v=0$.

Then there exist $C>0$ such that for every $u\in L^{2}(\mathbb{T}^{d})$ one
has:%
\begin{equation}
\left\Vert u\right\Vert _{L^{2}(\mathbb{T}^{d})}^{2}\leq
C\int_{0}^{T}\int_{U}\left\vert \overline{S}^{t}u\left( x\right) \right\vert
^{2}dxdt.  \label{e:nscobs}
\end{equation}
\end{corollary}

Note that an estimate such as (\ref{e:nscobs}) implies a unique
continuation result for solutions to (\ref{e:eq}):
$\overline{S}^{t}u|_{U}=0$ for $t\in \left( 0,T\right) $
$\Longrightarrow $ $u=0$. Corollary \ref{c:quadobs} shows in
particular that this (weaker) unique continuation property for
family of quadratic Hamiltonians in equations~(\ref{e:schrodL})
actually implies the stronger estimate (\ref{e:nscobs}). We also
want to stress the fact that Corollary \ref{c:quadobs} establishes
the unique continuation property for perturbations of
pseudodifferential type from the analogous property for
perturbations that are merely multiplication by a potential.

It should be also mentioned that the proof of Theorem 4 in \cite%
{AnantharamanMacia} can be adapted almost word by word to prove estimate (%
\ref{e:nscobs}) in the case when $V$ does not depend on $t$, \emph{%
without relying} in any \emph{a priori }unique continuation result except those for eigenfunctions. In fact,
the function~$V$ can be supposed less regular than smooth: it suffices that it is
continuous outside of a set of null Lebesgue measure.

\begin{theorem}
\label{t:qobs}Suppose $V$ only depends on $x$; let $U\subset \mathbb{T}^{d}$
a nonempty open set, and let~$T>0$. Then (\ref{e:nscobs}) holds; in particular,
any solution $\overline{S}^{t}u$ that vanishes identically on $\left(
0,T\right) \times U$ must vanish everywhere.
\end{theorem}

In the proof of Theorem \ref{t:qobs}, unique continuation for solutions to
the time dependent Schr\"{o}dinger equation is replaced by a unique
continuation result for eigenfunctions of $H_{A,\theta }\left( D_{x}\right)
+V\left( t,\cdot \right) $. This allows to reduce the proof of (\ref%
{e:nscobs}) to that of a semiclassical observability estimate (\ref{e:scobs}%
) with a cut-off $\chi $ vanishing close to $\xi =0$. At this point, the
validity of (\ref{e:nscobs}) is reduced to the validity of the corresponding
estimate on $\mathbb{T}^{d-1}$. The proof of (\ref{e:nscobs}) is completed
by applying an argument of induction on the dimension $d$. We refer the
reader to the proof of \cite{AnantharamanMacia}, Theorem 4 for additional
details (see also \cite{MaciaDispersion} for a proof in a simpler case in $%
d=2$).

Let us finally mention that Theorem \ref{t:qobs} was first proved
in the case $H\left( \xi \right) =\left\vert \xi \right\vert ^{2}$
in \cite{JaffardPlaques} for $V=0$, in \cite{BurqZworski11} for
$d=2$ and in \cite{AnantharamanMacia} for general $d$, the three
results having rather different proofs. We also refer the reader
to \cite{AnantharamanMaciaSurv, BurqZworskiJAMS, LaurentSurv,
MaciaDispersion} for additional results and
references concerning observability inequalities in the context of Schr\"{o}%
dinger-type equations.

\subsection{Generalization to  quantum completely integrable systems}\label{sec:gen}
Our results may be transferred to more general completely integrable systems as follows.
 Let $(M, dx)$ be a compact manifold of dimension $d$, equipped with a density $dx$. Assume we have a family $(\hat A_1, \ldots, \hat A_d)$ of $d$ commuting self-adjoint $h$-pseudodifferential operators of order $0$ in $h$. By this, we mean an operator $a(x, hD_x)$ where $a$ is in some classical symbol class $S^l$, or may even have an asymptotic expansion $a\sim \sum_{k=0}^{+\infty} h^k a_k$ in this $S^l$ (the term $a_0$ will then be called the principal symbol).
 Let $\tilde H=f(\hat A_1, \ldots, \hat A_d)$ where $f:\R^d\To \R$ is smooth. Let $A=(A_1, \ldots A_d): T^*M \To \R^d$ be the principal symbols of the operators $\hat A_i$. Note that the commutation $[\hat A_i, \hat A_j]=0$ implies the Poisson commutation $\{ A_i, A_j\}=0$. Assume that there is an open subset $W$ of $\R^d$ and a symplectomorphism $T : \T^d\times W\To A^{-1}(W)$ with $A_i\circ T=\xi_j$ (note that, by Arnold-Liouville Theorem,  this situation occurs locally where the differentials of the $A_i$ are linearly independent).
 Then, there exists a Fourier integral operator  $\hat U :L^2(\T^d)\To L^2(M)$ associated with $T$, such that $\hat U \hat U^*=I+O(h^\infty)$ microlocally on $A^{-1}(W)$, and such that
 $$\hat U^*\hat A_j\hat U= hD_{x_j}+\sum_{k\geq 1} h^k S_{j, k}(hD_x)$$
 on $\T^d\times W$ with $S_{j,k}\in{\mathcal C}^\infty(\R^d)$ (see \cite{CdV}, Theorem 78 (1)).

 \medskip

 This may be used to generalize our results to the equation
 \begin{equation}
\left\{
\begin{array}
[c]{l}
ih\partial_{t}\psi_{h}\left(  t,x\right)  =\left(\hat H +h^2 V\right)\,\psi_{h}\left(
t,x\right)  ,\qquad(t,x)\in\R\times M,\\
{\psi_{h}}_{|t=0}=u_{h},
\end{array}
\right.  \label{e:eq2}
\end{equation}
where $V$ is a pseudodifferential operator of order $0$.

\medskip

If $a$ is smooth, compactly supported inside $A^{-1}(W)$, if $\chi$ is supported in $A^{-1}(W)$ taking the value $1$ on the support of $a$, and if $t$ stays in a compact set of $\R$, we have
$$\Op_h (a) S^{\tau_h t}u_h=\Op_h(a)  S^{\tau_h t}\Op_h(\chi)u_h +o(1)$$
as long as $\tau_h\ll h^{-2}$ (or for all $\tau_h$ if $V=0$)
and
$$\Op_h(a) S^{\tau_h t}\Op_h(\chi)u_h=\Op_h(a)\hat U \hat U^* S^{\tau_h t}\hat U \hat U^*\Op_h(\chi)u_h+O(h^\infty).$$
Note that $\hat U^* S^{\tau_h t}\hat U \hat U^*\Op_h(\chi)u_h$
coincides modulo $o(1)$ with $\tilde S^{\tau_h t}\hat U^*\Op_h(\chi)u_h$
where $\tilde S^{\tau_h t}$ is the propagator associated to
\begin{equation}
\begin{array}
[c]{l}
ih\partial_{t}\psi_{h}\left(  t,x\right)  =\left(f_h(hD_{x} ) +h^2 \hat U^*V\hat U\right)\,\psi_{h}\left(
t,x\right)  ,\qquad(t,x)\in\R\times \T^d
\end{array}
.  \label{e:eq3}
\end{equation}
where $f_h(\xi )=f(\xi+\sum_{k\geq 1} h^k S_{k}(\xi))$.

The semiclassical measures associated with equation \eqref{e:eq2}, when restricted to $A^{-1}(W)$, are exactly the images under $T$ of the semiclassical measures coming from \eqref{e:eq3} supported on $\T^d\times W$. Applying Theorem \ref{t:main} to the solutions of \eqref{e:eq3} and transporting the statement by the symplectomorphism $T$, we obtain the following result ~:
\begin{theorem} If $\tau_h\geq h^{-1}$ then the semiclassical measures associated with solutions of~\eqref{e:eq3} are absolutely continuous measures of the lagrangian tori $A^{-1}(\xi)$, for $\bar\mu$-almost every $\xi\in V$ such that $d^2 f(\xi)$ is definite.
\end{theorem}

The observability results could also be rephrased in this more general setting.

\subsection{Organisation of the paper}

When $\tau_h\leq1/h$, the key argument of this article is a second microlocalisation on primitive
submodules which is the subject of Section~\ref{sec:reson} and leads to
Theorems~\ref{mu^Lambda} and~\ref{Thm Properties}.
Sections~\ref{s:finitedistance} and~\ref{s:successive} are devoted to the proof of these two theorems.
At that stage of the paper, the proofs of
 Theorem~\ref{t:precise} and Theorem~\ref{t:main}(2) when $\tau_h\sim 1/h$ are then achieved. Examples are developed in Section~\ref{s:halpha} in order to prove
Theorem~\ref{t:main}(1). Finally, the results concerning hierarchy
of time-scales are proved in Section~\ref{s:hierarchy} (and lead
to Theorem~\ref{t:main} for $\tau_h\gg1/h$), whereas the proof of
Theorem \ref{t:semiclassicalObs} is given in Section
\ref{sec:obs}.

\subsection*{Acknowledgements}

The authors would like to thank Jared Wunsch for communicating to
them the construction in example (3) in Section \ref{s:halpha}.3.
They are also grateful to  Luc Hillairet for helpful discussions
related to some of the results in Section~\ref{s:hierarchy}. Part
of this work was done as F. Maci\`a was visiting the Laboratoire
d'Analyse et de Math\'ematiques Appliqu\'ees at Universit\'e de
Paris-Est, Cr\'eteil. He wishes to thank this institution for its
support and hospitality.

%%%%%%%%%%%%%%%%%%%%%%%%%%%%%%%%%%%%%%%%%%%%%%%%%%%%%%%%%%%%%%%%%
%%%%%%%%%%%%%%%%%%%%%%%%%%%%%%%%%%%%%%%%%%%%%%%%%%%%%%%%%%%%%%%%%%

\section{Two-microlocal analysis of integrable systems on $\mathbb{T}^{d}$}
\label{sec:reson}

In this section, we develop the two-microlocal analysis of the elements of $\widetilde {\mathcal M}(\tau)$ that will be at the core of the proof of Theorems~\ref{t:main}, \ref{t:lagrangian} and~\ref{t:precise} in the case where $\tau_h\leq 1/h$. From now on, we shall assume that the time scale $\left(
\tau_{h}\right)
$ satisfies:
\begin{equation}
(h\tau_{h}) \quad \text{is a bounded sequence.} \label{e:htaubdd}
\end{equation}
Note however that the discussion of section~\ref{s:decompo} does not require this assumption.

\subsection{Invariant measures and a resonant partition of
phase-space\label{s:decompo}}

As in \cite{AnantharamanMacia}, the first step in our strategy to characterise the elements in $\widetilde
{\mathcal{M}}\left(  \tau\right)  $ consists in introducing a partition of
phase-space $T^{\ast}\mathbb{T}^{d}$ according to the order of ``resonance'' of its elements, that induces a decomposition of the
measures $\mu\in\widetilde{\mathcal{M}}\left(  \tau\right)  $.

Using the duality $((\mathbb{R}^{d})^{\ast},\mathbb{R}^{d})$, we denote by $A^\perp \subset \R^d$ the orthogonal of a set $A\in (\mathbb{R}^{d})^{\ast}$ and by $B^\perp \subset (\mathbb{R}^{d})^{\ast}$ the orthogonal of a set $B\in \mathbb{R}^{d}$.
 Recall that $\mathcal{L}$ is the family of all primitive submodules of
$(\mathbb{Z}^{d})^{\ast}$ and that with each $\Lambda\in\mathcal{L}$, we
associate the set $I_{\Lambda}$ defined in~(\ref{def:ILambda}):
$I_{\Lambda}=dH^{-1}(\Lambda^{\bot})$. Denote by $\Omega_{j}\subset
\mathbb{R}^{d}$, for $j=0,...,d$, the set of resonant vectors of order exactly
$j$, that is:
\[
\Omega_{j}:=\left\{  \xi\in(\mathbb{R}^{d})^*:\operatorname*{rk}\Lambda_{\xi
}=d-j\right\}  ,
\]
where
\[
\Lambda_{\xi}:=\left\{  k\in(\mathbb{Z}^{d})^{\ast}:k\cdot dH(\xi)=0\right\}=dH(\xi)^\perp \cap \Z^d
.
\]
Note that the sets $\Omega_{j}$ form a partition of $(\mathbb{R}^{d})^*$, and that
$\Omega_{0}=dH^{-1}\left(  \left\{  0\right\}  \right)  $; more generally,
$\xi\in\Omega_{j}$ if and only if the Hamiltonian orbit $\left\{  \phi
_{s}\left(  x,\xi\right)  :s\in\mathbb{R}\right\}  $ issued from any
$x\in\mathbb{T}^{d}$ in the direction $\xi$ is dense in a subtorus of
$\mathbb{T}^{d}$ of dimension $j$.
The set $\Omega:=\bigcup_{j=0}^{d-1}\Omega_{j}$ is usually called the set of
\emph{resonant }momenta, whereas $\Omega_{d}=(\mathbb{R}^{d})^*\setminus\Omega$ is
referred to as the set of \emph{non-resonant } momenta.
Finally, write
\begin{equation}\label{def:RLambda}
R_{\Lambda}:=I_{\Lambda}\cap\Omega_{d-\operatorname*{rk}\Lambda}.
\end{equation}
Saying that $\xi\in R_{\Lambda}$ is equivalent to any of the following statements:

\begin{enumerate}
\item[(i)] for any $x_{0}\in\IT^{d}$ the time-average $\frac{1}{T}\int_{0}
^{T}\delta_{x_{0}+tdH(\xi)}\left(  x\right)  dt$ converges weakly, as
$T\rightarrow\infty$, to the Haar measure on the torus $x_{0}+\mathbb{T}
_{\Lambda^{\perp}}\text{.}$ Here, we have used the notation $\mathbb{T}
_{\Lambda^{\perp}}:=\Lambda^{\perp}/\left(  2\pi\mathbb{Z}^{d}\cap
\Lambda^{\perp}\right)  $, which is a torus embedded in $\T^d$;

\item[(ii)] $\Lambda_{\xi}=\Lambda$.
\end{enumerate}

\noindent Moreover, if $\operatorname*{rk}\Lambda=d-1$ then $R_{\Lambda}=dH^{-1}\left(
\Lambda^{\perp}\setminus\{0\}\right)  =I_{\Lambda}\setminus\Omega_{0}$. Note
that,
\begin{equation}
(\mathbb{R}^{d})^*=\bigsqcup_{\Lambda\in\mathcal{L}}R_{\Lambda}, \label{part}
\end{equation}
that is, the sets $R_{\Lambda}$ form a partition of $(\mathbb{R}^{d})^*$. As a
consequence, any measure $\mu\in{\mathcal{M}}_{+}(T^{\ast}\R^{d})$ decomposes
as
\begin{equation}
\mu=\sum_{\Lambda\in\mathcal{L}}\mu\rceil_{\mathbb{T}^{d}\times R_{\Lambda}}.
\label{dec}
\end{equation}
Therefore, the analysis of a measure $\mu$ reduces to that of $\mu
\rceil_{\T^{d}\times R_{\Lambda}}$ for all primitive submodule $\Lambda$. 
Given an $H$-invariant measure $\mu$, it turns out that $\mu\rceil
_{\T^{d}\times R_{\Lambda}}$ are utterly determined by the Fourier
coefficients of $\mu$ in $\Lambda$. Indeed, define the complex measures on $\mathbb{R}^{d}
$:
\[
\widehat{\mu}_{k} :=\int_{\mathbb{T}^{d}}\frac{e^{-ik\cdot x}}{\left(
2\pi\right)  ^{d/2}}\mu\left(  dx,\cdot\right)  ,\quad k\in\mathbb{Z}^{d},
\]
so that, in the sense of distributions,
\[
\mu\left(  x,\xi\right)  =\sum_{k\in\mathbb{Z}^{d}}\widehat{\mu}_{k}\left(
\xi\right)  \frac{e^{ik\cdot x}}{\left(  2\pi\right)  ^{d/2}}.
\]
Then, the following proposition holds.

\begin{proposition}
\label{prop:decomposition} Let $\mu\in\mathcal{M}_{+}\left(  T^{\ast
}\mathbb{T}^{d}\right)  $ and $\Lambda\in\mathcal{L}$. The distribution:
\[
\left\langle \mu\right\rangle _{\Lambda}\left(  x,\xi\right)  :=\sum
_{k\in\Lambda}\widehat{\mu}_{k}\left(  \xi\right)  \frac{e^{ik\cdot x}
}{\left(  2\pi\right)  ^{d/2}}
\]
is a finite, positive Radon measure on $T^{\ast}\mathbb{T}^{d}$.\newline
Moreover, if $\mu$ is a positive $H$-invariant measure on $T^{\ast}
\mathbb{T}^{d}$, then every term in the decomposition (\ref{dec}) is a
positive $H$-invariant measure, and
\begin{equation}
\mu\rceil_{\mathbb{T}^{d}\times R_{\Lambda}}=\left\langle \mu\right\rangle
_{\Lambda}\rceil_{\mathbb{T}^{d}\times R_{\Lambda}}. \label{eq:muRLambda}
\end{equation}
Besides, identity~(\ref{eq:muRLambda}) is equivalent to the fact that
$\mu\rceil_{\mathbb{T}^{d}\times R_{\Lambda}}$ is invariant by the
translations
\[
\left(  x,\xi\right)  \longmapsto\left(  x+v,\xi\right)  ,\quad\text{for every
}v\in\Lambda^{\perp}.
\]

\end{proposition}

The proof of Proposition~\ref{prop:decomposition} follows the lines of those
of Lemmas~6 and~7 of~\cite{AnantharamanMacia}. We also point out that this
decomposition depends on the function $H$ through the definition of
$I_{\Lambda}$.
In the following, our aim is to determine $\mu$ restricted to $\T^d\times R_\Lambda$ for any $\Lambda\in {\mathcal L}$.

%%%%%%%%%%%%%%%%%%%%%%%%%%%%%%%%%%%%%%%%%%%%%%%%%%%%

\subsection{Second microlocalization on a resonant  submanifold\label{s:second}}

 Let $\left(
u_{h}\right) $ be a bounded sequence in $L^{2}\left(
\mathbb{T}^{d}\right)  $ and suppose (after extraction of a
subsequence) that its Wigner distributions $w_{h}(t)=w^h_{
S^{t\tau_h}_{h}u_h}$ converge to a
semiclassical measure $\mu\in L^{\infty}\left(  \mathbb{R};\mathcal{M}
_{+}\left(  T^{\ast}\mathbb{T}^{d}\right)  \right)  $ in the
weak-$\ast$ topology of $L^{\infty}\left(
\mathbb{R};\mathcal{D}^{\prime}\left(  T^{\ast
}\mathbb{T}^{d}\right)  \right)  $.

Given $\Lambda\in \cL$, the purpose of this section is to study
the measure
 $\mu\rceil_{\mathbb{T}^{d}\times R_{\Lambda}}$ by performing a second microlocalization along $I_\Lambda$ in
 the spirit of~\cite{Fermanian2micro, FermanianShocks, FermanianGerardCroisements, NierScat, MillerThesis}
and~\cite{AnantharamanMacia, MaciaTorus}. By
Proposition~\ref{prop:decomposition}, it suffices to characterize
the action of $\mu\rceil_{\mathbb{T}^{d}\times R_{\Lambda}}$ on
test functions having only $x$-Fourier modes in $\Lambda$. With
this in mind, we shall introduce two auxiliary ``distributions''
which describe more precisely how $w_{h}\left( t\right)  $
concentrates along $\mathbb{T}^{d}\times I_\Lambda$. They are
actually not mere distributions, but lie in the dual of the class
of symbols $\mathcal{S}_{\Lambda}^{1}$ that we define below.

\medskip

In what follows, we fix $\xi_{0}\in R_{\Lambda}$ such that
$d^{2}H(\xi_{0})$ is definite and,   by applying a cut-off in
frequencies to the data, we restrict our discussion to
normalised sequences of initial data $(u_{h})$ that satisfy:
\[
\widehat{u_{h}}\left(  k\right)  =0,\qquad\text{for }hk\in\mathbb{R}
^{d}\setminus B(\xi_{0};\eps/2),
\]
where $B(\xi_{0},\eps/2)$ is the ball of radius $\eps/2$
centered at $\xi_{0}$. The parameter $\eps>0$ is taken small
enough, in order that
\[
d^{2}H(\xi)\text{ is definite for  all }\xi\in B(\xi_{0},\eps)\text{;}
\]
this implies that $I_{\Lambda}\cap B(\xi_{0},\eps)$ is a
submanifold of dimension $d-\mathrm{rk}\,\Lambda$, everywhere
transverse to $\la\Lambda\ra$, the vector subspace of $(\R^{d})^*$
generated by $\Lambda$. Note that this is actually achieved under
the weaker hypothesis that $d^{2}H(\xi)$ is non-singular and
defines a definite bilinear form on $\la \Lambda \ra \times \la
\Lambda \ra$ (Section \ref{s:suff} gives a set of assumptions which is weaker than definiteness but sufficient for our results). By eventually reducing $\eps$, we have
\[
B(\xi_{0},\eps/2)\subset(I_{\Lambda}\cap
B(\xi_{0},\eps))\oplus\la\Lambda\ra,
\]
by which we mean that any element $\xi\in B(\xi_{0},\eps/2)$ can
be decomposed in a unique way as $\xi=\sigma+\eta$ with $\sigma\in
I_{\Lambda}\cap B(\xi _{0},\eps)$ and $\eta\in\la\Lambda\ra$. We
thus get a map
\begin{align}
F:B(\xi_{0},\eps/2)  & \To\left(  I_{\Lambda}\cap
B(\xi_{0},\eps)\right)
\times\la\Lambda\ra\label{e:defF}\\
\xi & \longmapsto(\sigma(\xi),\eta(\xi))\nonumber
\end{align}

With this decomposition  of the space of frequencies, we associate
two-microlocal test-symbols.

\begin{definition} We denote by
$\mathcal{S}_{\Lambda}^{1}$ the
class of smooth functions $a\left(  x,\xi,\eta\right)  $ on $T^{\ast}
\mathbb{T}^{d}\times \la\Lambda\ra$ that are:

\begin{enumerate}
\item[(i)] compactly supported on $\left(  x,\xi\right)  \in
T^{\ast }\mathbb{T}^{d}$, $\xi\in B(\xi_0,\eps/2)$,\smallskip

\item[(ii)] homogeneous of degree zero at infinity w.r.t.
$\eta\in\la\Lambda\ra$, $\emph{i.e.}$ such that there exist
$R_{0}>0$ and $a_{\rm{hom}}\in \cC_{c}^{\infty}\left(
T^{\ast}\mathbb{T}^{d}\times\mathbb{S}\la\Lambda\ra\right)  $
with
\[
a\left(  x,\xi,\eta\right)  =a_{\rm{hom}}\left(
x,\xi,\frac{\eta }{\left\vert \eta\right\vert }\right)
\text{,\quad for }\left\vert
\eta\right\vert >R_{0}\text{ and }\left(  x,\xi\right)  \in T^{\ast}
\mathbb{T}^{d}
\]
(we have denoted by $\mathbb{S}\la\Lambda\ra$ the unit sphere in
$\la\Lambda\ra\subseteq(\mathbb{R}^{d})^{*}$, identified later on with the sphere at infinity);\smallskip

\item[(iii)] such that their non vanishing Fourier coefficients
(in the $x$
variable) correspond to frequencies $k\in\Lambda$:
\[
a\left(  x,\xi,\eta\right)
=\sum_{k\in\Lambda}\widehat{a}_{k}\left(\xi ,\eta\right)
\frac{e^{ik\cdot x}}{\left(  2\pi\right) ^{d/2}}.
\]
We will also express this fact by saying that $\emph{a}$\emph{ has
only $x$-Fourier modes in $\Lambda$.}
\end{enumerate}
\end{definition}

The index $1$ in the notation ${\mathcal S}^1_\Lambda$ refers to the fact that we have added {\em one} variable ($\eta$) to the standard class of symbols corresponding to the second microlocalisation. In Section~\ref{s:successive}, we will perform successive higher order microlocalisations corresponding to the addition of $k\geq 1$ variables and we will consider spaces denoted ${\mathcal S}^k_\Lambda$.

\medskip

For $a\in{\mathcal S}^1_\Lambda$, we introduce the notation
$$\Op^\Lambda_h(a(x,\xi,\eta)):=\Op_h\left(a\left(x,\xi,\tau_h\eta(\xi)\right)\right).$$
Notice that, for all
$\beta\in\N^{d}$,
\begin{equation}\label{estforgain}
\left\Vert \partial_{\xi}^{\beta}\left(a\left(
x,h\xi,\tau_{h}{\eta(h\xi )}\right)\right)\right\Vert
_{L^{\infty}}\leq C_{\beta }\left(  \tau_{h}{h}\right)
^{|\beta|}.
\end{equation}
The Calder\'{o}n-Vaillancourt theorem (see~\cite{CV} or the
appendix of~\cite{AnantharamanMacia} for a precise statement)
therefore ensures that
there exist $N\in\N$ and $C_{N}>0$ such that
\begin{equation}
\forall a\in{\mathcal{S}}_{\Lambda}^{1},\quad\left\Vert
\Op_{h}^{\Lambda }(a)\right\Vert
_{{\mathcal{L}}(L^{2}(\R^{d}))}\leq C_{N}\sum_{|\alpha|\leq
N}\Vert\partial_{x,\xi,\eta}^{\alpha}a\Vert_{L^{\infty}},\label{eq:CVeta}
\end{equation}
since $\left(  h\tau_{h}\right)  $ is assumed to be bounded. Therefore the family of operators $\Op^\Lambda_h(a)$ is a bounded family of $L^2(\T^d)$.

We are going to use this formalism to decompose the Wigner transform $w_h(t)$.
Let $\chi\in \cC_{c}^{\infty}\left(\la\Lambda\ra\right)  $ be a
nonnegative cut-off function that is identically equal to one near
the origin. For $a\in\cS_{\Lambda}^{1}$,  $R>1$, $\delta<1$, we
decompose $a$ into:
$a(x,\xi,\eta)=\sum_{j=1}^3a_j(x,\xi,\eta)$ with
\begin{eqnarray}
\nonumber
 a_1(x,\xi,\eta)  & :=  & a(x,\xi,\eta)\left(1-\chi\left({\eta\over R}\right)\right)\left(1- \chi\left({\eta(\xi)\over\delta}\right) \right),\\
 \label{def:a2}
 a_2(x,\xi,\eta) & := & a(x,\xi,\eta)\left(1-\chi\left({\eta\over R}\right)\right) \chi\left({\eta(\xi)\over\delta}\right) ,\\
 \label{def:a3}
 a_3(x,\xi,\eta) & := & a(x,\xi,\eta)\chi\left({\eta\over R}\right).
 \end{eqnarray}
 Since any smooth compactly supported function with Fourier modes in $\Lambda$ can be viewed as an element of ${\mathcal S}^1_\Lambda$ (which is constant in the variable $\eta$),  this induces a decomposition of the Wigner
 distribution: $$w_h(t)=w_{h,R, \delta}^{I_\Lambda}\left(  t\right)+w_{I_\Lambda,h,R}\left(  t\right)+  w^{I_\Lambda^c}_{h,R, \delta}\left(  t\right),$$
where:
\[
\left\langle w_{h,R, \delta}^{I_\Lambda}\left(  t\right)
,a\right\rangle :=\int _{T^{\ast}\mathbb{T}^{d}}   a_2\left(
x,\xi,\tau_h \eta(\xi) \right)  w_{h}\left(  t\right) \left(
dx,d\xi\right)  ,
\]
\begin{equation}
\left\langle w_{I_\Lambda,h,R}\left(  t\right)  ,a\right\rangle
:=\int_{T^{\ast }\mathbb{T}^{d}} a_3\left(  x,\xi,\tau_h\eta(\xi)
\right)  w_{h}\left(
t\right)  \left(  dx,d\xi\right)  ,
\end{equation}
and
$$
\left\langle w^{I_\Lambda^c}_{h,R, \delta}\left(  t\right)
,a\right\rangle :=\int_{T^{\ast }\mathbb{T}^{d}} a_1\left(
x,\xi,\tau_h\eta(\xi)\right)  w_{h}\left(
t\right)  \left(  dx,d\xi\right)  ,
$$
that we shall analyse in the limits $h\To 0^+$, $R\To +\infty$ and
$\delta\To 0$ (taken in that order).

\medskip

The distributions $w_{h,R,\delta}^{I_\Lambda}(t)$ and
$w_{I_\Lambda,h,R}(t)$ can be expressed for all $t\in\R$ by
\begin{eqnarray}\label{eq:w1}
\langle w_{h,R,\delta}^{I_\Lambda}(t) ,a\rangle & = & \langle u_h, S^{\tau_ht *}_{h}\Op^\Lambda_h(a_2)S^{\tau_ht}_{h}u_h\rangle_{L^{2}(\T^d)} ,\\
\label{eq:w2}
\langle w_{I_\Lambda,h,R}^{\Lambda}(t) ,a\rangle & = & \langle
u_h,
S^{\tau_ht *}_{h}\Op^\Lambda_h(a_3)S^{\tau_ht}_{h}u_h\rangle_{L^{2}(\T^d)}
.
\end{eqnarray}
As a consequence of~(\ref{eq:CVeta}), both $ w_{h,R, \delta}^{I_\Lambda}$ and
$w_{I_\Lambda,h,R}$ are
bounded in $L^{\infty}\left(  \mathbb{R};\left(  \mathcal{S}_{\Lambda}
^{1}\right)  ^{\prime}\right)  $.

\medskip

The first observation  is that
$$\displaylines{\qquad
\lim_{\delta\To 0}\lim_{R\rightarrow\infty}\lim_{h\rightarrow
0}\int_{\mathbb{R}} \theta(t)\left\langle w^{I_\Lambda^c}_{h,R,
\delta}\left( t\right)  ,a\right\rangle dt \hfill\cr\hfill
=\int_{\mathbb{R}}\int
_{T^{*}\T^d} \theta(t) a_{\rm{hom}}\left(x,\xi,{\eta(\xi)\over
|\eta(\xi)|} \right)\mu(t, dx,d\xi)\rceil_{\T^d\times I_\Lambda^c
}dt\qquad \cr}$$
where $\mu\in\mathcal{\widetilde{M}}\left( \tau_h\right) $ is
the semiclassical measure obtained through the sequence $(u_h)$.
Since $R_\Lambda\subset I_\Lambda$, the restriction of the measure thus obtained to $\T^d\times
R_\Lambda$ vanishes, and we do not need to further analyse
the term involving the distribution $w^{I_\Lambda^c}_{h,R,
\delta}\left( t\right)$.

\medskip

Then,
after possibly extracting
subsequences, one defines limiting objects $\tilde\mu_\Lambda$ and $\tilde\mu^\Lambda$ such that  for every $\varphi\in L^{1}\left(
\mathbb{R}\right)  $ and
$a\in\mathcal{S}_{\Lambda}^{1}$,
\[
\int_{\mathbb{R}}\varphi\left(  t\right)  \left\langle
\tilde{\mu}^{\Lambda
}\left(  t,\cdot\right)  ,a\right\rangle dt:=\lim_{\delta\To 0}\lim_{R\rightarrow\infty}
\lim_{h\rightarrow0^{+}}\int_{\mathbb{R}}\varphi\left( t\right)
\left\langle w_{h,R, \delta}^{I_\Lambda}\left( t\right)
,a\right\rangle dt,
\]
and
\begin{equation}
\int_{\mathbb{R}}\varphi\left(  t\right)  \left\langle
\tilde{\mu}_{\Lambda
}\left(  t,\cdot\right)  ,a\right\rangle dt:=\lim_{R\rightarrow\infty}
\lim_{h\rightarrow0^{+}}\int_{\mathbb{R}}\varphi\left( t\right)
\left\langle
w_{I_\Lambda,h,R}\left(  t\right)  ,a\right\rangle dt. \label{doublelim}
\end{equation}
From the decomposition $w_h(t)=w_{h,R, \delta}^{I_\Lambda}\left(  t\right)+w_{I_\Lambda,h,R}\left(  t\right)+  w^{I_\Lambda^c}_{h,R, \delta}\left(  t\right)$ (when testing against symbols having Fourier modes in $\Lambda$), it is immediate that the measure $\mu\left(  t,\cdot\right)
\rceil_{\mathbb{T}^{d}\times R_{\Lambda}}$ is related to
$\tilde\mu^\Lambda$ and $\tilde\mu_\Lambda$ according to the
following Proposition.

\begin{proposition}\label{p:decomposition}
Let
\[
\mu^{\Lambda}\left(  t,\cdot\right)  :=\int_{\left\langle
\Lambda\right\rangle
}\tilde{\mu}^{\Lambda}\left(  t,\cdot,d\eta\right)  \rceil_{\mathbb{T}
^{d}\times R_{\Lambda}},\quad\mu_{\Lambda}\left(  t,\cdot\right)
:=\int_{\left\langle \Lambda\right\rangle
}\tilde{\mu}_{\Lambda}\left( t,\cdot,d\eta\right)
\rceil_{\mathbb{T}^{d}\times R_{\Lambda}}.
\]
Then both $\mu^{\Lambda}\left(  t,\cdot\right)  $ and
$\mu_{\Lambda}\left(
t,\cdot\right)  $ are $H$-invariant positive measures on $T^{\ast}\mathbb{T}^{d}$and satisfy:
\begin{equation}
\mu\left(  t,\cdot\right)  \rceil_{\mathbb{T}^{d}\times R_{\Lambda}}
=\mu^{\Lambda}\left(  t,\cdot\right)  +\mu_{\Lambda}\left(
t,\cdot\right)  . \label{eq:decomposition}
\end{equation}
\end{proposition}

This proposition motivates the analysis of the structure of the
accumulation points $\tilde{\mu}_{\Lambda}\left( t,\cdot\right)  $
and $\tilde{\mu}^{\Lambda}\left(  t,\cdot\right)  $.
 It turns out that both $\tilde{\mu}^{\Lambda}$
and $\tilde{\mu}_{\Lambda}$ have some extra regularity in the
variable $x$, although for two different reasons. Our next two
results form one of the key steps towards the proof of Theorem
\ref{t:main}.

\begin{remark}
All the results in this section remain valid if the Hamiltonian
$H_{h}$ depends on $h$ as stated in Remark \ref{r:Hh}. The proofs
are completely
analogous, the only difference being that the the resonant manifolds $%
I_{\Lambda }$ and the coordinate system $F=\left( \sigma ,\eta
\right) $ now also vary with $h$. Therefore, the definitions of
$w_{h,R,\delta }^{I_{\Lambda }}$ and $w_{I_{\Lambda },h,R}$ should
be modified accordingly.
\end{remark}

\subsection{Properties of two-microlocal semiclassical measures}

We define, for $\left(  x,\xi,\eta\right)  \in T^{\ast}
\mathbb{T}^{d}\times \left(\la\Lambda\ra\setminus \{0\}\right)$ and $s\in\mathbb{R}$,
$$\displaylines{
\phi_{s}^{0}\left(  x,\xi,\eta\right)  :=\left(
x+sdH(\xi),\xi,\eta\right)  ,\cr \phi_{s}^{1}\left(
x,\xi,\eta\right)  :=\left(
x+sd^2H(\sigma(\xi))\cdot\frac{\eta}{|\eta|},\xi,\eta\right)
.\cr}$$
This second definition extends in an obvious way to $\eta\in \mathbb{S}\la\Lambda\ra$ (the sphere at infinity). On the other hand, the map
$\left(
x,\xi,\eta\right)\mapsto \phi_{s|\eta|}^{1}\left(
x,\xi,\eta\right)$ extends to $\eta=0$.

\medskip

We first focus on the measure $\tilde\mu^\Lambda$. We point out that because of the existence of
$R_{0}>0$ and of $a_{\rm{hom}}\in C_{c}^{\infty}\left(
T^{\ast }\mathbb{T}^{d}\times\mathbb{S}\left\langle
\Lambda\right\rangle \right)  $ such that
\[
a\left(  x,\xi,\eta\right)  =a_{\rm{hom}}\left(
x,\xi,\frac{\eta }{\left\vert \eta\right\vert }\right)
,\quad\text{for }\left\vert \eta\right\vert \geq R_{0},
\]
 the value $\left\langle
w_{h,R,\delta}^{I_\Lambda}\left( t\right)  ,a\right\rangle $ only
depends on $a_{\rm{hom}}$. Therefore, the limiting object
$\tilde{\mu}^{\Lambda}\left(  t,\cdot\right)\in \left(\mathcal{S}_{\Lambda}^{1}\right)'$ is
zero-homogeneous in the last variable $\eta\in\mathbb{R}^{d}$, supported at infinity, and,
by construction, it is supported on $\xi\in I_\Lambda$. This can
be also expressed as the fact that $\tilde\mu^\Lambda$ is a
``distribution'' on $\T^{d}\times I_{\Lambda }\times
\overline{\la\Lambda\ra}$ (where $\overline{\la\Lambda\ra}$ is the
compactification of $\la\Lambda\ra$ by adding the sphere
$\mathbb{S}\la\Lambda\ra$ at infinity) supported on
$\{\eta\in\mathbb{S}\la\Lambda\ra\}$. Moreover, we have the following result.

\begin{theorem}\label{mu^Lambda}
$\tilde{\mu}^{\Lambda}\left( t,\cdot\right)  $ is a positive
measure on $\T^{d}\times I_{\Lambda }\times
\overline{\la\Lambda\ra}$ supported on the sphere at infinity
$\mathbb{S}\la\Lambda\ra$ in the variable $\eta$. Besides, for
a.e. $t\in\mathbb{R}$, the measure $\tilde{\mu}^{\Lambda}\left(
t,\cdot\right)  $ satisfies the invariance properties:
\begin{equation}
\left(  \phi_{s}^{0}\right)  _{\ast}\tilde{\mu}^{\Lambda}\left(
t,\cdot\right)  =\tilde{\mu}^{\Lambda}\left(  t,\cdot\right)
,\quad\left( \phi_{s}^{1}\right)
_{\ast}\tilde{\mu}^{\Lambda}\left(  t,\cdot\right)
=\tilde{\mu}^{\Lambda}\left(  t,\cdot\right)  ,\quad
s\in\mathbb{R}.
\label{2minv}
\end{equation}
\end{theorem}

Note that this result holds whenever $\tau_h\ll 1/h$ or
$\tau_h=1/h$. This is in contrast with the situation we encounter
when dealing with $\tilde\mu_\Lambda(t,\cdot)$. The regularity of
this object indeed depends on the properties of the scale.

\begin{theorem}
\label{Thm Properties} (1) The distributions
$\tilde\mu_\Lambda(t,\cdot)$ are supported on $\T^d\times
I_\Lambda\times \langle \Lambda\rangle$ and are continuous with
respect to $t\in\R$. \\
(2) If $\tau_h\ll1/h$ then $\tilde{\mu}_{\Lambda}\left(
t,\cdot\right)  $ is a positive measure. \\
(3) If $\tau_h=1/h$,
  the projection of $\tilde{\mu}_{\Lambda}\left(
t,\cdot\right)  $ on $T^{\ast}\mathbb{T}^{d}$ is a positive
measure, whose projection on $\mathbb{T}^{d}$ is absolutely
continuous with respect to the Lebesgue measure.\\
(4) If $\tau_h\ll 1/h$, then $\tilde\mu_{\Lambda}$ satisfy the following
propagation law:
\begin{equation}\label{eq:etafini}
\forall t\in\R,\qquad \tilde\mu_{\Lambda}(t,x,\xi,\eta)=(
\phi^1_{t|\eta|})_* \tilde\mu_{\Lambda}(0,x,\xi,\eta) .
\end{equation}
\end{theorem}
Note that (4) implies the continuous dependence of $\tilde\mu_\Lambda(t,\cdot)$  with
respect to $t$ in the case $\tau_h\ll 1/h$.
For $\tau_h= 1/h$ the dependence of $\tilde\mu_{\Lambda}(t,x,\xi,\eta)$ on $t$ will be investigated in Section~\ref{s:finitedistance}.

\begin{remark}\label{rem:decmuenL}
Consider the decomposition
$
\mu\left(  t,\cdot\right)
=\sum_{\Lambda\in\mathcal{L}}\mu^{\Lambda}\left( t,\cdot\right)
+\sum_{\Lambda\in\mathcal{L}}\mu_{\Lambda}\left( t,\cdot\right)  .
$
given by Proposition~\ref{p:decomposition}.
When $\tau_h=1/h$, Theorem \ref{Thm Properties}(3) implies that the second term defines a
positive measure whose projection on $\mathbb{T}^{d}$ is
absolutely continuous with respect to the Lebesgue measure.
\end{remark}

Theorem~\ref{Thm Properties} calls for a few comments.

\medskip

The fact that the distribution
$\tilde \mu_\Lambda$ is  supported on $\T^{d}\times I_{\Lambda
}\times \la \Lambda \ra$ is straightforward. Indeed, we have for all~$t$,
\begin{equation}\label{eq:a3decompose}
\langle w_{I_\Lambda, h, R}(t), a(x,\xi,\eta)\rangle = \langle
w_{I_\Lambda,h,R}(t), a(x,\sigma(\xi),\eta) \rangle+O(\tau_h^{-1})
\end{equation}
since, by~(\ref{eq:CVeta}),
\begin{eqnarray*}
\Op_h^\Lambda(a_3(x,\xi,\eta)) & = &
\Op_h^\Lambda(a(x,\sigma(\xi)+\tau_h^{-1} \eta,\eta)
\chi(\eta/R))\\ & = & \Op_h^\Lambda(a(x,\sigma(\xi),\eta)
\chi(\eta/R))+O(\tau_h^{-1})
\end{eqnarray*}
where the $O(\tau_h^{-1})$ term is understood in the sense of the
operator norm of ${\mathcal L}(L^2(\R^d))$ and depends on $R$ (the
fact that we first let $h$ go to $0^+$ is crucial here).

\medskip

When
$\tau_{h}\ll1/h$ the quantization of our symbols
generates a semi-classical pseudodifferential calculus with gain
$h\tau_{h}$. The operators $\Op_{h}^{\Lambda
}(a)$ are semiclassical both in $\xi$ and~$\eta$. This implies that the accumulation points
$\tilde\mu^\Lambda$ and $\tilde\mu_\Lambda$ are positive measures (see for instance~\cite{MillerThesis} or~\cite{FermanianGerardCroisements}). \label{r:positiv}

\medskip

When  $\tau_h=1/h$, we will see in Theorem~\ref{prop:opvame} in Section~\ref{s:finitedistance} that the distributions
$\tilde{\mu}_{\Lambda}\left( t,\cdot\right)$ satisfy an invariance law that can be interpreted in
terms of a Schr\"odinger flow type propagator.

\medskip

Let us now comment on the invariance by the flows. Note first that
of major importance is the observation that
 for all $\xi\in\left(
\mathbb{R}^{d}\right)^*\setminus C_{H}$ (recall that
$C_{H}$ stands for the points where the Hessian
$d^{2}H\left(  \xi\right)  $ is not definite) we have the decomposition $\mathbb{R}^{d}
=\Lambda^{\perp}\oplus d^{2}H\left(  \xi\right)  \left\langle
\Lambda \right\rangle $. Therefore, the flows $\phi_{s}^{0}$ and
$\phi_{s}^{1}$ are independent on $\mathbb{T}^{d}\times\left(
R_{\Lambda}\setminus C_{H}\right) \times\left\langle
\Lambda\right\rangle $. Then, the following remark holds:

\begin{remark}\label{r:nice}
 In the case where $\operatorname*{rk}\Lambda=1$ then (\ref{2minv})
implies that, for a.e. $t\in\mathbb{R}$, and for any
$\nu$ in the $1$-dimensional space $\la\Lambda\ra$, the
measure $\tilde{\mu}^{\Lambda}\left(  t,\cdot\right)  \rceil_{\mathbb{T}
^{d}\times R_{\Lambda}\times\left\langle \Lambda\right\rangle }$
is invariant under
\[
(x,\sigma,\eta)\longmapsto(x+d^{2}H(\sigma)\cdot\nu,\sigma,\eta).
\]
On the other hand, the invariance by the Hamiltonian flow and
Proposition~\ref{prop:decomposition}, imply that
$\tilde{\mu}^{\Lambda}\left( t,\cdot\right)
\rceil_{\mathbb{T}^{d}\times R_{\Lambda}\times\left\langle
\Lambda\right\rangle }$ is also invariant under
\[
(x,\sigma,\eta)\longmapsto(x+v,\sigma,\eta)
\]
for every $v$ in the hyperplane $\Lambda^{\perp}$. Using the independence of the different flows and
the fact that the Hessian
$d^{2}H\left(  \sigma\right)  $ is definite on the support of $\tilde{\mu}^{\Lambda}\left(
t,\cdot\right) \rceil_{\mathbb{T}^{d}\times
R_{\Lambda}\times\left\langle \Lambda \right\rangle }$, we conclude that the
measure $\tilde{\mu}^{\Lambda}\left(  t,\cdot\right)
\rceil_{\mathbb{T}^{d}\times R_{\Lambda}\times\left\langle \Lambda
\right\rangle }$ is constant in $x\in\mathbb{T}^{d}$ in this case. For $\operatorname*{rk}\Lambda >1$, we will develop a similar argument thanks to successive microlocalisations (see section~\ref{s:successive}).
\end{remark}

\medskip

In the next subsection, we
 prove the invariance properties stated in  Theorems \ref{mu^Lambda} and~\ref{Thm Properties} (2) and (4). For $\tau_h=h^{-1}$ the  detailed analysis of the measure $\tilde\mu_\Lambda$ is performed in section~\ref{s:finitedistance} and the proof of the absolute continuity of its projection on $T^*\T^d$ is done in section~\ref{s:successive}.

%%%%%%%%%%%%%%%%%%%%%%%%%%%%%%%%%%%%%%%%%%%%%%%%%%%%

\subsection{Invariance properties of two-microlocal semiclassical measures}\label{sec:infini}

\begin{proof}[Proof of Theorem~\ref{mu^Lambda}.]
The positivity
of $\tilde{\mu}^{\Lambda}\left(  t,\cdot\right)  $ can be deduced
following the lines of \cite{FermanianGerardCroisements} \S 2.1,
or those of the proof of Theorem~1 in \cite{GerardMDM91}; see also
the appendix of~\cite{AnantharamanMacia}. The proof
of invariance of $\tilde{\mu}^{\Lambda}\left(  t,\cdot\right)  $ under $\phi_s^0$ is similar to the proof of invariance of $\mu$ under $\phi_s$ done in the appendix.
\\
Let us now check the invariance property (\ref{2minv}).
Using~(\ref{def:a2}), we have (along any convergent subsequence)
\begin{align}
\int_{\mathbb{R}}\varphi(t)\langle\tilde{\mu}^{\Lambda}(t,\cdot),a\rangle
dt &
=\lim_{\delta\rightarrow0}\lim_{R\rightarrow+\infty}\lim_{h\rightarrow
0}\int_{\mathbb{R}}\varphi(t)\left\langle
w_{h,R,\delta}^{I_{\Lambda}}\left(
t\right)  ,a\right\rangle dt\label{e:limmul}\\
&
=\lim_{\delta\rightarrow0}\lim_{R\rightarrow+\infty}\lim_{h\rightarrow
0}\int_{\mathbb{R}}\varphi(t)\langle w_{h}(t),a_{2}\left(
x,\xi,\tau_{h} \eta(\xi)\right)  \rangle dt\nonumber.
\end{align}
Notice that the symbol
\[
a_{2}\circ\phi_{s}^{1}(x,\xi,\eta)=a_{2}\left(
x+sd^{2}H(\sigma(\xi )){\frac{\eta}{|\eta|}},\xi,\eta\right) ,
\]
is a well-defined element of ${\mathcal{S}}_{\Lambda}^{1}$, since, for fixed $R$,
$a_{2}$ is identically equal to zero near $\eta=0$; moreover
\[
\forall\omega\in{\mathbb{S}}\la\Lambda\ra ,\;\;(a_{2}
\circ\phi_{s}^{1})_{\mathrm{hom}}(x,\xi,\omega)=a_{\mathrm{hom}}
(x+sd^{2}H(\sigma(\xi))\omega,\xi,\omega).
\]
 We write
$$ \frac{d}{ds}|_{ s=0}\left(  a_{2}\circ\phi_{s}^{1}\right)  \left(
x,\xi,\tau_{h}\eta (\xi)\right)=\left( d^{2}H(\sigma(\xi)){\frac{\eta(\xi)}{|\eta(\xi)|}}\right)\cdot \partial_ x a_{2}\left(  x ,\xi,\tau_{h}\eta (\xi)\right).$$
Using the Taylor expansion
$$d^{2}H(\sigma(\xi))\eta(\xi)
+G(\xi)[\eta(\xi),\eta(\xi)]=dH(\xi)-dH(\sigma(\xi))$$
where \begin{equation}
G(\xi)=\int_{0}^{1}d^{3}H(\sigma(\xi)+t\eta(\xi))(1-t)dt,
\label{e:G}
\end{equation}
is uniformly bounded, and taking into account the fact that $\eta(\xi)=O(\delta)$ on the support of $a_2$, we have
$$\left( d^{2}H(\sigma(\xi)){\frac{\eta(\xi)}{|\eta(\xi)|}}\right)\cdot \partial_ x a_{2}\left(  x ,\xi,\tau_{h}\eta (\xi)\right)=\left( {\frac{dH(\xi)-dH(\sigma(\xi))}{|\eta(\xi)|}}\right)\cdot \partial_ x a_{2}\left(  x ,\xi,\tau_{h}\eta (\xi)\right) +O(\delta).$$
Because $a_2$ has only $x$-Fourier coefficients in $\Lambda$ and $dH(\sigma(\xi))\in\Lambda^\perp$, we can write
$$\left( {\frac{dH(\xi)-dH(\sigma(\xi))}{|\eta(\xi)|}}\right)\cdot \partial_ x a_{2}\left(  x ,\xi,\tau_{h}\eta (\xi)\right) =
\left( {\frac{dH(\xi)}{|\eta(\xi)|}}\right)\cdot \partial_ x a_{2}\left(  x ,\xi,\tau_{h}\eta (\xi)\right).
$$
Note now that
\begin{multline*}\Op_h\left({\frac{dH(\xi)}{|\eta(\xi)|}}\cdot \partial_ x a_{2}\left(  x ,\xi,\tau_{h}\eta (\xi)\right)\right)
=\frac{i}{h}\left[H(hD_x)+h^2 {\mathbf V}_h(t), \Op_h\left( \frac{ a_{2}\left(  x ,\xi,\tau_{h}\eta (\xi)\right)}{|\eta(\xi)|}\right)\right]
\\+ O(h)+O\left(\frac{h\tau_h}{R}\right).
\end{multline*}
For the last term, we have only used that ${\mathbf V}_h(t)$ is a bounded operator on $L^2$ and
$$\left\Vert\Op_h\left( \frac{ a_{2}\left(  x ,\xi,\tau_{h}\eta (\xi)\right)}{|\eta(\xi)|}\right)\right\Vert_{L^2\To L^2}=O\left(\frac{\tau_h}{R}\right)$$
since ${\frac{1}{|\eta(\xi)|}}= O\left(\frac{\tau_h}{R}\right)$ on the support of $a_2$.

To conclude, take a test function $\varphi(t)\in \cC_c^\infty(\R)$
(those are dense in $L^1(\R)$).
$$\displaylines{
\int_{\R}\varphi(t) \left\la w_{h,R, \delta}^{I_\Lambda}\left(  t\right), \frac{d}{ds}|_{ s=0}\left(  a_{2}\circ\phi_{s}^{1}\right)  \left(
x,\xi,\tau_{h}\eta (\xi)\right) \right\ra dt \hfill\cr\hfill
= \int_{\R}\varphi(t)
\left\la S_h^{\tau_h t} u_h,  \frac{i}{h}\left[H(hD_x)+h^2 {\mathbf V}_h(t), \Op_h\left( \frac{ a_{2}\left(  x ,\xi,\tau_{h}\eta (\xi)\right)}{|\eta(\xi)|}\right)\right] S_h^{\tau_h t} u_h\right\ra dt \cr\hfill
+ O(h)+O\left(\frac{h\tau_h}{R}\right)+O(\delta)\qquad\cr\hfill
 =\frac{1}{\tau_h} \int_{\R}\varphi(t)
\frac{d}{dt}\left\la S_h^{\tau_h t} u_h,  \Op_h\left( \frac{ a_{2}\left(  x ,\xi,\tau_{h}\eta (\xi)\right)}{|\eta(\xi)|}\right) S_h^{\tau_h t} u_h\right\ra dt + O(h)+O\left(\frac{h\tau_h}{R}\right)+O(\delta)\cr\hfill
=-\frac{1}{\tau_h} \int_{\R}\varphi'(t)
 \left\la S_h^{\tau_h t} u_h,  \Op_h\left( \frac{ a_{2}\left(  x ,\xi,\tau_{h}\eta (\xi)\right)}{|\eta(\xi)|}\right) S_h^{\tau_h t} u_h\right\ra dt + O(h)+O\left(\frac{h\tau_h}{R}\right)+O(\delta)\cr\hfill
 =O(\tau_h^{-1})+ O(h)+O\left(\frac{h\tau_h}{R}\right)+O(\delta).\cr
}$$
Taking the limit $h\To 0$ followed by $R\To +\infty$ and $\delta\To 0$, we obtain
$$\int_{\R}\varphi(t) \left\la  \tilde\mu^\Lambda(t), \frac{d}{ds}|_{ s=0}\left(  a\circ\phi_{s}^{1}\right)    \right\ra dt =0$$
for any $\varphi$ and $a$, which ends the proof of Theorem~\ref{mu^Lambda}.
\end{proof}

\medskip

\begin{proof}[Proof of (1), (2) and (4) of
Theorem~\ref{Thm Properties} for $h\tau_h\To 0$.]
The statement
on the support of the measure~$\tilde{\mu}_{\Lambda}$ has already been discussed in
Section \ref{s:second} (after Remark~\ref{rem:decmuenL}).
The positivity of $\tilde{\mu}_{\Lambda}$ is standard once we notice that the two-scale quantization admits
the gain $h\tau_{h}$ (in view of~(\ref{estforgain})). Note also that (4) implies the continuous dependence with respect to $t$.\\
Thus let us prove part (4) of the theorem. The propagation law (and hence, the
continuity with respect to $t$) is proved as follows.
Let
$$\tilde \phi^1_t(x, \xi, \eta)= \phi^1_{t|\eta|} (x, \xi, \eta)=(x+t d^2H(\sigma(\xi))\eta, \xi, \eta).$$
We write
$$\frac{d}{dt}_{|t=0}a_3\circ \tilde \phi^1_t(x, \xi, \tau_h\eta(\xi))=\tau_h d^2H(\sigma(\xi))\eta(\xi)\cdot\partial_x a_3(x, \xi, \tau_h\eta(\xi))$$
and the same argument as in the previous proof now yields
$$\tau_h d^2H(\sigma(\xi))\eta(\xi)\cdot\partial_x a_3(x, \xi, \tau_h\eta(\xi))=\tau_h dH(\xi)\cdot\partial_x a_3(x, \xi, \tau_h\eta(\xi)) +O\left(\frac{R^2}{\tau_h}\right)$$
where we now use that $|\eta(\xi)|= O\left(\frac{R}{\tau_h}\right)$ on the support of $a_3$.
Note now that
$$\tau_h \Op_h\left(dH(\xi)\cdot\partial_x a_3(x, \xi, \tau_h\eta(\xi))\right)
=\frac{i\tau_h}{h}\left[H(hD_x)+h^2{\mathbf V}_h(t), \Op_h\left( a_3(x, \xi, \tau_h\eta(\xi)\right)\right]
+O\left(h\tau_h\right),
$$
using only the fact that ${\mathbf V}_h(t)$ is a bounded operator.

To conclude, take a test function $\varphi(t)\in C_c^\infty(\R)$.
$$\displaylines{
\int_{\R}\varphi(t) \left\la w_{I_\Lambda, h,R} \left(  t\right),\frac{d}{dt}_{|t=0}a_3\circ \tilde \phi^1_t(x, \xi, \tau_h\eta(\xi))\right\ra dt \hfill\cr\hfill
= \int_{\R}\varphi(t)
\left\la S_h^{\tau_h t} u_h,  \frac{i \tau_h}{h}\left[H(hD_x)+h^2 {\mathbf V}_h(t), \Op_h\left( a_{3}\left(  x ,\xi,\tau_{h}\eta (\xi)\right)\right)\right] S_h^{\tau_h t} u_h\right\ra dt \cr\hfill
  +O\left(h\tau_h\right)+ O\left(\frac{R^2}{\tau_h}\right)
\qquad\cr\hfill = \int_{\R}\varphi(t)
\frac{d}{dt}\left\la S_h^{\tau_h t} u_h,   \Op_h\left( a_{3}\left(  x ,\xi,\tau_{h}\eta (\xi)\right)\right) S_h^{\tau_h t} u_h\right\ra dt  +O\left(h\tau_h\right)+ O\left(\frac{R^2}{\tau_h}\right)\cr\hfill
=  -\int_{\R}\varphi'(t)
 \left\la S_h^{\tau_h t} u_h,  \Op_h\left( a_{3}\left(  x ,\xi,\tau_{h}\eta (\xi)\right)\right) S_h^{\tau_h t} u_h\right\ra dt  +O\left(h\tau_h\right)+ O\left(\frac{R^2}{\tau_h}\right)
\cr}$$
Taking the limit $h\To 0$ followed by $R\To +\infty$, we obtain
$$\int_{\R}\varphi(t) \left\la  \tilde\mu_\Lambda(t), \frac{d}{dt}|_{ t=0}\left(  a\circ \tilde \phi^1_t\right)    \right\ra dt =-\int_{\R}\varphi'(t) \left\la  \tilde\mu_\Lambda(t),    a   \right\ra dt $$
for any $\varphi$ and $a$, which is the announced result.\end{proof}

%%%%%%%%%%%%%%%%%%%%%%%%%%%%%%%%%%%%%%%%%%%%%%%%%%%%%

\section{Regularity and transport of~$ \tilde\mu_\Lambda$.}\label{s:finitedistance}

In this section, we suppose $\tau _{h}=1/h$ and we prove statement (3) of Theorem \ref{Thm
Properties}. This constitutes a first
step towards the proof of Theorem \ref{t:precise} (and Theorem \ref{t:main}(2)) which will be achieved in Section~\ref{s:successive}.
In Theorem \ref{prop:opvame} below, we give a description of the
measure~$\tilde{\mu}_{\Lambda }$. The
first part of our result implies in particular that the projection of $%
\tilde{\mu}_{\Lambda }$ onto $\mathbb{T}^{d}$ is absolutely
continuous. For this result to hold we only assume that
$\mathbf{V}_h(t)$ is a general bounded self-adjoint perturbation as
described in the Introduction. The second part of Theorem
\ref{prop:opvame} shows that $\tilde{\mu}_{\Lambda }$ satisfies a
propagation law that involves a Heisenberg equation. For that part
we need to assume that $\mathbf{V}_{h}\left( t\right)
=\Op\nolimits_{h}\left( V\left( t,\cdot \right) \right) $ for some
smooth bounded symbol $V$.
\medskip

Recall that for $\omega $ in the torus $\la\Lambda %
\ra/\Lambda $, we denote by $L_{\omega }^{2}(\R^{d},\Lambda )$ the
subspace of $L_{\text{loc}}^{2}(\R^{d})\cap \mathcal{S}^{\prime
}(\mathbb{R}^{d})$ formed by the functions whose Fourier transform
is supported in $\Lambda -\omega $. In other words, $f\in
L_{\omega }^{2}(\R^{d},\Lambda )$ if and
only if $f\in L_{\text{loc}}^{2}(\R^{d})$ and:%
\begin{equation*}
f\left( \cdot +v\right) =f,\quad \forall v\in \Lambda ^{\perp },
\end{equation*}%
\begin{equation}
f\left( \cdot +k\right) =e^{-i\omega \cdot k}f,\quad \forall k\in
2\pi \mathbb{Z}^{d},  \label{e:blochf}
\end{equation}%
where, recall, $\Lambda ^{\perp }$ stands for the orthogonal of
$\Lambda $ in the duality sense. Clearly, $f\in L_{\omega
}^{2}(\R^{d},\Lambda )$ if
and only if there exists $g\in L^{2}(\mathbb{T}^{d},\Lambda )$ (the set of $L^2$ function on~$\mathbb{T}^d$ which have Fourier modes in $\Lambda$) such that $%
f\left( x\right) =e^{-i\omega \cdot x}g\left( x\right) $; this
characterization induces a natural Hilbert structure on $L_{\omega }^{2}(\R%
^{d},\Lambda )$.\medskip

We introduce an auxiliary lattice $\tilde{\Lambda}\subset 2\pi
\Z^{d}$ such
that $2\pi \Z^{d}\subset \Lambda ^{\perp }\oplus \tilde{\Lambda}$. Let $D_{%
\tilde{\Lambda}}\subset \la\tilde{\Lambda}\ra$ be a fundamental
domain of
the action of $\tilde{\Lambda}$ on $\la\tilde{\Lambda}\ra$. Each space $%
L_{\omega }^{2}(\R^{d},\Lambda )$ is naturally isomorphic to $L^{2}(D_{%
\tilde{\Lambda}})$, (simply by extending by continuity the
restriction of functions in $C^{\infty }(\R^{d})\cap L_{\omega
}^{2}(\R^{d},\Lambda )$). Under this isomorphism, the norm in
$L^{2}(D_{\tilde{\Lambda}})$ equals a factor $\left\vert
D_{\tilde{\Lambda}}\right\vert ^{1/2}/\left( 2\pi \right) ^{d/2}$
times the norm in $L_{\omega }^{2}(\R^{d},\Lambda )$.\medskip

Introduce a vector bundle $\mathfrak{F}$ over $(\la\Lambda
\ra/\Lambda
)\times I_{\Lambda }$, formed of pairs $(\omega ,\sigma ,f)$ where $%
(\omega ,\sigma )\in (\la\Lambda \ra/\Lambda )\times I_{\Lambda }$
and $f\in L_{\omega }^{2}(\R^{d},\Lambda )$. This vector bundle is
trivial, it may be
identified with $(\la\Lambda \ra/\Lambda )\times I_{\Lambda }\times L^{2}(D_{%
\tilde{\Lambda}})$ just by taking the above isomorphism (note,
however, that the subbundle formed of triples $(\omega ,\sigma ,f)$
such that $f$ is smooth does not admit a smooth
trivialisation).\medskip

We denote by $\cL(\mathfrak{F)}$ (resp. $\cK(\mathfrak{F})$, $\cL^{1}(%
\mathfrak{F})$) the vector bundles over $(\la\Lambda \ra/\Lambda
)\times
I_{\Lambda }$ formed of pairs $(\omega ,\sigma ,Q)$ where $Q\in \cL(L^{2}(\R%
^{d},\Lambda ,\omega ))$ (resp. $\cK(L^{2}(\R^{d},\Lambda ,\omega ))$, $\cL%
^{1}(L^{2}(\R^{d},\Lambda ,\omega ))$). Again, all these bundles
are trivial. Recall that, given a Hilbert space ${\mathcal H}$,
${\mathcal{L}}({\mathcal H})$, ${\mathcal{K}}({\mathcal H})$ and $\mathcal{L}^{1}\left(
{\mathcal H}\right) $ stands respectively for the space of bounded, compact
and trace-class operators acting on ${\mathcal H}$.

\begin{remark}
The Bloch-Floquet spectral decomposition shows that $L^{\infty }$
sections of $\cL(\mathfrak{F)}$ are in one-to-one correspondence
with $L^{\infty }$ maps $I_{\Lambda }\ni \sigma \longmapsto
Q(\sigma )\in \cL(L^{2}(\R^{d}))$, where, in addition, a.e.
$Q(\sigma )$ commutes with all translations by vectors $k\in 2\pi
\Z^{d}+\Lambda ^{\perp }$.
\end{remark}

We denote by $\Gamma (\cK(\mathfrak{F}))$ the space of continuous
sections of $\cK(\mathfrak{F)}$. Using the previous
trivialisation, it is isomorphic
to $\cC((\la\Lambda \ra/\Lambda )\times I_{\Lambda },\cK(L^{2}(D_{\tilde{%
\Lambda}})))$, the space of continuous functions on $(\la\Lambda \ra/\Lambda )\times I_{\Lambda }$ taking values in $\cK(L^{2}(D_{\tilde{
\Lambda}}))$. The dual space $\Gamma (\cK(\mathfrak{F}))^{\prime }$ to $
\Gamma (\cK(\mathfrak{F))}$ is isomorphic to $\cM((\la\Lambda
\ra/\Lambda )\times I_{\Lambda
},\cL^{1}(L^{2}(D_{\tilde{\Lambda}})))$, the space of measures on $(\la\Lambda
\ra/\Lambda )\times I_{\Lambda
}$ taking values in $\cL^{1}(L^{2}(D_{\tilde{\Lambda}}))$.

\medskip

We denote by $\Gamma _{+}(\cK(\mathfrak{F}))$ the subset of
positive sections, and $\Gamma (\cK(\mathfrak{F}))_{+}^{\prime }$
the positive
elements of the dual, which correspond to elements of $\cM_{+}((\la\Lambda %
\ra/\Lambda )\times I_{\Lambda
},\cL^{1}(L^{2}(D_{\tilde{\Lambda}})))$ (the space of measures taking values in positive trace-class operators). Note that, by the
Radon-Nikodym theorem (see for instance the appendix in
\cite{GerardMDM91}), an element $\rho \in \Gamma
_{+}(\cK(\mathfrak{F}))$ can be written as $\rho =M\nu $ where
$\nu =\mathrm{Tr}\,\rho $ is a positive
measure on $(\la\Lambda \ra/\Lambda )\times I_{\Lambda }$ and $M$ is a $\nu $%
-integrable section of $\cL^{1}(\mathfrak{F)}$. We shall denote by
$\Gamma ^{1}(\cL^{1}(\mathfrak{F)};\nu )\mathfrak{\ }$the set of
such sections.\medskip

In order to state the propagation law obeyed by
$\tilde{\mu}_{\Lambda }$ when
$\mathbf{V}_{h}(t)=\Op\nolimits_{h}\left( V(t,x,\xi)\right) $, let us
introduce one more notation. Write $x=s+y\in \mathbb{R}^{d}$ with
$(s,y)\in \Lambda ^{\perp }\times \la\tilde{\Lambda}\ra$ and let
$\sigma \in I_{\Lambda }$; we denote by $\la V(t,y,\sigma
)\ra_{\Lambda }$ the average of $V(t,s+y,\sigma ) $ w.r.t.~$s$,
thus getting a function that does not depend on $s$. We
denote by $\la V(t)\ra_{\Lambda ,\sigma }$ the multiplication operator on $%
L_{\omega }^{2}(\R^{d},\Lambda )$ associated to the multiplication by the $%
\sigma $-depending function $\la V(t,y,\sigma )\ra_{\Lambda }$.
In the trivialisation introduced above, this operator does not depend on $%
\omega $.
In the case of a function $a(x,\xi, \eta )\in
\cS^1_\Lambda$, the function $%
a(s+y,\xi , \eta)$ does not depend on~$s$ so that $\la a(y,\sigma, \eta
)\ra_{\Lambda }=a(y,\sigma , \eta)$. We denote by $a_{\sigma }$ the
section of $\cL(\mathfrak{F)} $ that associates to $(\omega
,\sigma )$ the operator acting on $L_{\omega
}^{2}(\R^{d},\Lambda )$ by $a(y,\sigma, D_y )$ (Weyl quantization).
 In the case of a function $a(x,\xi )\in
{\mathcal{C}}_{0}^{\infty
}(T^{\ast }\T^{d})$ with Fourier modes in $\Lambda $ (and independent of $\eta$), the function $%
a(s+y,\xi )$ does not depend on~$s$ so that $\la a(y,\sigma
)\ra_{\Lambda }=a(y,\sigma )$. In this case $a_{\sigma }$ is the
section of $\cL(\mathfrak{F)} $ that associates to $(\omega
,\sigma )$ the operator acting on $L_{\omega
}^{2}(\R^{d},\Lambda )$ by multiplication by $a(y,\sigma )$. Finally, $%
(d^{2}H(\sigma )D_{y}\cdot D_{y})_{\omega }$ will be used to
denote the
operator $d^{2}H(\sigma )D_{y}\cdot D_{y}$ acting on $L_{\omega }^{2}(\R%
^{d},\Lambda )$.

\begin{theorem}
\label{prop:opvame}There exists $m_{\Lambda }\in \mathcal{M}_{+}((\la\Lambda %
\ra/\Lambda )\times I_{\Lambda })$ and $M_{\Lambda }^{0}$\emph{\ }a $%
m_{\Lambda }$-integrable section of $\cL^{1}(\mathfrak{F})$, which
only
depend on the sequence of initial data, such that for all $a\in \cS_\Lambda^1$ and all $%
t\in \mathbb{R}$:
\begin{equation}
\langle \tilde{\mu}_{\Lambda }\left( t,\cdot \right) ,a\rangle =\int_{(\la%
\Lambda \ra/\Lambda )\times I_{\Lambda }}\mathrm{Tr}_{L_{\omega }^{2}(%
\mathbb{R}^{d},\Lambda )}\left( a_{\sigma }M_{\Lambda }\left(
t,\omega ,\sigma \right) \right)m_{\Lambda }(d\omega ,d\sigma ) .
\label{e:eqmul}
\end{equation}%
where $M_{\Lambda }(t,\omega ,\sigma )$ solves,
$m_{\Lambda }$-a.e. $(\omega ,\sigma )\in (\la\Lambda \ra/\Lambda
)\times R_{\Lambda }$,
\begin{equation}\tag{{\rm Heis}$_{\Lambda, \omega, \sigma}$}
i\partial _{t}M_{\Lambda }(t,\omega ,\sigma )=\left[
\frac{1}{2}\left( d^{2}H(\sigma )D_{y}\cdot D_{y}\right) +\la V(t)\ra_{\Lambda ,\sigma },M_{\Lambda }(t,\omega ,\sigma
)\right] ,  \label{eq:SchroM(t)}
\end{equation}%
with $M_{\Lambda }(0,\cdot )=M_{\Lambda }^{0}$. In other words,%
\begin{equation}
M_{\Lambda }(t,\omega ,\sigma )=U_{\Lambda ,\omega ,\sigma }\left(
t\right) M_{\Lambda }^{0}(\omega ,\sigma )U_{\Lambda ,\omega
,\sigma }^{\ast }\left( t\right) ,  \label{e:propML}
\end{equation}%
where $U_{\Lambda ,\omega ,\sigma }\left( t\right) $ is the
propagator starting at $t=0$ of the unitary evolution associated
to the operator $\frac{1}{2}\left( d^{2}H(\sigma )D_{y}\cdot D_{y}\right)
+\la V(t)\ra_{\Lambda ,\sigma }$ on $L^2_\omega(\R^d,\Lambda)$.
\end{theorem}

\begin{remark}
i) Note that the fact that $M_{\Lambda }\left( t,\cdot \right) $
is given by (\ref{eq:SchroM(t)}) and (\ref{e:propML}) implies that
it is a continuous function of $t$. Therefore,
$\tilde{\mu}_{\Lambda }\left( t,\cdot \right) $
itself can be identified to a family of positive measures depending continuously on time .%
\newline
ii) The proof of Theorem~\ref{prop:opvame} is carried out
using the trivialisation obtained by identifying $L_{\omega
}^{2}(\R^{d},\Lambda )$ with $L^{2}(D_{\tilde{\Lambda}})$ and the
final result does not depend on the choice of $\tilde{\Lambda}$
and $D_{\tilde{\Lambda}}$.\newline iii) Identity (\ref{e:eqmul})
holds when $\mathbf{V}_h(t)$ is a bounded family of
perturbations as described in the introduction. In that case, the measure $%
m_{\Lambda }$ may also depend on time and equation
(\ref{e:propML}) is not available.
\end{remark}

The proof of Proposition~\ref{prop:opvame} is divided in three
steps:

\begin{enumerate}
\item We first define an operator $K_{h}$ which maps functions on
$\R^{d}$ to distributions with Fourier frequencies only in
$\la\Lambda \ra$; in addition, this operator maps $(2\pi
\Z^{d})$-periodic functions to
distributions on $I_{\Lambda }$ taking values in functions satisfying a Bloch-Floquet periodicity
condition.

\item Then, we express $w_{I_{\Lambda },h,R}(t)$ in terms of
$K_{h}$ and take limits, first $h\rightarrow 0^{+}$ then followed
by $R\rightarrow +\infty $. This defines an element $\rho
_{\Lambda }\in L^{\infty }\left( \mathbb{R};\Gamma
_{+}(\cK(\mathfrak{F}))\right) $.

\item We study the dependence in $t$ of the limit object $\rho
_{\Lambda }$
and show that it obeys a Heisenberg equation similar to~(\ref{eq:SchroM(t)}%
). Note that the latter is of lower dimension than the original one %
\eqref{e:eq} as soon as $\mathrm{rk}\,\Lambda <d$.
\end{enumerate}

Each of the next subsections is devoted to one of the steps of the
proof.

\subsection{First Step: Construction of the operator $K_{h}$}

Take $m\in {\mathcal{C}}_{0}^{\infty }((\R^{d})^{\ast })$
supported in the ball $B(\xi _{0},\eps)\subset (\R^{d})^{\ast }$,
and identically equal to~$1$ on $B(\xi _{0},\eps/2)$. For $f$ a
tempered distribution, we let
\begin{equation*}
\mathcal{F}f(\xi ):=\int_{\mathbb{R}^{d}}f(x)e^{-i\xi \cdot x}\frac{dx}{%
\left( 2\pi \right) ^{d/2}}.
\end{equation*}%
In particular, if $f$ is a $2\pi \Z^{d}$-periodic function, we have $%
\mathcal{F}f=\sum_{k\in \Z^{d}}\widehat{f}(k)\delta _{k}$. In what
follows we shall denote $d_{\Lambda }:=\mathrm{rk}\,\Lambda $ and
$d_{\Lambda ^{\perp }}:=\mathrm{rk}\,\Lambda ^{\perp }$.

The operator $K_{h}$ maps a tempered distribution $f$ to a distribution on $%
I_{\Lambda }\times \mathbb{R}^{d}$ as follows:
\begin{align*}
K_{h}f(\sigma ,y)& :=\int_{x\in \R^{d}}f(x)\int_{\eta \in \la\Lambda \ra%
}m(\sigma )e^{{\frac{i}{h}}\eta \cdot y}e^{-\frac{i}{h}(\sigma
+\eta )\cdot
x}\frac{d\eta }{(2\pi h)^{d_{\Lambda }/2}}\frac{dx}{(2\pi h)^{d/2}} \\
& =\frac{m(\sigma )}{h^{d/2}}\int_{\la\Lambda \ra}\mathcal{F}f\left( {\frac{%
\sigma +\eta }{h}}\right) e^{{\frac{i}{h}}\eta \cdot y}\frac{d\eta
}{(2\pi h)^{d_{\Lambda }/2}}.
\end{align*}%
In order to get more insight on the properties of $K_{h}f$ it is
useful to
introduce some notations. Let $\pi _{\tilde{\Lambda}}$ be the projection on $%
\la\tilde{\Lambda}\ra$, in the direction of~$\Lambda ^{\perp }$. We have $%
\pi _{\tilde{\Lambda}}(2\pi \Z^{d})=\tilde{\Lambda}\subset
\Z^{d}$. For $\xi \in (\R^{d})^{\ast }$, we shall denote by $\xi
_{\Lambda }\in \la\Lambda \ra$ the linear form $\xi _{\Lambda
}(y):=\xi \cdot \pi _{\tilde{\Lambda}}(y)$ (in other words, the
projection of $\xi $ on~$\la\Lambda \ra$, in the direction
$\tilde{\Lambda}^{\perp }$). Note that for $\xi \in
\mathbb{Z}^{d}$
one has $\xi _{\Lambda }\in \Lambda $. We fix a bounded fundamental domain $%
D_{\Lambda }$ for the action of $\Lambda $ on $\la\Lambda \ra$.
For~$\eta \in \la\Lambda \ra$, there is a unique $\{\eta \}\in
D_{\Lambda }$ (the
\textquotedblleft fractional part\textquotedblright\ of $\eta $) such that $%
\eta -\{\eta \}\in \Lambda $. \medskip

Sometimes we shall use the decomposition $(\mathbb{R}^{d})^{\ast }=\tilde{%
\Lambda}^{\perp }\oplus \left\langle \Lambda \right\rangle $. This
decomposition is related to the one given by the local coordinate system $%
F\left( \xi \right) =\left( \sigma \left( \xi \right) ,\eta \left(
\xi \right) \right) $ defined in~(\ref{e:defF}) as follows. Let
$\left( \xi _{1},\xi _{2}\right) \in \tilde{\Lambda}^{\perp
}\times \left\langle \Lambda \right\rangle $ such that $F$ is
defined on $\xi =\xi _{1}+\xi _{2}$ (and
therefore $\xi =\sigma \left( \xi \right) +\eta \left( \xi \right) $). Then $%
\sigma \left( \xi \right) =\sigma \left( \xi _{1}\right) $ does
not depend on $\xi _{2}$ and
\begin{equation*}
\xi _{2}=\eta \left( \xi \right) +\sigma \left( \xi _{1}\right)
_{\Lambda }.
\end{equation*}%
In other words, $\left( \xi _{1},\xi _{2}\right) \in
\tilde{\Lambda}^{\perp }\times \left\langle \Lambda \right\rangle
$ corresponds through $F$ (whenever $F\left( \xi \right) $ is
defined) to a pair $\left( \sigma ,\eta
\right) \in I_{\Lambda }\times \left\langle \Lambda \right\rangle $ given by:%
\begin{equation}
\sigma =\sigma \left( \xi _{1}\right) ,\qquad \eta =\xi
_{2}-\sigma \left( \xi _{1}\right) _{\Lambda }.
\label{e:changeco}
\end{equation}%
These relations imply that for every $\xi _{1}\in
\tilde{\Lambda}^{\perp }$
and $y\in \mathbb{R}^{d}$ the following holds:%
\begin{equation}
K_{h}f(\sigma \left( \xi _{1}\right) ,y)=e^{-i\frac{\sigma \left(
\xi _{1}\right) _{\Lambda }}{h}\cdot y}m(\sigma (\xi
_{1}))\int_{\Lambda ^{\perp
}}f(s+\pi _{\Lambda }\left( y\right) )e^{-i\frac{\xi _{1}}{h}\cdot s}\frac{ds%
}{(2\pi h)^{d_{\Lambda ^{\perp }}/2}}.  \label{e:partialFT}
\end{equation}%
In the above formula, $ds$ is the dual density of $d\xi _{1}$,
which in turn is defined to have $d\xi =d\xi _{1}d\xi _{2}$ where
$d\xi _{2}$ stands for
the natural density on $\left\langle \Lambda \right\rangle $. Note that $%
d\xi _{1}$ is a constant multiple of the natural density on $\tilde{\Lambda}%
^{\perp }$.

\medskip

If $f$ is a $2\pi \Z^{d}$-periodic function then:
\begin{equation*}
K_{h}f(\sigma ,y)=\frac{h^{d_{\Lambda ^{\perp }}/2}}{\left( 2\pi
\right) ^{d_{\Lambda }/2}}\sum_{k_{\sigma }\in \sigma
(h\Z^{d})}\delta _{k_{\sigma }}\left( \sigma \right) \sum_{k_{\eta
}\in \la\Lambda \ra,\,\left( k_{\sigma
},k_{\eta }\right) \in F\left( h\Z^{d}\right) }m(k_{\sigma })\widehat{f}%
\left( {\frac{k_{\sigma }+k_{\eta }}{h}}\right)
e^{{\frac{i}{h}}k_{\eta }\cdot y}.
\end{equation*}%
It is clear from the above formula that for every $y\in
\mathbb{R}^{d}$, the
distribution $K_{h}f\left( \cdot ,y\right) $ is supported on the set%
\begin{equation*}
I_{\Lambda }^{h}:=\left\{ \sigma \in I_{\Lambda }~:~\frac{\sigma
}{h}\in \mathbb{Z}^{d}-\left\langle \Lambda \right\rangle \right\}
.
\end{equation*}%
We gather the properties of the operator $K_{h}$ that will be used
in the sequel in the following lemma.

\begin{lemma}
\label{l:Kh}(i) The Fourier transform of $K_{h}f(\sigma ,\cdot )$
w.r.t. the
second variable is:%
\begin{equation}
\mathcal{F}K_{h}f\left( \sigma ,\eta \right) =\left( \frac{2\pi
}{h}\right)
^{d_{\Lambda ^{\perp }}/2}m\left( \sigma \right) \mathcal{F}u\left( \frac{%
\sigma }{h}+\eta \right) \delta _{\left\langle \Lambda
\right\rangle }\left( \eta \right) ,  \label{e:ftkh}
\end{equation}%
in particular it is supported in $\la\Lambda \ra$.\\
 (ii) The
support of $K_{h}f(\sigma ,\cdot )$ is included in $\supp
f+\Lambda ^{\perp }$.\newline (iii) If $f$ is $(2\pi
\Z)^{d}$-periodic, then $\supp K_{h}f\left( \cdot ,y\right)
\subset I_{\Lambda }^{h}$ and $K_{h}f(\sigma ,\cdot )$ satisfies a
Floquet periodicity condition:
\begin{equation}
K_{h}f(\sigma ,y+k)=K_{h}f(\sigma ,y)e^{-i\omega _{h}\left( \sigma
\right) \cdot k}  \label{e:BF}
\end{equation}%
for all $k\in 2\pi \Z^{d}$, where
\begin{equation*}
\omega _{h}:I_{\Lambda }^{h}\longrightarrow \left\langle \Lambda
\right\rangle /\Lambda :\sigma \longmapsto \left\{ \frac{\sigma _{\Lambda }}{%
h}\right\} .
\end{equation*}%
Statement (\ref{e:BF}) is equivalent to the fact that
$K_{h}f(\sigma ,\cdot ) $ has only frequencies in $\Lambda -\omega
_{h}\left( \sigma \right) $.\\
(iv) Let $f$ be a $2\pi \Z^{d}$-periodic function, and let $\chi $
be a
compactly supported function on $\R^{d}$ such that $\sum_{k\in 2\pi \Z%
^{d}}\chi (\cdot +k)\equiv 1$. Then
\begin{equation}
\sum_{k\in \tilde{\Lambda}}K_{h}(\chi f)(\sigma ,y+k)e^{i\omega
_{h}\left( \sigma \right) \cdot k}=K_{h}f(\sigma ,y).
\label{e:periodization}
\end{equation}
(v) If $f\in L^{2}(\mathbb{T}^{d})$ then the following
Plancherel-type formula holds:
\begin{equation}
\sum_{k\in \Z^{d}}|\hat{f}(k)m(hk)|^{2}=\sum_{\sigma \in
I_{\Lambda }^{h}}\int_{\T^{d}}|K_{h}f(\sigma ,y)|^{2}dy.
\label{e:planche}
\end{equation}
\end{lemma}

\begin{proof}
All points are quite obvious except (iii), which we prove below. Formula (%
\ref{e:ftkh}) shows that the Fourier transform of $K_{h}f(\sigma
,\cdot )$
is supported in $\la\Lambda \ra$. On the other hand, if $f$ is $2\pi \Z^{d}$%
-periodic then its Fourier transform is supported in $\Z^{d}$.
Therefore,
because of identity (\ref{e:ftkh}), on the support of $\mathcal{F}%
K_{h}f(\sigma ,\eta )$ one must have:%
\begin{equation*}
\frac{\sigma }{h}+\eta \in \mathbb{Z}^{d},\qquad \eta \in
\la\Lambda \ra.
\end{equation*}%
In other words, $\sigma \in I_{\Lambda }^{h}$ and $\eta \in
\la\Lambda \ra$.
Taking the projection on $\la\Lambda \ra$ in the direction $\tilde{\Lambda}%
^{\perp }$ yields $\eta \in \Lambda -\frac{\sigma _{\Lambda
}}{h}$, which is equivalent to $\eta \in \Lambda -\{\frac{\sigma
_{\Lambda }}{h}\}$. Note that any other choice of the auxiliary
lattice used to define the projection onto $\left\langle \Lambda
\right\rangle $ would lead to a $\sigma _{\Lambda
}^{\prime }\in \left\langle \Lambda \right\rangle $ that differs from $%
\sigma _{\Lambda }$ on an element of $h\Lambda $. This shows, in
particular, that the mapping $\omega _{h}$ is well-defined on
$I_{\Lambda }^{h}$.
\end{proof}

\begin{remark}
\label{r:floq}Let $f$ be $2\pi \Z^{d}$-periodic and let $\theta \in (\mathbb{%
R}^{d})^{\ast }/(\mathbb{Z}^{d})^*$. Let $g_{\theta }\left( y\right)
:=e^{-i\theta \cdot y}f\left( y\right) $; then the proof of Lemma
\ref{l:Kh} (iii) shows that $K_{h}g_{\theta }$ satisfies the
following Bloch-Floquet periodicity condition:
\begin{equation*}
K_{h}g_{\theta }(\sigma ,y+k)=K_{h}g_{\theta }(\sigma
,y)e^{-i\left( \omega _{h}\left( \sigma \right) +\theta _{\Lambda
}\right) \cdot k},
\end{equation*}%
for every $k\in 2\pi \Z^{d}$.
\end{remark}

\subsection{Second step: Link between $w_{I_{\Lambda },h,R}$ and $K_{h}$}

Now we show how the two-microlocal Wigner distributions $\left(
w_{I_{\Lambda },h,R}\right) $ of the sequence $\left(
S_{h}^{t/h}u_{h}\right) $ can be expressed in terms of $\left(
K_{h}S_{h}^{t/h}u_{h}\right) $. Let $\chi \in \mathcal{C}_{c}^{\infty }(\R%
^{d})$ be such that $\sum_{k\in 2\pi \Z^{d}}\chi (\cdot +k)\equiv
1$. All the following identities hold independently of the choice
of such $\chi $. We start expressing the standard Wigner
distributions $w_{u_{h}}^{h}$ in
terms of the decomposition $\xi =(\xi _{1},\xi _{2})\in \tilde{\Lambda}%
^{\perp }\times \la\Lambda \ra$ of $(\mathbb{R}^{d})^{\ast }$. Let $b\in {%
\mathcal{C}}_{c}^{\infty }(\T^{d}\times \R^{d})$, possibly
depending on~$h$. The following holds:
\begin{multline*}
\cI(b, h):=\int_{T^{\ast }\mathbb{T}^{d}}b\left( x,\xi \right)
w_{u_{h}}^{h}(dx,d\xi ) \\
=\frac{1}{(2\pi )^{d/2}}\int_{(\R^{d})^{\ast }\times (\R^{d})^{\ast }}%
\mathcal{F}\left( \chi u_{h}\right) (\xi
)\overline{\mathcal{F}u_{h}(\xi ^{\prime })}\mathcal{F}b\left( \xi
^{\prime }-\xi ,h{\frac{\xi +\xi ^{\prime
}}{2}}\right) d\xi d\xi ^{\prime } \\
=\frac{1}{(2\pi )^{d/2}h^{d_{\Lambda ^{\perp }}}}\int_{\xi
_{1},\xi
_{1}^{\prime }\in \tilde{\Lambda}^{\perp },\xi _{2},\xi _{2}^{\prime }\in \la%
\Lambda \ra}\mathcal{F}\left( \chi u_{h}\right) (\xi _{1},\xi _{2})\overline{%
\mathcal{F}u_{h}(\xi _{1}^{\prime },\xi _{2}^{\prime })} \\
\mathcal{F}b\left( \frac{\xi _{1}^{\prime }-\xi _{1}}{h}+\xi
_{2}^{\prime }-\xi _{2},{\frac{\xi _{1}+\xi _{1}^{\prime
}}{2}}+h{\frac{\xi _{2}+\xi _{2}^{\prime }}{2}}\right) d\xi
_{1}d\xi _{1}^{\prime }d\xi _{2}d\xi _{2}^{\prime }.
\end{multline*}%
Identity (\ref{e:changeco}) shows that:%
\begin{equation*}
\xi =\frac{\xi _{1}}{h}+\xi _{2}=\frac{\sigma }{h}+\eta ,
\end{equation*}%
where $\sigma =\sigma (\xi _{1})$ and $\eta =\xi _{2}-\frac{\sigma
\left( \xi _{1}\right) _{\Lambda }}{h}$. For every $\sigma \in
I_{\Lambda }$, we
denote by $\xi _{1}\left( \sigma \right) $ the element in $\tilde{\Lambda}%
^{\perp }$ characterised by $\xi _{1}\left( \sigma \right) -\sigma \in \la%
\Lambda \ra$. The density $d\xi _{1}$ on $\tilde{\Lambda}^{\perp
}$ is transferred to a density $d\sigma $ on $I_{\Lambda }$ (note
that $d\xi =d\xi _{1}d\xi _{2}=d\sigma d\eta $).

\medskip

With this in mind, we obtain using (\ref{e:ftkh}), provided
we assume that $u_{h}$ has frequencies in $B(\xi _{0},\eps/2)$:
\begin{multline}
\cI(b, h)=\frac{(2\pi h)^{-d_{\Lambda ^{\perp }}}}{(2\pi
)^{d/2}}\int_{\sigma
,\sigma ^{\prime }\in I_{\Lambda },\eta ,\eta ^{\prime }\in \R^{d}}\mathcal{F%
}K_{h}\chi u_{h}(\sigma ,\eta
)\overline{\mathcal{F}K_{h}u_{h}(\sigma
^{\prime },\eta ^{\prime })}  \label{e:transfo} \\
\mathcal{F}b\left( \frac{\xi _{1}(\sigma ^{\prime })-\xi
_{1}\left( \sigma
\right) }{h}+\frac{\sigma _{\Lambda }^{\prime }}{h}+\eta ^{\prime }-\frac{%
\sigma _{\Lambda }}{h}-\eta ,{\frac{\sigma \oplus \sigma ^{\prime }}{2}}%
-\left( {\frac{\sigma \oplus \sigma ^{\prime }}{2}}\right) _{\Lambda }+{%
\frac{\sigma _{\Lambda }+\sigma _{\Lambda }^{\prime
}}{2}}+h{\frac{\eta
+\eta ^{\prime }}{2}}\right) \\
d\sigma d\sigma ^{\prime }d\eta d\eta ^{\prime },
\end{multline}%
where ${\frac{\sigma \oplus \sigma ^{\prime }}{2}}$ is a notation
for the
image under the coordinate map $\sigma \left( \xi \right) $ of ${\frac{%
\sigma +\sigma ^{\prime }}{2}}$.\medskip

If $b$ is invariant in the direction $\Lambda ^{\perp }$ then the
previous integral reduces to an integral over $\sigma =\sigma
^{\prime }$ and takes the simpler form
\begin{equation*}
\cI(b, h)=\frac{(2\pi h)^{-d_{\Lambda ^{\perp }}}}{(2\pi )^{d_{\Lambda }/2}}%
\int_{\sigma \in I_{\Lambda },\eta ,\eta ^{\prime }\in \R^{d}}\mathcal{F}%
K_{h}\chi u_{h}(\sigma ,\eta
)\overline{\mathcal{F}K_{h}u_{h}(\sigma ,\eta ^{\prime
})}\mathcal{F}b\left( \eta ^{\prime }-\eta ,\sigma +h{\frac{\eta
+\eta ^{\prime }}{2}}\right) d\sigma d\eta d\eta ^{\prime },
\end{equation*}%
where now $\mathcal{F}b$ should just be interpreted as a partial
Fourier transform in the direction~$\la\tilde{\Lambda}\ra$.\medskip

In \eqref{e:transfo}, the function $K_{h}\chi u_{h}$ may be
replaced by any other function satisfying the
identity~\eqref{e:periodization}. In particular, we may replace
$K_{h}\chi u_{h}$ by $\chi _{0}K_{h}u_{h}$, where $\chi _{0}$ is a
function that satisfies
\begin{equation*}
\sum_{k\in \tilde{\Lambda}}\chi _{0}(\cdot +k)\equiv 1
\end{equation*}%
and $\chi _{0}$ is constant in the direction $\Lambda ^{\perp }$.
In what
follows, we take $\chi _{0}$ to be the characteristic functions of $D_{%
\tilde{\Lambda}}$, a fundamental domain for the action of
$\tilde{\Lambda}$ on $\R^{d}$.\medskip

By the changes of variables $\eta \To \eta-\frac{\sigma _{\Lambda }}{h}+\frac{1%
}{h}\left( {\frac{\sigma \oplus \sigma ^{\prime }}{2}}\right)
_{\Lambda
}$ and $\eta' \To \eta'-\frac{\sigma' _{\Lambda }}{h}+\frac{1%
}{h}\left( {\frac{\sigma \oplus \sigma ^{\prime }}{2}}\right)
_{\Lambda
}$, \eqref{e:transfo} may be transformed into
\begin{multline}
\cI(b, h)=\frac{(2\pi h)^{-d_{\Lambda ^{\perp }}}}{(2\pi
)^{d/2}}\int_{\sigma
,\sigma ^{\prime }\in I_{\Lambda },\eta ,\eta ^{\prime }\in \R^{d}}\mathcal{F%
}\chi _{0}K_{h}u_{h}\left( \sigma ,\eta -\frac{\sigma _{\Lambda }}{h}+\frac{1%
}{h}\left( {\frac{\sigma \oplus \sigma ^{\prime }}{2}}\right)
_{\Lambda
}\right)  \label{e:transfo2} \\
\times \,\overline{\mathcal{F}K_{h}u_{h}\left( \sigma ^{\prime
},\eta
^{\prime }-\frac{\sigma _{\Lambda }^{^{\prime }}}{h}+\frac{1}{h}\left( {%
\frac{\sigma \oplus \sigma ^{\prime }}{2}}\right) _{\Lambda }\right) } \\
\times \mathcal{F}b\left( \frac{\xi _{1}(\sigma ^{\prime })-\xi
_{1}\left( \sigma \right) }{h}+\eta ^{\prime }-\eta ,{\frac{\sigma
\oplus \sigma ^{\prime }}{2}}+h{\frac{\eta +\eta ^{\prime
}}{2}}\right) d\sigma d\sigma ^{\prime }d\eta d\eta ^{\prime }.
\end{multline}%
Next we write
\begin{equation*}
K_{h}u_{h}\left( \sigma ,y\right) =\sum_{k\in \tilde{\Lambda}}\chi
_{0}(y+k)K_{h}u_{h}\left( \sigma ,y\right) =\sum_{k\in
\tilde{\Lambda}}(\chi _{0}K_{h}u_{h})(\sigma ,y+k)e^{i\omega
_{h}\left( \sigma \right) \cdot k},
\end{equation*}%
so that
\begin{equation*}
\mathcal{F}K_{h}u_{h}(\sigma ,\eta )=\mathcal{F}\chi
_{0}K_{h}u_{h}(\sigma ,\eta )\delta _{\Lambda -\omega _{h}(\sigma
)+\tilde{\Lambda}^{\perp }}\left( \eta \right) ,
\end{equation*}%
Thus \eqref{e:transfo2} is also
\begin{multline}
\cI(b, h)=\frac{(2\pi h)^{-d_{\Lambda ^{\perp }}}}{(2\pi
)^{d/2}}\int_{\sigma
,\sigma ^{\prime }\in I_{\Lambda },\eta ,\eta ^{\prime }\in \R^{d}}\mathcal{F%
}\chi _{0}K_{h}u_{h}\left( \sigma ,\eta -\frac{\sigma _{\Lambda }}{h}+\frac{1%
}{h}\left( {\frac{\sigma \oplus \sigma ^{\prime }}{2}}\right)
_{\Lambda
}\right)  \label{e:transfo3} \\
\overline{\mathcal{F}\chi _{0}K_{h}u_{h}\left( \sigma ^{\prime
},\eta
^{\prime }-\frac{\sigma _{\Lambda }^{^{\prime }}}{h}+\frac{1}{h}\left( {%
\frac{\sigma \oplus \sigma ^{\prime }}{2}}\right) _{\Lambda }\right) } \\
\mathcal{F}b\left( \frac{\xi _{1}(\sigma ^{\prime })-\xi
_{1}\left( \sigma
\right) }{h}+\eta ^{\prime }-\eta ,{\frac{\sigma \oplus \sigma ^{\prime }}{2}%
}+h{\frac{\eta +\eta ^{\prime }}{2}}\right) d\sigma d\sigma
^{\prime }d\eta d\eta ^{\prime }\delta _{\eta ^{\prime }\in
\tilde{\Lambda}^{\perp }+\omega _{h}({\frac{\sigma \oplus \sigma
^{\prime }}{2}})}.
\end{multline}%
This motivates the following definition~:

\begin{definition}\label{d:Psigma}
If $Q(\omega ,s,\sigma )$ is a smooth compactly supported function on $\la%
\Lambda \ra/\Lambda \times \Lambda ^{\perp }\times I_{\Lambda }$,
taking
values in $\cL(L^{2}(D_{\tilde{\Lambda}}))$, we define:%
\begin{equation*}
P_{Q}^{h}\left( s,\sigma \right) :=e^{i\frac{\sigma _{\Lambda }\bullet }{h}%
}Q\left( \omega _{h}\left( \sigma \right) ,s,\sigma \right) e^{-i\frac{%
\sigma _{\Lambda }\bullet }{h}},\quad v_{h}\left( \sigma ,y\right)
:=\chi _{0}\left( y\right) e^{i\frac{\sigma _{\Lambda }}{h}\cdot
y}K_{h}u_{h}\left( \sigma ,y\right) ,
\end{equation*}%
where $e^{i\frac{\sigma _{\Lambda }\bullet }{h}}$ denotes multiplication by $%
e^{i\frac{\sigma _{\Lambda }}{h}\cdot y}$. Define%
\begin{equation}
\la\rho _{u_{h}}^{h},Q\ra:=\left\langle v_{h},P_{Q}^{h}\left(
hD_{\sigma
},\cdot \right) v_{h}\right\rangle _{L^{2}\left( I_{\Lambda };L^{2}(D_{%
\tilde{\Lambda}})\right) },  \label{defop1}
\end{equation}%
where $P_{Q}^{h}\left( hD_{\sigma },\sigma \right) $ is obtained from $%
P_{Q}^{h}$ by Weyl quantization.
\end{definition}

More explicitly, we have:
\begin{multline}
\la\rho _{u_{h}}^{h},Q\ra=\frac{1}{{(2\pi h)^{d_{\Lambda ^{\perp }}}}}%
\int_{\sigma ,\sigma ^{\prime }\in I_{\Lambda }}\int_{s\in \Lambda
^{\perp }}e^{-i\frac{\xi _{1}(\sigma ^{\prime })-\xi _{1}(\sigma
)}{h}\cdot s}
\label{e:defop} \\
\left\langle v_{h}(\sigma ^{\prime },\cdot ),P_{Q}^{h}\left(
s,\frac{\sigma \oplus \sigma ^{\prime }}{2}\right) v_{h}(\sigma
^{\prime },\cdot )\right\rangle
_{L^{2}(D_{\tilde{\Lambda}})}d\sigma d\sigma ^{\prime }ds.
\end{multline}%
The interest of this definition becomes clearer if we realize the
following identity. Let $b\in \mathcal{C}_{c}^{\infty }(T^{\ast
}\mathbb{T}^{d})$, then formula \eqref{e:transfo3} is equivalent
to the identity
\begin{equation}
\la u_{h},\Op_{h}(b)u_{h}\ra_{L^{2}(\mathbb{T}^{d})}=\la\rho
^{h}_{u_h},Q_{b}^{h}\ra,  \label{identityQha}
\end{equation}%
where $Q_{b}^{h}(\omega ,s,\sigma )$ is the operator on $L^{2}(D_{\tilde{%
\Lambda}})$ given by the kernel
\begin{equation}
Q_{b}^{h}(\omega ,s,\sigma )(\tilde{y}^{\prime
},\tilde{y})=\frac{1}{(2\pi
)^{d_{\Lambda }}}\sum_{k\in \tilde{\Lambda}}e^{i\omega k}\int_{\eta \in \la%
\Lambda \ra}b\left( s+\frac{\tilde{y}+\tilde{y}^{\prime
}}{2},\sigma +h\eta \right) e^{i\eta \cdot (\tilde{y}^{\prime
}-\tilde{y}+k)}d\eta , \label{e:pseudo}
\end{equation}%
where, recall we have written $x=s+\tilde{y}\in \Lambda ^{\perp }\oplus \la%
\tilde{\Lambda}\ra$. Note that if one identifies
$L^{2}(D_{\tilde{\Lambda}})$ with $L_{\omega
}^{2}(\mathbb{R}^{d},\Lambda )$ as we did before, the
operator $Q_{b}^{h}(\omega ,s,\sigma )$ then corresponds to the ($\omega $%
-independent) Weyl pseudodifferential operator $b(y,\sigma
+hD_{y})$ acting on $L_{\omega }^{2}(\mathbb{R}^{d},\Lambda )$.
\medskip

Let us now consider $a\in \cS_{\Lambda }^{1}$ and let
$b_{h,R}\left( x,\xi \right) :=a(x,\xi ,\eta (\xi )/h)\chi (\eta
(\xi )/hR)$ . We then have
\begin{eqnarray}
\la S_{h}^{t/h}u_{h},\Op_{h}(b_{h})S_{h}^{t/h}u_{h}\ra_{L^{2}(\mathbb{T}%
^{d})} &=&\la w_{I_{\Lambda },h,R}\left( t\right) ,a\ra  \notag
\label{link}
\\
&=&\la\rho ^{h}(t),Q_{a,R}^{h}\ra+O_{R}(h)
\end{eqnarray}%
where $\rho ^{h}(t):=\rho _{S_{h}^{t/h}u_{h}}^{h}$ and $%
Q_{a,R}^{h}:=Q_{b_{h,R}}^{h}$. Note that $Q_{a,R}^{h}$ does not depend on $s$%
, since as $a$ has only frequencies in $\Lambda $ it is a function
independent on $s$.

\medskip

We now take limits as $h$ tends to zero~:

\begin{proposition}
After extraction of a subsequence, there exist
\begin{equation*}
\rho _{\Lambda }\in L^{\infty }\left( \R_{t},\cD^{\prime }\left( \la\Lambda %
\ra/\Lambda \times \Lambda ^{\perp }\times I_{\Lambda
};\cL^{1}\left( L^{2}\left( D_{\tilde{\Lambda}}\right) \right)
\right) \right)
\end{equation*}%
such that for every $Q\in \mathcal{C}_{c}^{\infty }\left( \la\Lambda \ra%
/\Lambda \times \Lambda ^{\perp }\times I_{\Lambda };\cK\left(
L^{2}\left( D_{\tilde{\Lambda}}\right) \right) \right) $ and every
$\phi \in L^{1}\left( \mathbb{R}\right) $:
\begin{equation*}
\int_{\mathbb{R}}\phi \left( t\right) \la\rho ^{h}(t),Q\ra dt\To\int_{%
\mathbb{R}}\phi \left( t\right) \la\rho _{\Lambda }(t),Q\ra dt.
\end{equation*}%
In addition, $\rho _{\Lambda }$ is positive when restricted to symbols $%
Q(\omega ,s,\sigma )$ that do not depend on $s$.
\end{proposition}

\begin{proof}
Note that Lemma \ref{l:Kh}, v) implies that $\left( v_{h}\right) $
is bounded in $L^{2}\left( I_{\Lambda
};L^{2}(D_{\tilde{\Lambda}})\right) $ and
that the Calder\'{o}n-Vaillancout theorem gives that the operators $%
P_{Q}^{h}\left( hD_{\sigma },\sigma \right) $ are uniformly
bounded with respect to $h$. Therefore, the linear map
\begin{equation*}
L_h: Q\mapsto \int_{\mathbb{R}}\la\rho ^{h}(t),Q\left(
t\right) \ra dt
\end{equation*}%
is uniformly bounded as ~$h\To 0$. Therefore, for any~$Q$, up to
extraction of a
subsequence, it has a limit~$l(Q)$. \\
Considering a countable dense subset of~$%
L^{1}\left( \mathbb{R};\mathcal{C}_{c}^{\infty }\left( \la\Lambda \ra%
/\Lambda \times \Lambda ^{\perp }\times I_{\Lambda };\cK\left(
L^{2}\left( D_{\tilde{\Lambda}}\right) \right) \right) \right) $,
and using a diagonal
extraction process, one finds a sequence~$\left( h_{n}\right) $ tending to~$%
0 $ as~$n$ goes to~$+\infty $ such that for any~$Q\in L^{1}\left( \mathbb{R};%
\mathcal{C}_{c}^{\infty }\left( \la\Lambda \ra/\Lambda \times
\Lambda
^{\perp }\times I_{\Lambda };\cK\left( L^{2}\left( D_{\tilde{\Lambda}%
}\right) \right) \right) \right) $, the
sequence~$L_{h_{n}}(Q)$ has a limit as~$n$ goes
to~$+\infty $.\\
 The limit is a linear form on $L^{1}\left( \mathbb{R};%
\mathcal{C}_{c}^{\infty }\left( \la\Lambda \ra/\Lambda \times
\Lambda
^{\perp }\times I_{\Lambda };\cK\left( L^{2}\left( D_{\tilde{\Lambda}%
}\right) \right) \right) \right) $,
characterized by an element $\rho _{\Lambda }$ of the dual
bundle~$L^{\infty }\left( \R_{t},\cD^{\prime }\left( \la\Lambda
\ra/\Lambda \times \Lambda
^{\perp }\times I_{\Lambda };\cL^{1}\left( L^{2}\left( D_{\tilde{\Lambda}%
}\right) \right) \right) \right) $. Finally, note that if
$Q(t,\omega ,\sigma )$ is a positive operator independent of $s$,
equation~\eqref{defop1} gives $L_{h}(Q)\geq 0$, whence the
positivity of $\rho _{\Lambda }$ when restricted to symbols that
do not depend on $s$.
\end{proof}

As a consequence and in view of~(\ref{link}), letting $h$ going to $0$, then~%
$R$ to~$+\infty $, we have (possibly along a subsequence) for every $a\in {%
\mathcal{S}}_{\Lambda }^{1}$ and $\phi \in L^{1}\left( \mathbb{R}\right) $:%
\begin{eqnarray*}
\lim_{R\rightarrow +\infty }\lim_{h\rightarrow
0^{+}}\,\int_{\mathbb{R}}\phi \left( t\right) \la w_{I_{\Lambda
},h,R}(t),a\ra dt &=&\lim_{R\rightarrow
+\infty }\lim_{h\rightarrow 0^{+}}\,\,\int_{\mathbb{R}}\phi \left( t\right) %
\la\rho ^{h}(t),Q_{a,R}\ra dt \\
&=&\lim_{R\rightarrow +\infty }\,\int_{\mathbb{R}}\phi \left( t\right) \la%
\rho _{\Lambda }(t),Q_{a,R}\ra dt \\
&=&\int_{\mathbb{R}}\phi \left( t\right) \la\rho _{\Lambda }(t),Q_{a,\infty }%
\ra dt,
\end{eqnarray*}%
where $Q_{a,\infty }(\omega ,s,\sigma )$ is the bounded operator on $%
L^{2}(D_{\tilde{\Lambda}})$ given by the kernel
\begin{equation}
Q_{a,\infty }(\omega ,s,\sigma )(\tilde{y}^{\prime },\tilde{y})=\frac{1}{%
(2\pi )^{d_{\Lambda }}}\sum_{k\in \tilde{\Lambda}}e^{i\omega
k}\int_{\eta\in\la\Lambda\ra} a\left( s+\frac{\tilde{y}+\tilde{y}^{\prime }}{2},\sigma
,\eta \right) e^{i\eta \cdot (\tilde{y}^{\prime
}-\tilde{y}+k)}d\eta .
\end{equation}%
As discussed before, the operator $Q_{a,\infty }$ corresponds to
the Weyl
operator $a(s+y,\sigma ,D_{y})$ acting on $L_{\omega }^{2}(\mathbb{R}^{d}\la%
\Lambda \ra)$. In particular, when $a\in {\mathcal{C}}_{c}^{\infty
}(\T^{d})$ has only frequences in $\Lambda $, the operator
$Q_{a,\infty }$ is the multiplication operator $a_{\sigma }$
appearing in identity (\ref{e:eqmul}) of
Theorem~\ref{prop:opvame}.

\medskip

At this stage of the analysis, we have
completed the proof of the first part of
Theorem~\ref{prop:opvame} (equation \eqref{e:eqmul}), using only the fact that $\mathbf{V}_h(t)$ is a bounded perturbation~:
we let $m_{\Lambda }\left( t,\omega ,s, \sigma \right)
:=\mathrm{Tr}_{L_{\omega }^{2}(\mathbb{R}^{d},\Lambda )}\rho
_{\Lambda }(t,\omega ,s, \sigma )$. We have $\rho _{\Lambda }=M_{\Lambda }m_{\Lambda }$ where $\mathrm{Tr}M_{\Lambda } =1$.

\medskip

As already noted, equation \eqref{e:eqmul} implies the absolute continuity result, Theorem \ref{t:main} (2).

\subsection{Step 3: Showing a propagation law}

From now on we shall assume $\mathbf{V}_{h}\left( t\right)
=\Op_{h}\left(
V\left( t,\cdot \right) \right) $ and prove the propagation law~(\ref%
{eq:SchroM(t)}). The first crucial observation is the following
lemma.

\begin{lemma}
\label{p:averaging} The measure $\rho _{\Lambda }$ is invariant by
the
Hamiltonian flow. More precisely, for every $Q\in \mathcal{C}_{c}^{\infty }(%
\la\Lambda \ra/\Lambda \times \Lambda ^{\perp }\times I_{\Lambda },\cK%
(L^{2}(D_{\tilde{\Lambda}})))$ and a.e. $t$,
\begin{equation*}
\la\rho _{\Lambda }\left( t\right) ,dH(\sigma)\cdot \partial _{s}Q\ra=0.
\end{equation*}
\end{lemma}
In particular, the restriction of $\rho _{\Lambda }\left( t\right)$ to $\sigma\in R_\Lambda$  is invariant under the action of $\Lambda^\perp$ by translation on the parameter $s$.

\begin{proof}
This lemma may be understood as a consequence of the invariance by
the classical flow of semi-classical measures. Indeed, the same
arguments than those of Appendix~\ref{sec:Ap1} give that for all
$\ell \in \R$, we have
\begin{equation*}
\la w_{I_{\Lambda },h,R}(t),a\ra=\la w_{I_{\Lambda
},h,R}(t),a\circ \phi _{\ell }^{0}\ra+o(1)
\end{equation*}%
as $h$ goes to $0$ ($R$ fixed). As a consequence, we
deduce that for all $\ell \in \R$,
\begin{equation*}
\la\rho ^{h}(t),Q_{a,R}\ra=\la\rho ^{h}(t),Q_{a\circ \phi _{\ell }^{0},R}\ra%
+o(1)
\end{equation*}%
as $h$ goes to $0$ and $R$ to $+\infty $. Recall that $\phi _{\ell
}^{0}(x,\xi )=\left( x+\ell dH(\xi ),\xi \right) $ and that for
$\sigma \in
I_{\Lambda }$, $dH(\sigma )\in \Lambda ^{\perp }$. As a consequence, if $%
x=s+y\in \Lambda ^{\perp }\oplus \la\widetilde{\Lambda }\ra$, then for $%
\sigma \in I_{\Lambda }$, the vector $x+\ell dH(\sigma )$
decomposes as
\begin{equation*}
x+\ell dH(\sigma )=\left( s+\ell dH(\sigma )\right) +y\in \Lambda
^{\perp }\oplus \la\widetilde{\Lambda }\ra,
\end{equation*}%
and the kernel of the operator $Q_{a\circ \phi _{\ell
}^{0},R}(\omega ,s,\sigma )$ is the function
\begin{equation*}
\displaylines{\qquad
Q_{a\circ\phi^0_\ell,R}(\omega,s,\sigma)(\tilde y',\tilde y)=
\frac{1}{(2\pi)^{d_\Lambda}}\sum_{k\in\tilde\Lambda}e^{i\omega
k}\int a\left(s+\ell dH(\sigma)+ \frac{\tilde y+\tilde y'}2,
\sigma,\eta\right) \hfill\cr\hfill \times\,
\chi\left(\eta/R\right)e^{i\eta\cdot(\tilde y'-\tilde
y+k)}d\eta.\qquad\cr}
\end{equation*}
The result follows if we note that compact pseudodifferential
operators (e.g. operators of the form $Q_{a, R}(\omega, s,
\sigma)$) are dense in $\cK (L^{2}(D_{\tilde{\Lambda}}))$ for the
weak topology of operators .
\end{proof}

We finally show that $\rho _{\Lambda }$, restricted to $\sigma \in
R_{\Lambda }$ obeys the propagation law~(\ref{eq:SchroM(t)}). From now on, we only consider test symbols $Q(\omega, \sigma)$ that do not depend on the parameter $s\in \Lambda^\perp$.
We recall that we write $V(t,x,\xi )$ as $V(t,s+\tilde{y},\sigma
+\eta )$ and that we use the notation $\la V(t,\tilde{y},\sigma
+\eta )\ra_{\Lambda }$ to mean that we are averaging
$V(t,s+\tilde{y},\sigma +\eta )$ w.r.t.~$s$, thus getting a
function that does not depend on $s$ (nor $\eta $ when $\xi
\in I_{\Lambda }$). In the notation of~\eqref{e:pseudo}, we note that $Q_{%
\la V(t,\cdot ,\sigma )\ra_{\Lambda }}^{h=0}=\la V(t)\ra_{\Lambda
,\sigma }$ defines a multiplication operator on
$L^{2}(D_{\tilde{\Lambda}})$ (for which the Floquet-Bloch periodicity conditions are transparent).
To simplify the notation, we set in what follows:
\begin{equation*}
A\left( \sigma ,\eta \right) :=\frac{1}{2}d^{2}H(\sigma )\eta
\cdot \eta .
\end{equation*}%
To prove that $\rho _{\Lambda }$ satisfies a Schr\"{o}dinger-type
equation,
we note that 
\begin{eqnarray*}
K_{h}H(hD_{x})f(\sigma ,y) &=&\frac{1}{h^{d/2}}m(\sigma )\int_{\eta \in \la%
\Lambda \ra}H(\sigma +\eta )\mathcal{F}f\left( {\frac{\sigma +\eta }{h}}%
\right) e^{{\frac{i}{h}}\eta \cdot y}\frac{d\eta }{(2\pi
h)^{d_{\Lambda }/2}}
\\
&=&H(\sigma +hD_{y})K_{h}f(\sigma ,y),
\end{eqnarray*}%
and that
\begin{equation*}
K_{h}\Op_{h}\left( V\left( t,\cdot \right) \right) f(\sigma
,y)=P_{Q_{V}^{h}}^{h}(hD_{\sigma },\sigma )K_{h}f(\sigma ,y)
\end{equation*}%
(we used the notation of Definition \ref{d:Psigma}).
Therefore, $w_{h}\left( t,\cdot \right) :=K_{h}S_{h}^{t/h}u_{h}$ solves:%
\begin{equation*}
i\partial _{t}w_{h}\left( t,\sigma ,y\right) =\left( h^{-2}H\left(
\sigma +hD_{y}\right) +P_{Q_{V}^{h}}^{h}(hD_{\sigma },\sigma
)\right) w_{h}\left( t,\sigma ,y\right) .
\end{equation*}%
 Note that if $Q(\omega ,s,\sigma )$ does not depend on $s$, we have
 \begin{multline*}
\la\rho^{h}(t),[h^{-2}H(\sigma +hD_{y}) ,Q]\ra \\=
 \la \chi _{0} e^{i\frac{\sigma _{\Lambda }}{h}\cdot
}\left( h^{-2}H\left(
\sigma +hD_{y}\right)\right) w_h(t) , P_{Q}^{h}\left(
hD_{\sigma
},\cdot \right)  \chi _{0} e^{i\frac{\sigma _{\Lambda }}{h}\cdot
} w_h(t) \ra  \\-  \la \chi _{0} e^{i\frac{\sigma _{\Lambda }}{h}\cdot
}w_h(t) , P_{Q}^{h}\left(
hD_{\sigma
},\cdot \right)  \chi _{0} e^{i\frac{\sigma _{\Lambda }}{h}\cdot
} \left( h^{-2}H\left(
\sigma +hD_{y}\right)\right) w_h(t) \ra_{L^{2}\left( I_{\Lambda };L^{2}(D_{%
\tilde{\Lambda}})\right) }
\end{multline*}
Hence, passing to the limit $h\To 0$,
\begin{equation*}
\int_{\mathbb{R}}\phi ^{\prime }\left( t\right) \la\rho ^{h}(t),Q\ra %
dt=i\int_{\mathbb{R}}\phi \left( t\right) \la\rho
^{h}(t),[h^{-2}H(\sigma +hD_{y})_{\omega }+Q_{V}^{h}(s,\sigma
),Q]\ra dt+o(1).
\end{equation*}%
Above, the index $\omega $ in $H(\sigma +hD_{y})_{\omega }$
indicates that the operator acts on $L^{2}(\la
D_{\tilde{\Lambda}}\ra)$ with Floquet
periodicity conditions (\ref{e:BF}). We perform a Taylor expansion of $%
H(\sigma +hD_{y})$ and write, in ${\mathcal{L}}(L^{2}(\la D_{\tilde{\Lambda}}%
\ra))$, for any $Q\in {\mathcal{K}}\left( L^{2}(\la D_{\tilde{\Lambda}}\ra%
)\right) $,
\begin{equation*}
H(\sigma +hD_{y})Q=H(\sigma )Q+hdH(\sigma )D_{y}Q+h^{2}A(\sigma
,D_{y})Q+O(h^{3}).
\end{equation*}%
At this point, note that $dH(\sigma )D_{y}=0$ (since $\sigma \in
I_{\Lambda } $ one has $dH(\sigma )\in \Lambda ^{\perp }$).
Therefore, for $Q\in \mathcal{C}_{c}^{\infty }(\la\Lambda
\ra/\Lambda \times \Lambda ^{\perp }\times I_{\Lambda
},\cK(L^{2}(D_{\tilde{\Lambda}})))$,
\begin{equation*}
\left[ h^{-2}H(\sigma +hD_{y})_{\omega },Q(\omega ,hD_{\sigma },\sigma )%
\right] =\left[ A(\sigma ,D_{y})_{\omega },Q(\omega ,hD_{\sigma },\sigma )%
\right] +O(h).
\end{equation*}%
As a consequence, we obtain:
\begin{equation*}
\int_{\mathbb{R}}\phi ^{\prime }\left( t\right) \la\rho ^{h}(t),Q\ra %
dt=i\int_{\mathbb{R}}\phi \left( t\right) \la\rho
^{h}(t),[A(\sigma ,D_{y})_{\omega }+Q_{V}^{h}(hD_{\sigma },\sigma
),Q\left( \omega ,hD_{\sigma },\sigma \right) ]\ra dt+o(1).
\end{equation*}%
Taking limits, we obtain
\begin{eqnarray*}
\int_{\mathbb{R}}\phi ^{\prime }\left( t\right) \la\rho _{\Lambda }(t),Q\ra %
dt &=&i\int_{\mathbb{R}}\phi \left( t\right) \la\rho _{\Lambda
}(t),[A(\sigma ,D_{y})_{\omega }+Q_{V}^{0}(s,\sigma ),Q]\ra dt \\
&=&i\int_{\mathbb{R}}\la\rho _{\Lambda }(t),[A(\sigma ,D_{y})_{\omega }+Q_{%
\la V\ra}^{0}(s,\sigma ),Q]\ra dt\label{e:props}
\end{eqnarray*}%
where the potential $\la V\ra(s,\sigma )$ is averaged along the flow $%
s\mapsto s+tdH(\sigma )$ (because of Lemma~\ref{p:averaging}). But
$\la V\ra$ does not depend on $s$ for $\sigma \in R_{\Lambda }$,
and it is simply the average of~$V$ w.r.t. $s$. Hence $Q_{\la
V\ra}^{0}=\la V(t)\ra_{\Lambda ,\sigma }$ and $\rho _{\Lambda }$
satisfies the following Heisenberg
equation for $\sigma \in R_{\Lambda }$:%
\begin{equation*}
i\partial _{t}\rho _{\Lambda }(t,\omega ,\sigma )=\left[ A(\sigma
,D_{y})_{\omega }+\la V(t)\ra_{\Lambda ,\sigma },\rho _{\Lambda
}(t,\omega ,\sigma )\right]
\end{equation*}%
(note that $\rho _{\Lambda }(t,\omega ,\sigma )$ does not depend on $s$ for $\sigma\in R_\Lambda$).
Let $$m_{\Lambda }\left( t,\omega ,s, \sigma \right)
:=\mathrm{Tr}_{L_{\omega }^{2}(\mathbb{R}^{d},\Lambda )}\rho
_{\Lambda }(t,\omega ,s, \sigma );$$ the propagation law \eqref{e:props}
implies that $m_{\Lambda }$ does not depend on
$t$. Therefore, $\rho _{\Lambda }=M_{\Lambda }m_{\Lambda }$ where $%
M_{\Lambda }(\cdot, \omega, \sigma)$ solves (\ref{eq:SchroM(t)}) for $\sigma\in R_\Lambda$ and $\mathrm{Tr}M_{\Lambda } =1$. This concludes the
proof of Theorem~\ref{prop:opvame} (in the statement, the parameter $s$ disappeared since all test functions are independent of $s$).

%%%%%%%%%%%%%%%%%%%%%%%%%%%%%%%%%%%%%%%%%%%%%%%%%%%%%%%%%%%%
%%%%%%%%%%%%%%%%%%%%%%%%%%%%%%%%%%%%%%%%%%%%%%%%%%%%%%%%%%%%%%

\section{An iterative procedure for computing $\protect\mu $}

\label{s:successive}

In this section, we develop the iterative procedure which leads to
the proof of
Theorem~\ref{t:precise}

\subsection{First step of the construction}

\label{s:firststep}

What was done in the previous section can be considered as the
first step of an iterative procedure that allows to effectively
compute the measure~$\mu (t,\cdot )$ solely in terms of the sequence of
initial data $\left( u_{h}\right) $. Recall that we assumed in~\S
\ref{s:second}, without loss of generality, that the projection on
$\xi $ of $\mu \left( t,\cdot \right) $ was supported in a ball
contained in $\mathbb{R}^{d}\setminus C_{H}$. We have decomposed
this measure as a sum
\begin{equation*}
\mu (t,\cdot )=\sum_{\Lambda \in {\mathcal{L}}}\mu _{\Lambda
}(t,\cdot )+\sum_{\Lambda \in {\mathcal{L}}}\mu ^{\Lambda
}(t,\cdot ),
\end{equation*}%
where $\Lambda $ runs over the set of primitive submodules of
$\IZ^{d}$, and where
\begin{equation*}
\mu _{\Lambda }(t,\cdot )=\int_{\left\langle \Lambda \right\rangle }\tilde{%
\mu}_{\Lambda }(t,\cdot ,d\eta )\rceil _{\IT^{d}\times R_{\Lambda
}},\qquad
\mu ^{\Lambda }(t,.)=\int_{\overline{\left\langle \Lambda \right\rangle }}%
\tilde{\mu}^{\Lambda }(t,\cdot ,d\eta )\rceil _{\IT^{d}\times
R_{\Lambda }}.
\end{equation*}%
From Theorem \ref{Thm Properties}, the distributions
$\tilde{\mu}_{\Lambda }$ have the following properties~:\smallskip

\begin{enumerate}
\item $\tilde{\mu}_{\Lambda}(t,dx,d\xi,d\eta)$ is in
$\mathcal{C}\left( \IR;
(\cS_\Lambda^1)^{\prime }\right) $ and all its $x$-Fourier modes are in $%
\Lambda$; with respect to the variable $\xi$,
$\tilde{\mu}_{\Lambda} (t,dx,d\xi,d\eta)$ is supported in
$I_{\Lambda}$;\smallskip\smallskip

\item if $\tau_{h}\ll1/h$ then for every $t\in\mathbb{R}$, $\tilde{\mu }%
_{\Lambda}\left( t,\cdot\right) $ is a positive measure and:
\begin{equation*}
\tilde{\mu}_{\Lambda}\left( t,\cdot\right) =\left(
\tilde{\phi}_{t} ^{1}\right) _{\ast}\tilde{\mu}_{\Lambda}\left(
0,\cdot\right) ,
\end{equation*}
where:
\begin{equation*}
\tilde{\phi}_{s}^{1}:(x,\xi,\eta)\longmapsto(x+sd^{2}H(\sigma(\xi))\eta
,\xi,\eta);
\end{equation*}

\item if $\tau _{h}=1/h$ then$\int_{\left\langle \Lambda \right\rangle }%
\tilde{\mu}_{\Lambda }(t,\cdot ,d\eta )$ is in $\mathcal{C}(\IR;\cM%
_{+}(T^{\ast }\IT^{d}))$ and $\int_{\IR^{d}\times \left\langle
\Lambda \right\rangle }\tilde{\mu}_{\Lambda }(t,\cdot ,d\xi ,d\eta
)$ is an absolutely continuous measure on $\IT^{d}$. In fact, with
the notations of Section \ref{sec:infini}, we have, for every
$a\in \mathcal{C}_{c}^{\infty }\left( T^{\ast }\IT^{d}\right) $
with Fourier modes in $\Lambda $,
\begin{equation*}
\int_{\IT^{d}\times I_{\Lambda }\times \left\langle \Lambda
\right\rangle }a(x,\xi )\tilde{\mu}_{\Lambda }(t,dx,d\xi ,d\eta
)=\int_{(\langle \Lambda \rangle /\Lambda )\times I_{\Lambda
}}\Tr\left( a_{\sigma }\rho _{\Lambda }(t,d\omega ,d\sigma
)\right)
\end{equation*}%
where $\rho _{\Lambda }\in L^{\infty }\left( \R_{t},\Gamma \left( \mathcal{K}%
({\mathfrak{F}})\right) _{+}^{\prime }\right) $ and $a_{\sigma }$
is the section of ${\mathcal{L}}({\mathfrak{F}})$ defined by the
map~$(\omega
,\sigma )\mapsto $ multiplication by $a(y,\sigma )$. In addition, if $%
\mathbf{V}_{h}\left( t\right) =\Op_{h}(V\left( t,\cdot \right) )$
then $\rho
_{\Lambda }=M_{\Lambda }m_{\Lambda }$ where $m_{\Lambda }\in \mathcal{M}%
_{+}((\la\Lambda \ra/\Lambda )\times I_{\Lambda })$, $M_{\Lambda
}$ is a
section of $\mathcal{L}^{1}({\mathfrak{F}})$ integrable with respect to $%
m_{\Lambda }$. Moreover, $\mathrm{Tr}_{L_{\omega }^{2}(\mathbb{R}%
^{d},\Lambda )}M_{\Lambda }\left( t,\omega ,\sigma \right) =1$ and $%
M_{\Lambda }(\cdot, \omega, \sigma)$ satisfies a Heisenberg equation
(\ref{eq:SchroM(t)}).
\end{enumerate}

On the other hand, the measures $\tilde{\mu}^{\Lambda}$
satisfy:\smallskip

\begin{enumerate}
\item for $a\in\cS_{\Lambda}^{1}$,
$\la\tilde{\mu}^{\Lambda}(t,dx,d\xi ,d\eta),a(x,\xi,\eta)\ra$ is
obtained as the limit of
\begin{equation*}
\left\langle w_{h,R,\delta}^{I_{\Lambda}}\left( t\right)
,a\right\rangle
=\int_{T^{\ast}\mathbb{T}^{d}}\chi\left( \frac{\eta\left( \xi\right) }{\delta%
}\right) \left( 1-\chi\left( \frac{\tau_{h}\eta(\xi)}{R}\right)
\right) a\left( x,\xi,\tau_{h}\eta(\xi)\right) w_{h}\left(
t\right) \left( dx,d\xi\right) ,
\end{equation*}
in the weak-$\ast$ topology of $L^{\infty}(\IR,\left(
\cS_{\Lambda}
^{1}\right) ^{\prime})$, as $h\To0^{+}$, $R\To+\infty$ and then $\delta\To%
0^{+}$ (possibly along subsequences);\smallskip

\item $\tilde{\mu}^{\Lambda}(t,dx,d\xi,d\eta)$ is in $L^{\infty} (\IR,\cM%
_{+}(T^{\ast}\IT^{d}\times\overline{\la\Lambda\ra}))$ and all its $x$-Fourier
modes are
in $\Lambda$. With respect to the variable $\eta$, the measure $\tilde{\mu}%
^{\Lambda }(t,dx,d\xi,d\eta)$ is $0$-homogeneous and supported at
infinity~: we see it as a measure on the sphere at infinity
$\mathbb{S}\la\Lambda\ra$. With respect to the variable~$\xi$, it
is supported on $\{\xi\in I_{\Lambda} \}$;\smallskip

\item $\tilde{\mu}^{\Lambda}$ is invariant by the two flows,
\begin{equation*}
\phi_{s}^{0}:(x,\xi,\eta)\longmapsto(x+sdH(\xi),\xi,\eta),\text{\quad
and\quad}\phi_{s}^{1}:(x,\xi,\eta)\longmapsto(x+sd^{2}H(\sigma(\xi))\frac {%
\eta}{\left\vert \eta\right\vert },\xi,\eta).
\end{equation*}
\end{enumerate}

This is the first step of an iterative procedure; the next step is
to decompose the measure $\mu^{\Lambda}(t,\cdot)$ according to
primitive
submodules of $\Lambda$. We need to adapt the discussion of~\cite%
{AnantharamanMacia}; to this aim, we introduce some additional
notation.

\medskip

Fix a primitive submodule $\Lambda\subseteq\mathbb{Z}^{d}$ and
$\sigma\in I_{\Lambda}\setminus C_{H}$. For
$\Lambda_{2}\subseteq\Lambda_{1} \subseteq \Lambda$ primitive
submodules of $(\Z^d)^*$, for $\eta\in\la\Lambda_1\ra$, we denote
\begin{eqnarray*}
\Lambda_{\eta}\left( \sigma, \Lambda_1\right) &:=&
\left(\Lambda_1^\perp
\oplus \R\, d^2H(\sigma).\eta\right)^\perp \cap (\Z^d)^* \\
&=& \left( \R\, d^2H(\sigma).\eta\right)^\perp \cap \Lambda_1,
\end{eqnarray*}
where the orthogonal is always taken in the sense of duality. We note that $%
\Lambda_{\eta}\left( \sigma, \Lambda_1\right) $ is a primitive submodule of $%
\Lambda_1$, and that the inclusion $\Lambda_{\eta}\left( \sigma,
\Lambda_1\right) \subset \Lambda_1$ is strict if $\eta\not= 0$ since $%
d^2H(\sigma)$ is definite. We define:
\begin{equation*}
R_{\Lambda_{2}}^{\Lambda_{1}}(\sigma):= \{\eta\in \la\Lambda_1\ra,
\Lambda_{\eta}\left( \sigma, \Lambda_1\right)=\Lambda_2\}.
\end{equation*}
Because $d^2H(\sigma)$ is definite, we have the decomposition
$(\R^d)^*=
(d^2H(\sigma).\Lambda_2)^\perp\oplus \la\Lambda_2\ra.$ We define $%
P^\sigma_{\Lambda_{2}}$ to be the projection onto
$\la\Lambda_2\ra$ with respect to this decomposition.

\subsection{Step $k$ of the construction\label{s:stepk}}

In the following, we set $\Lambda=\Lambda_{1}$, corresponding to
step $k=1$. We now describe the outcome of our decomposition at
step $k$ ($k\geq 1$); we will indicate in \S \ref{s:rec} how to go
from step $k$ to $k+1$, for $k\geq 1$.

\medskip

At step $k$, we have decomposed $\mu(t,\cdot)$ as a sum
\begin{equation*}
\mu(t,\cdot)=\sum_{1\leq l\leq
k}\sum_{\Lambda_{1}\supset\Lambda_{2}
\supset\ldots\supset\Lambda_{l}}\mu_{\Lambda_{l}}^{\Lambda_{1}\Lambda
_{2}\ldots\Lambda_{l-1}}(t,\cdot)+\sum_{\Lambda_{1}\supset\Lambda_{2}
\supset\ldots\supset\Lambda_{k}}\mu^{\Lambda_{1}\Lambda_{2}\ldots\Lambda_{k}
}(t,\cdot),
\end{equation*}
where the sums run over the \emph{strictly decreasing} sequences
of primitive submodules of $(\IZ^{d})^*$ (of lengths $l\leq k$ in
the first term, of length $k$ in the second term). We have
\begin{equation*}
\mu_{\Lambda_{l}}^{\Lambda_{1}\Lambda_{2}\ldots\Lambda_{l-1}}(t,x,\xi
)=\int_{R_{\Lambda_{2}}^{\Lambda_{1}}(\xi)\times\ldots\times
R_{\Lambda_{l}
}^{\Lambda_{l-1}}(\xi)\times \la\Lambda_l\ra}\tilde{\mu}_{\Lambda_{l}}^{%
\Lambda
_{1}\Lambda_{2}\ldots\Lambda_{l-1}}(t,x,\xi,d\eta_{1},\ldots,d\eta_{l}
)\rceil_{\IT^{d}\times R_{\Lambda_{1}}},
\end{equation*}
\begin{equation*}
\mu^{\Lambda_{1}\Lambda_{2}\ldots\Lambda_{k}}(t,x,\xi)=\int_{R_{\Lambda_{2}
}^{\Lambda_{1}}(\xi)\times\ldots\times
R_{\Lambda_{k}}^{\Lambda_{k-1}}
(\xi)\times \mathbb{S}\la\Lambda_k\ra}\tilde{\mu}^{\Lambda_{1}\Lambda_{2}%
\ldots\Lambda_{k}
}(t,x,\xi,d\eta_{1},\ldots,d\eta_{k})\rceil_{\IT^{d}\times
R_{\Lambda_{1}}}.
\end{equation*}
The distributions
$\tilde{\mu}_{\Lambda_{l}}^{\Lambda_{1}\Lambda_{2}
\ldots\Lambda_{l-1}}$ have the following properties~:

\begin{enumerate}
\item $\tilde{\mu}_{\Lambda_{l}}^{\Lambda_{1}\Lambda_{2}\ldots\Lambda_{l-1}}%
\in \mathcal{C}\left( \IR,\mathcal{D}^{\prime}\left(
T^{\ast}\T^d\times
\mathbb{S}\la\Lambda_1\ra\times \ldots \times\mathbb{S}\la\Lambda_{l-1}\ra %
\times \la\Lambda_l\ra \right) \right) $ and all its $x$-Fourier
modes are in $\Lambda_{l}$; with respect to $\xi$ it is supported
in $I_{\Lambda_{1}}$;\smallskip

\item for every $t\in\mathbb{R}$, $\tilde{\mu }_{\Lambda_{l}}^{\Lambda_{1}%
\Lambda_{2}\ldots\Lambda_{l-1}}(t,\cdot)$ is invariant under the flows $%
\phi_s^j$ ($j=0,1, \ldots, l-1$) defined by
\begin{equation*}
\phi_s^0(x,\xi,\eta_{1},...,\eta_{l})=(x+sdH(\xi),\xi,\eta_{1},...,%
\eta_{l-1},\eta_{l});
\end{equation*}
\begin{equation*}
\phi_s^j(x,\xi,\eta_{1},..., \eta_{l})=(x+sd^2H(\xi)\frac{\eta_j}{|\eta_j|}%
,\xi,\eta_{1},..., \eta_{l});
\end{equation*}

\item if $\tau_{h}\ll1/h$ then for every $t\in\mathbb{R}$, $\tilde{\mu }%
_{\Lambda_{l}}^{\Lambda_{1}\Lambda_{2}\ldots\Lambda_{l-1}}(t,\cdot)$
is a positive measure and
\begin{equation*}
\tilde{\mu}_{\Lambda_{l}}^{\Lambda_{1}\Lambda_{2}\ldots\Lambda_{l-1}}
(t,\cdot)=\left( \tilde{\phi}_{t}^{l}\right)
_{\ast}\tilde{\mu}_{\Lambda
_{l}}^{\Lambda_{1}\Lambda_{2}\ldots\Lambda_{l-1}}(0,\cdot),
\end{equation*}
where, for $\left( x,\xi,\eta_{1},..,\eta_{l}\right) \in
T^{\ast}\T^d\times
\mathbb{S}\la\Lambda_1\ra\times \ldots \times\mathbb{S}\la\Lambda_{l-1}\ra %
\times \la\Lambda_l\ra $ we define:
\begin{equation*}
\tilde{\phi}_{s}^{l}:(x,\xi,\eta_{1},...,\eta_{l})\longmapsto(x+sd^{2}
H(\xi)\eta_{l},\xi,\eta_{1},...,\eta_{l});
\end{equation*}

\item if $\tau _{h}=1/h$ then $\int_{\la\Lambda
_{l}\ra}\tilde{\mu}_{\Lambda _{l}}^{\Lambda _{1}\Lambda _{2}\ldots
\Lambda _{l-1}}(t,\cdot ,d\eta _{l})$ is in
$\mathcal{C}(\IR,\cM_{+}(T^{\ast }\IT^{d}\times
\mathbb{S}\la\Lambda _{1}\ra\times \ldots \times
\mathbb{S}\la\Lambda _{l-1}\ra))$ and the measure
$\int_{(\R^{d})^{\ast }\times \mathbb{S}\la\Lambda _{1}\ra\times
\ldots \times \mathbb{S}\la\Lambda _{l-1}\ra\times \la\Lambda _{l}\ra}\tilde{%
\mu}_{\Lambda _{l}}^{\Lambda _{1}\Lambda _{2}\ldots \Lambda
_{l-1}}(t,\cdot ,d\xi ,d\eta _{1},\ldots ,d\eta _{l})$ is an
absolutely continuous measure
on $\IT^{d}$. Besides, if $a\in \mathcal{C}_{c}^{\infty }\left( T^{\ast }\IT%
^{d}\right) $ has only Fourier modes in $\Lambda _{l}$, then, define ${%
\mathcal{L}}({\mathfrak{F}}_{l})$ the bundle over $(\langle
\Lambda _{l}\rangle /\Lambda _{l})\times I_{\Lambda _{1}}\times
\mathbb{S}\la\Lambda _{1}\ra\times \ldots \times
\mathbb{S}\la\Lambda _{l-1}\ra$ formed of
elements $(\omega ,\sigma ,\eta _{1},\cdots ,\eta _{l-1},Q)$ where $Q\in {%
\mathcal{L}}(L^{2}_\omega(\R^{d},\Lambda _{l}))$, define similarly ${%
\mathcal{K}}({\mathfrak{F}}_{l})$ and ${\mathcal{L}}^{1}({\mathfrak{F}}_{l})$%
, then
\begin{multline*}
\int_{T^{\ast }\T^{d}\times \mathbb{S}\la\Lambda _{1}\ra\times
\ldots \times
\mathbb{S}\la\Lambda _{l-1}\ra\times \la\Lambda _{l}\ra}a(x,\xi )\tilde{\mu}%
_{\Lambda _{l}}^{\Lambda _{1}\Lambda _{2}\ldots \Lambda
_{l-1}}(t,dx,d\xi
,d\eta _{1},\ldots ,d\eta _{l})= \\
\int_{(\langle \Lambda _{l}\rangle /\Lambda _{l})\times I_{\Lambda
_{1}}\times \mathbb{S}\la\Lambda _{1}\ra\times \ldots \times \mathbb{S}\la%
\Lambda _{l-1}\ra}\Tr\left( a_{\sigma }\rho _{\Lambda
_{l}}^{\Lambda _{1}\Lambda _{2}\cdots \Lambda _{l-1}}(t,d\sigma
,d\eta _{1},\cdots ,d\eta _{\ell -1})\right) ,
\end{multline*}%
where $\rho _{\Lambda _{l}}^{\Lambda _{1}\Lambda _{2}\cdots
\Lambda _{l-1}}$
is $L^{\infty }$ in $t$, a positive section element of $\Gamma (K({\mathfrak{%
F}}_{l}))^{\prime }$ and where $a_{\sigma }$ is the section of ${\mathcal{L}}%
({\mathfrak{F}}_{l})$ defined by multiplication by $a(\sigma ,y)$. \\
When $%
\mathbf{V}_{h}\left( t\right) =\Op_{h}(V\left( t,\cdot \right) )$
then $\rho _{\Lambda _{l}}^{\Lambda _{1}\Lambda _{2}\cdots \Lambda
_{l-1}}=M_{\Lambda _{l}}^{\Lambda _{1}\Lambda _{2}\cdots \Lambda
_{l-1}}m_{\Lambda _{l}}^{\Lambda _{1}\Lambda _{2}\cdots \Lambda
_{l-1}}$ where $$m_{\Lambda
_{l}}^{\Lambda _{1}\Lambda _{2}\cdots \Lambda _{l-1}}\in \mathcal{M}%
_{+}((\langle \Lambda _{l}\rangle /\Lambda _{l})\times I_{\Lambda
_{1}}\times \mathbb{S}\la\Lambda _{1}\ra\times \ldots \times \mathbb{S}\la%
\Lambda _{l-1}\ra),$$ $M_{\Lambda _{l}}^{\Lambda _{1}\Lambda
_{2}\cdots \Lambda _{l-1}}$ is a section of
$\mathcal{L}^{1}({\mathfrak{F}}_{l})$ integrable with respect to
$m_{\Lambda _{l}}^{\Lambda _{1}\Lambda _{2}\cdots
\Lambda _{l-1}}$. Moreover, $\mathrm{Tr}_{L_{\omega }^{2}(\mathbb{R}%
^{d},\Lambda _{l})}M_{\Lambda _{l}}^{\Lambda _{1}\Lambda
_{2}\cdots \Lambda _{l-1}}=1$ and $M_{\Lambda _{l}}^{\Lambda
_{1}\Lambda _{2}\cdots \Lambda _{l-1}}$ satisfies a Heisenberg
equation (\ref{eq:SchroM(t)}) with $\Lambda =\Lambda _{l}$.
\end{enumerate}

On the other hand
$\tilde{\mu}^{\Lambda_{1}\Lambda_{2}\ldots\Lambda_{k}}$ satisfy:

\begin{enumerate}
\item $\tilde{\mu}^{\Lambda_{1}\Lambda_{2}\ldots\Lambda_{k}}$ is in $%
L^{\infty}(\IR,\cM_{+}(T^{\ast}\IT^{d}\times\mathbb{S}\la\Lambda_1\ra\times
\ldots \times\mathbb{S}\la\Lambda_{k}\ra))$ and all its
$x$-Fourier modes are in $\Lambda_{k}$;\smallskip

\item $\tilde{\mu}^{\Lambda_{1}\Lambda_{2}\ldots\Lambda_{k}}$ is
invariant by the $k+1$ flows,
$\phi_{s}^{0}:(x,\xi,\eta)\mapsto(x+sdH(\xi),\xi,\eta
_{1},\ldots,\eta_{k})$, and
$\phi_{s}^{l}:(x,\xi,\eta_{1},\ldots,\eta
_{k})\longmapsto(x+sd^{2}H(\sigma(\xi))\frac{\eta_{l}}{\left\vert
\eta _{l}\right\vert },\xi,\eta_{1},\ldots,\eta_{k})$ (where
$l=1,\ldots,k$).
\end{enumerate}

Finally, we define the space $\cS_{\Lambda_{k}}^{k}$ which is the
class of
smooth functions $a(x,\xi ,\eta_{1},\ldots,\eta_{k})$ on $T^{\ast}\IT%
^{d}\times \la \Lambda_1\ra\times\ldots\times \la\Lambda_k\ra$
that are

\begin{enumerate}
\item[(i)] smooth and compactly supported in $(x, \xi)\in
T^{*}\IT^{d}$;

\item[(ii)] homogeneous of degree $0$ at infinity in each variable $%
\eta_{1}, \ldots, \eta_{k}$;

\item[(iii)] such that their non-vanishing $x$-Fourier
coefficients correspond to frequencies in $\Lambda_{k}$.
\end{enumerate}

\subsection{From step $k$ to step $k+1$ ($k\geq 1$)\label{s:rec}}

After step $k$, we leave untouched the term $\sum_{1\leq l\leq k}
\sum_{\Lambda_{1}\supset\Lambda_{2}\supset\ldots\supset\Lambda_{l}}
\mu_{\Lambda_{l}}^{\Lambda_{1}\Lambda_{2}\ldots\Lambda_{l-1}}$ and
decompose further $\sum_{\Lambda
_{1}\supset\Lambda_{2}\supset\ldots\supset\Lambda_{k}}\mu^{\Lambda_{1}
\Lambda_{2}\ldots\Lambda_{k}}$. Using the positivity of $\tilde{\mu}%
^{\Lambda_{1} \Lambda_{2}\ldots\Lambda_{k}}$, we use the procedure
described in Section \ref{s:decompo} to write
\begin{equation}  \label{restriction}
\tilde{\mu}^{\Lambda_{1}\Lambda_{2}\ldots\Lambda_{k}}(\sigma,\cdot
)=\sum_{\Lambda_{k+1}\subset\Lambda_{k}}\tilde{\mu}^{\Lambda_{1}\Lambda
_{2}\ldots\Lambda_{k}}\rceil_{\eta_{k}\in
R_{\Lambda_{k+1}}^{\Lambda_{k} }(\sigma)},
\end{equation}
where the sum runs over all primitive submodules $\Lambda_{k+1}$ of $%
\Lambda_{k}$. Moreover, by Proposition~\ref{prop:decomposition}, all the $x$%
-Fourier modes of
$\tilde{\mu}^{\Lambda_{1}\Lambda_{2}\ldots\Lambda_{k}
}\rceil_{\eta_{k}\in R_{\Lambda_{k+1}}^{\Lambda_{k} }(\sigma)}$ are in $%
\Lambda_{k+1}$. To generalize the analysis of Section
\ref{s:second}, we consider test functions in
$\cS_{\Lambda_{k+1}}^{k+1}$. We let
\begin{equation*}
\displaylines{\qquad
w_{h,R_{1},\ldots,R_{k+1}}^{\Lambda_{1}\Lambda_{2}\ldots
\Lambda_{k+1}}\left( t,x,\xi,\eta_1,\cdots,\eta_{k+1}\right) :=
\left( 1-\chi\left( {\eta_{k+1}\over R_{k+1}} \right) \right)
\hfill\cr\hfill \times\,
w^{\Lambda_1\Lambda_2\cdots\Lambda_k}_{h,R_1,\cdots,R_k}(t,x,\xi,\eta_1,%
\cdots,\eta_k)\otimes\delta_{
P^{\xi}_{\Lambda_{k+1}}(\eta_k)}(\eta_{k+1}) ,\qquad\cr}
\end{equation*}
and
\begin{equation*}
\displaylines{\qquad
w_{\Lambda_{k+1}h,R_{1},\ldots,R_{k+1}}^{\Lambda_{1}\Lambda_{2}\ldots
\Lambda_{k}}\left( t,x,\xi,\eta_1,\cdots,\eta_{k+1}\right) :=
\chi\left( { \eta_{k+1}\over R_{k+1}}\right) \hfill\cr\hfill
\times\,
w^{\Lambda_1\Lambda_2\cdots\Lambda_k}_{h,R_1,\cdots,R_k}(t,x,\xi,\eta_1,%
\cdots,\eta_k)\otimes\delta_{
P^{\xi}_{\Lambda_{k+1}}(\eta_k)}(\eta_{k+1} ) .\qquad\cr}
\end{equation*}

By the Calder\'{o}n-Vaillancourt theorem, both $w_{\Lambda_{k+1}
,h,R_{1},\ldots,R_{k}}^{\Lambda_{1}\Lambda_{2}\ldots\Lambda_{k}}$ and $%
w_{h,R_{1},\ldots,R_{k}}^{\Lambda_{1}\Lambda_{2}\ldots\Lambda_{k+1}}$
are bounded in
$L^{\infty}(\IR,(\cS_{\Lambda_{k+1}}^{k+1})^{\prime}).$ After
possibly extracting subsequences, we can take the following
limits~:
\begin{equation*}
\lim_{R_{k+1}\To+\infty}\cdots\lim_{R_{1}\To+\infty}\lim_{h\To0}\left\langle
w_{h,R_{1},\ldots,R_{k}}^{\Lambda_{1}\Lambda_{2}\ldots\Lambda_{k+1}}\left(
t\right) ,a\right\rangle =:\left\langle
\tilde{\mu}^{\Lambda_{1}\Lambda
_{2}\ldots\Lambda_{k+1}}(t),a\right\rangle ,
\end{equation*}
and
\begin{equation*}
\lim_{R_{k+1}\To+\infty}\cdots\lim_{R_{1}\To+\infty}\lim_{h\To0}\left\langle
w_{\Lambda_{k+1},h,R_{1},\ldots,R_{k}}^{\Lambda_{1}\Lambda_{2}\ldots
\Lambda_{k}}\left( t\right) ,a\right\rangle =:\left\langle \tilde{\mu }%
_{\Lambda_{k+1}}^{\Lambda_{1}\Lambda_{2}\ldots\Lambda_{k}}(t),a\right\rangle
.
\end{equation*}
Then the properties listed in the preceding subsection are a
direct generalisation of Theorems~\ref{mu^Lambda} and~\ref{Thm
Properties} (see also~\cite{AnantharamanMacia}, Section 4) and of
the identity
\begin{eqnarray}  \label{decomposition}
\widetilde\mu^{\Lambda_1\Lambda_2\cdots
\Lambda_k}(t,.)\rceil_{\eta_{k}\in R_{\Lambda_{k+1}}^{\Lambda_{k}
}(\sigma)} & = & \int_{\langle \Lambda_{k+1}\rangle} \widetilde
\mu^{\Lambda_1\Lambda_2\cdots \Lambda_{k+1}}(t,.,d\eta_{k+1})
\rceil_{\eta_{k}\in
R_{\Lambda_{k+1}}^{\Lambda_{k}}(\sigma)} \\
& & + \int_{\langle \Lambda_{k+1}\rangle} \widetilde
\mu^{\Lambda_1\Lambda_2\cdots
\Lambda_{k}}_{\Lambda_{k+1}}(t,.,d\eta_{k+1}) \rceil_{\eta_{k}\in
R_{\Lambda_{k+1}}^{\Lambda_{k}}(\sigma)} .  \notag
\end{eqnarray}

\begin{remark}
By construction, if $\Lambda_{k+1}=\{0\}$, we have
$\tilde\mu^{\Lambda _{1}\Lambda_{2}\ldots\Lambda_{k+1}}=0$, and
the induction stops. Similarly
to Remark \ref{r:nice}, one can also see that if $\operatorname{rk} \Lambda_{k+1}=1$%
, the invariance properties of $\tilde\mu^{\Lambda_{1}
\Lambda_{2}\ldots\Lambda_{k+1}}$ imply that it is constant in $x$.
\end{remark}

\begin{remark}
Note that in the preceding definition of $k$-microlocal Wigner
transform for $k\geq 1$, we did not use a parameter $\delta$
tending to $0$ as we did when
$k=0$ in order to isolate the part of the limiting measures supported above $%
R^{\Lambda_k}_{\Lambda_{k+1}}(\sigma)$. This comes directly from
the restrictions made in~(\ref{restriction})
and~(\ref{decomposition}).
\end{remark}

\subsection{Proof of Theorem \protect\ref{t:precise}}

This iterative procedure allows to decompose $\mu $ along
decreasing sequences of submodules. In particular, when $\tau
_{h}\sim 1/h$, it implies
Theorem~\ref{t:precise}. Indeed, to end the proof of Theorem \ref{t:precise}
, we let after the final step of the induction
\begin{eqnarray*}
{\mathbb{\mu}}^{{\rm final}} _{\Lambda }(t,\cdot ) &=&\sum_{0\leq k\leq d}\sum_{\Lambda
_{1}\supset \Lambda _{2}\supset \cdots \supset \Lambda _{k}\supset
\Lambda }\mu
_{\Lambda }^{\Lambda _{1}\Lambda _{2}\ldots \Lambda _{k}}(t,\cdot ) \\
&=&\sum_{0\leq k\leq d}\sum_{\Lambda _{1}\supset \Lambda
_{2}\supset \cdots \supset \Lambda _{k}\supset \Lambda
}\int_{R_{\Lambda _{2}}^{\Lambda
_{1}}(\xi )\times \ldots \times R_{\Lambda }^{\Lambda _{k}}(\xi )\times \la%
\Lambda \ra}\tilde{\mu}_{\Lambda }^{\Lambda _{1}\Lambda _{2}\ldots
\Lambda _{k}}(t,\cdot ,d\eta _{1},\ldots ,d\eta _{k})\rceil
_{\T^{d}\times R_{\Lambda _{1}}},
\end{eqnarray*}%
where $\Lambda _{1},\ldots ,\Lambda _{k}$ run over the set of
strictly decreasing sequences of submodules ending with $\Lambda
$.  
We know that $\mu
_{\Lambda }^{\Lambda _{1}\Lambda _{2}\ldots \Lambda _{k}}$ is supported on $%
\{\xi \in I_{\Lambda _{1}}\}$, and since $\Lambda \subset \Lambda
_{1}$ we have $I_{\Lambda _{1}}\subset I_{\Lambda }.$
%The measure $\mu_\Lambda$ of Theorem~\ref{t:precise} is $\mu^{{\rm final}}_\Lambda$.

\medskip

 We also let
\begin{equation*}
\rho^{{\rm final}} _{\Lambda }(t, \omega ,\sigma )=\sum_{0\leq k\leq
d}\sum_{\Lambda _{1}\supset \Lambda _{2}\supset \cdots \supset
\Lambda _{k}\supset \Lambda }\int_{R_{\Lambda _{2}}^{\Lambda
_{1}}(\xi )\times \ldots \times R_{\Lambda }^{\Lambda _{k}}(\xi
)}\tilde{\rho}_{\Lambda }^{\Lambda _{1}\Lambda _{2}\ldots \Lambda
_{k}}\left(t,  \omega ,\sigma ,d\eta _{1},\ldots ,d\eta _{k}\right)
\rceil _{\sigma \in R_{\Lambda _{1}}},
\end{equation*}%
where the $\tilde{\rho}_{\Lambda }^{\Lambda _{1}\Lambda _{2}\ldots
\Lambda _{k}}$ are the operator-valued measures appearing in \S
\ref{s:stepk}.

\begin{remark}
\label{r:rhoL}It is clear from this construction that $\rho^{{\rm final}}
_{\Lambda }$ and
$\mu^{{\rm final}} _{\Lambda }(t,\cdot )$ can only charge those $\sigma \in (\mathbb{R}%
^{d})^{\ast }$ with $\Lambda \subseteq dH(\sigma)^\perp$. Moreover, if $%
\mathbf{V}_{h}\left( t\right) =\Op_{h}(V\left( t,\cdot \right) )$
the measure $\rho^{{\rm final}}_{\Lambda }$ admits a decomposition $\rho^{{\rm final}}
_{\Lambda }=N_{\Lambda }{\overline\mu}_{\Lambda }$ where ${\overline\mu}_{\Lambda }$ is a
measure that does not depend on $t$ and $N_{\Lambda }(\cdot, \omega, \sigma)$ is a family
of positive, trace-class operators on $L^2_\omega(\R^d, \Lambda)$ with $\Tr N_{\Lambda }\equiv 1$, satisfying the propagation law (\ref{eq:SchroM(t)}).
\end{remark}

As already mentioned, Theorem~\ref{t:precise} implies Theorem
\ref{t:main} in the case $\tau _{h}\sim 1/h$. The proof of
Theorem~\ref{t:main} in the case $\tau _{h}\ll 1/h$ is discussed
in Section \ref{s:halpha} and in the case $\tau _{h}\gg 1/h$, in
Section \ref{s:hierarchy}.

\subsection{Sufficient assumptions\label{s:suff}}

In the induction, we used the fact that
\begin{equation}  \label{e:S}
\la \Lambda_k\ra = \la \Lambda_{k+1}\ra \oplus (d^2H(\xi)\cdot\la %
\Lambda_{k+1}\ra^\perp\cap \la \Lambda_k\ra) \text{ for all }k.
\end{equation}

Definiteness of the Hessian $d^2 H(\xi)$ is certainly a sufficient
assumption for this, but we see that we actually need less if we
note we are not using this property for arbitrary $\Lambda_k$, but
only for the ones arising in the construction (remember for
instance that for $k=1$ we only need \eqref{e:S} for $\Lambda_1=
dH(\xi)^\perp\cap \Z^d$).

\medskip

A careful analysis of the proof shows that a sufficient set of
assumptions is the following~:

\begin{ass} \label{ass:A}
For every integer $k$, for all $\xi, \eta_1, \ldots,
\eta_k \in (\R^d)^*$, for every \emph{strictly} decreasing
sequence of primitive submodules $\Lambda_1 \supset
\Lambda_2\supset \cdots \supset \Lambda_k\supset \{ 0\}$ such
that:
\begin{equation*}
\Lambda_1= dH(\xi)^\perp\cap \Z^d, \Lambda_2= (d^2 H(\xi)\cdot
\eta_1)^\perp \cap \Lambda_1,\ldots, \Lambda_k= (d^2 H(\xi)\cdot
\eta_{k-1})^\perp \cap \Lambda_{k-1},
\end{equation*}
and
\begin{equation*}
\eta_1\in \la\Lambda_1\ra\setminus \{0\}, \eta_2\in (d^2
H(\xi)\cdot \eta_1)^\perp\setminus (d^2 H(\xi)\cdot
\Lambda_1)^\perp, \eta_k\in (d^2 H(\xi)\cdot \eta_{k-1})^\perp
\setminus(d^2 H(\xi)\cdot \Lambda_{k-1})^\perp
\end{equation*}
then $d^2 H(\xi)\cdot \eta_k\not\in \Lambda_k^\perp$.
\end{ass}

We leave it to the reader to check that Assumptions~\ref{ass:A} implies \eqref{e:S}
and thus is a sufficient assumption for all our results.

\medskip

In dimension $d=2$, Assumptions~\ref{ass:A} is implied by isoenergetic non-degeneracy
(whereas we saw that it is no longer the case for $d\geq 3$). In
dimension $2$, what happens is that, either $dH(\xi)$ is a vector
with rationally independent entries (in which case
$\Lambda_1=\{0\}$ and the conditions of Assumptions~\ref{ass:A} are empty), or
$dH(\xi)$ is a non zero vector with rationally dependent entries :
in this case (and this is very special to dimension $2$),
$\Lambda_1^\perp$ is one-dimensional and coincides with $\R
dH(\xi)$. Thus Assumptions~\ref{ass:A} just says that
\begin{equation*}
dH(\xi)\cdot \eta_1 =0, \eta_1\not=0 \Longrightarrow
dH^2(\xi)\cdot \eta_1\not\in \R dH(\xi)
\end{equation*}
which is isoenergetic non-degeneracy. Remark that $dH(\xi)=0$ is
forbidden by isoenergetic non-degeneracy.

\medskip

Note, finally, that isoenergetic non-degeneracy is only a local
condition at $\xi$ (since it involves only $dH(\xi), d^2 H(\xi)$)
whereas condition Assumptions~\ref{ass:A} contains some global features, namely the
relations between $dH(\xi), d^2 H(\xi)$ and the ring $\Z^d$, which
is the homology group of $\T^d$.
%%%%%%%%%%%%%%%%%%%
%%%%%%%%%%%%%%%%%%%
\section{Some examples of singular concentration\label{s:halpha}}

In Subsection \ref{s:exsing}, assuming $\mathbf{V}_h(t)=0$, we present some
examples of singular concentrations for the scales $\tau_h\ll h$
and, in that manner, we conclude the proof of Theorem~\ref{t:main}
by proving the only remaining point (1). Then the two other
subsections are devoted to the analysis of other cases of singular
concentration which arise when the assumptions of
Theorem~\ref{t:main} are not satisfied.

\subsection{Singular concentration for time scales $\tau_{h}\ll1/h$\label{s:exsing}}

Assume $\mathbf{V}_h(t)=0$ and consider $\rho\in\cS\left(
\mathbb{R}^{d}\right)  $ with $\left\Vert \rho\right\Vert
_{L^{2}\left(  \mathbb{R}^{d}\right)  }=1$ and such that the
Fourier transform $\widehat{\rho}$ is compactly supported. Let
$\left(  x_{0},\xi_{0}\right)
\in\mathbb{R}^{d}\times\mathbb{R}^{d}$ and $\left(
\varepsilon_{h}\right)  $ a sequence of positive real numbers that
tends to zero as $h\longrightarrow
0^{+}$. Form the wave-packet:%
\begin{equation}
v_{h}\left(  x\right)  :=\frac{1}{\left(  \varepsilon_{h}\right)  ^{d/2}}%
\rho\left(  \frac{x-x_{0}}{\varepsilon_{h}}\right)  e^{i\frac{\xi_{0}}{h}\cdot
x}. \label{e:defv_h}%
\end{equation}
Define
$
u_{h}:=\mathbf{P}v_{h},
$
where $\mathbf{P}$ denotes the periodization operator $\mathbf{P}v\left(
x\right)  :=\sum_{k\in\mathbb{Z}^{d}}v\left(  x+2\pi k\right)  $. Since $\rho$ is rapidly decreasing, we have
$\left\Vert u_{h}\right\Vert _{L^{2}\left(  \mathbb{T}^{d}\right)
}\Lim_{h\To 0} 1$. The family $\left(  u_{h}\right)  $ is $h$-oscillatory if $\vareps_h\gg h$.

Theorem \ref{t:main}(1) is a consequence of our next result.

\begin{proposition}
\label{prop:pWP}Let $\left(  \tau_{h}\right)  $ be such that $\lim
_{h\rightarrow0^{+}}h\tau_{h}=0$; suppose that $\varepsilon_{h}\gg h\tau_{h}$.
Then the Wigner distributions of the solutions $S_{h}^{\tau_{h}t}u_{h}$
converge weakly-$\ast$ in $L^{\infty}\left(  \mathbb{R};\mathcal{D}^{\prime
}\left(  T^{\ast}\mathbb{T}^{d}\right)  \right)  $ to $\mu_{\left(  x_{0}%
,\xi_{0}\right)  }$, defined by:%
\begin{equation}
\int_{T^{\ast}\mathbb{T}^{d}}a\left(  x,\xi\right)  \mu_{\left(  x_{0},\xi
_{0}\right)  }\left(  dx,d\xi\right)  =\lim_{T\rightarrow\infty}\frac{1}%
{T}\int_{0}^{T}a\left(  x_{0}+tdH\left(  \xi_{0}\right)  ,\xi_{0}\right)
dt,\quad\forall a\in\cC_{c}(T^{\ast}\T^{d}). \label{e:orbitm}%
\end{equation}

\end{proposition}

We call the measures $\mu_{x_0,\xi_0}$ ``uniform orbit measures'' for $\phi_s$ (their definition and existence as a limit is specific to translation flows on the torus). They are $H$-invariant and the convex hull of the set of uniform orbit measures is dense in the set of $H$-invariant measures. Considering initial data that are linear combinations of wave packets of the form \eqref{e:defv_h}, we see that the convex hull of uniform orbit measures is contained in $\widetilde{\mathcal M}(\tau)$, and since the latter is closed, it contains all  measures invariant by $\phi_s$ as stated in Theorem~\ref{t:main}(1).

\begin{proof}
Start noticing that the sequence of initial conditions $\left(  u_{h}\right)  $ possesses the unique
semiclassical measure $\mu_{0}=\delta_{x_{0}}\otimes\delta_{\xi_{0}}$. Using
property (4) in the appendix, we deduce that the image $\overline{\mu}$ of
$\mu\left(  t,\cdot\right)  $ by the projection from $\mathbb{T}^{d}%
\times\mathbb{R}^{d}$ onto $\mathbb{R}^{d}$ satisfies:%
\[
\overline{\mu}=\sum_{\Lambda\in\mathcal{L}}\overline{\mu}\rceil_{ R_\Lambda}%
=\delta_{\xi_{0}}.
\]
Since the sets
$R_{\Lambda}$ form a partition of
$\mathbb{R}^{d}$, we conclude that $\overline {\mu}\rceil_{ R_\Lambda}=0$
unless $\Lambda=\Lambda_{\xi_{0}}$ and therefore $\mu
=\mu\rceil_{\T^d\times R_{\Lambda_{\xi_{0}}}}$. Therefore, in order to characterize
$\mu$ it suffices to test it against symbols with Fourier
coefficients in $\Lambda _{\xi_{0}}$. Let
$a\in\cC_{c}^{\infty}\left(  T^{\ast}\mathbb{T}^{d}\right)  $ be
such a symbol; we can restrict our attention to the case where $a$
is a trigonometric polynomial in $x$. Let $\varphi\in L^{1}\left(
\mathbb{R}\right)  $. Recall that the Wigner distributions
$w_{h}\left(
t\right)  $ of $S_{h}^{\tau_{h}t}u_{h}$ satisfy%
\[
\int_{\mathbb{R}}\varphi\left(  t\right)  \left\langle w_{h}\left(  t\right)
,a\right\rangle dt=\int_{\mathbb{R}}\varphi\left(  t\right)  \left\langle
w_{h}\left(  0\right)  ,a\circ\phi_{\tau_{h}t}\right\rangle dt+o\left(
1\right)  ;
\]
moreover the Poisson summation formula ensures that the Fourier coefficients
of $u_{h}$ are given by:%
\[
\widehat{u_{h}}\left(  k\right)  =\frac{\left(  \varepsilon_{h}\right)
^{d/2}}{\left(  2\pi\right)  ^{d/2}}\widehat{\rho}\left(  \frac{\varepsilon
_{h}}{h}\left(  hk-\xi_{0}\right)  \right)  e^{-i\left(  k-\xi
_{0}/h\right)  \cdot x_{0}}.
\]
Combining this with the explicit formula (\ref{e:Weylq}) for the Wigner
distribution presented in the appendix we get:%
\begin{equation}%
\begin{split}
\int_{\mathbb{R}}\varphi\left(  t\right)  \left\langle w_{h}\left(  t\right)
,a\right\rangle dt=\frac{\left(  \varepsilon_{h}\right)  ^{d}}{\left(
2\pi\right)  ^{3d/2}}  &  \sum_{k-j\in\Lambda_{\xi_{0}}}\widehat{\varphi
}\left(  \tau_{h}dH\left(  h\frac{k+j}{2}\right)  \cdot\left(  k-j\right)
\right)  \widehat{a}_{j-k}\left(  h\frac{k+j}{2}\right) \\
&  \widehat{\rho}\left(  \frac{\varepsilon_{h}}{h}\left(  hk-\xi_{0}\right)
\right)  \overline{\widehat{\rho}\left(  \frac{\varepsilon_{h}}{h}\left(
hj-\xi_{0}\right)  \right)  }e^{-i\left(  k-j\right)  \cdot x_{0}}+o\left(
1\right)  .
\end{split}
\label{e:wignWP}%
\end{equation}
Now, since $k-j\in\Lambda_{\xi_{0}}$ we can write:%
\begin{align*}
\left\vert dH\left(  h\frac{k+j}{2}\right)  \cdot\left(  k-j\right)
\right\vert  &  =\left\vert \left[  dH\left(  h\frac{k+j}{2}\right)
-dH\left(  \xi_{0}\right)  \right]  \cdot\left(  k-j\right)  \right\vert \\
&  \leq C\left\vert h\frac{k+j}{2}-\xi_{0}\right\vert \left\vert
k-j\right\vert .
\end{align*}
By hypothesis, both $\widehat{\rho}$ and $k\longmapsto\widehat{a_{k}}\left(
\xi\right)  $ are compactly supported, and hence the sum (\ref{e:wignWP}) only involves terms satisfying:%
\[
\frac{\varepsilon_{h}}{h}\left\vert h\frac{k}{2}-\xi_{0}\right\vert \leq
R,\;\frac{\varepsilon_{h}}{h}\left\vert h\frac{j}{2}-\xi_{0}\right\vert \leq
R\text{ and }\left\vert j-k\right\vert \leq R
\]
for some fixed $R$.
This in turn implies%
\[
\left\vert \tau_{h}dH\left(  h\frac{k+j}{2}\right)  \cdot\left(  k-j\right)
\right\vert \leq CR^{2}\,\frac{\tau_{h}h}{\varepsilon_{h}}.
\]
This shows that the limit of (\ref{e:wignWP}) as $h\longrightarrow0^{+}$
coincides with that of:%
\begin{align*}
&  \frac{\left(  \varepsilon_{h}\right)  ^{d}}{\left(  2\pi\right)  ^{3d/2}%
}\sum_{k-j\in\Lambda_{\xi_{0}}}\widehat{\varphi}\left(  0\right)
a_{j-k}\left(  h\frac{k+j}{2}\right)  \widehat{\rho}\left(  \frac
{\varepsilon_{h}}{h}\left(  hk-\xi_{0}\right)  \right)  \overline
{\widehat{\rho}\left(  \frac{\varepsilon_{h}}{h}\left(  hj-\xi_{0}\right)
\right)  }e^{-i\left(  k-j\right)  \cdot x_{0}}\\
&  =\widehat{\varphi}\left(  0\right)  \left\langle w_{h}\left(  0\right)
,a\right\rangle ,
\end{align*}
which is precisely:%
\[
\widehat{\varphi}\left(  0\right)  a\left(  x_{0},\xi_{0}\right)
=\widehat{\varphi}\left(  0\right)  \lim_{T\rightarrow\infty}\frac{1}{T}%
\int_{0}^{T}a\left(  x_{0}+tdH\left(  \xi_{0}\right)  ,\xi_{0}\right)  dt,
\]
since $a$ has only Fourier modes in $\Lambda_{\xi_{0}}$.
\end{proof}

We next present a slight modification of the previous example in order to
illustrate the two-microlocal nature of the elements of $\widetilde
{\mathcal{M}}\left(  \tau\right)  $. Define now, for $\eta_{0}\in
\mathbb{R}^{d}$:%
\[
u_{h}\left(  x\right)  =\mathbf{P}\left[  v_{h}\left(  x\right)  e^{i\eta
_{0}/\left(  h\tau_{h}\right)  }\right]  ,
\]
where $v_{h}$ was defined in (\ref{e:defv_h}).

\begin{proposition}
\label{prop:Diracmasses} Suppose that $\lim_{h\rightarrow0^{+}}h\tau_{h}=0$
and $\varepsilon_{h}\gg h\tau_{h}$. Suppose moreover that $d^{2}H\left(
\xi_{0}\right)  $ is definite and that $\eta_{0}\in\left\langle \Lambda
_{\xi_{0}}\right\rangle $. Then the Wigner distributions of $S_{h}^{\tau_{h}
t}u_{h}$ converge weakly-$\ast$ in $L^{\infty}\left(  \mathbb{R}
;\mathcal{D}^{\prime}\left(  T^{\ast}\mathbb{T}^{d}\right)  \right)  $ to the
measure:
\[
\mu\left(  t,\cdot\right)  =\mu_{\left(  x_{0}+td^{2}H\left(  \xi_{0}\right)
\eta_{0},\xi_{0}\right)  },\qquad t\in\mathbb{R},
\]
where $\mu_{\left(  x_{0},\xi_{0}\right)  }$ is the uniform orbit measure
defined in (\ref{e:orbitm}).
\end{proposition}

\begin{proof}
The same argument we used in the proof of Proposition \ref{prop:pWP} gives
$\mu
=\mu\rceil_{\T^d\times R_{\Lambda_{\xi_{0}}}}$. We claim that $w_{I_{\Lambda_{\xi_{0}}}
,h,R}\left(  0\right)  $ converges to the measure:
\begin{equation}\label{e:claim}
\widetilde{\mu}_{\Lambda_{\xi_{0}}}\left(  0,x,\xi,\eta\right)  =\mu_{\left(
x_{0},\xi_{0}\right)  }\left(  x,\xi\right)  \delta_{\eta_{0}}\left(
\eta\right)  .
\end{equation}
Assume this is the case, Theorem \ref{Thm Properties} (4) implies that:
\[
\widetilde{\mu}_{\Lambda_{\xi_{0}}}\left(  t,x,\xi,\eta\right)  =\mu_{\left(
x_{0}+td^{2}H\left(  \xi_{0}\right)  \eta_{0},\xi_{0}\right)  }\left(
x,\xi\right)  \delta_{\eta_{0}}\left(  \eta\right)  ,\quad\forall
t\in\mathbb{R},
\]
and, since $\widetilde{\mu}_{\Lambda_{\xi_{0}}}\left(  t,\cdot\right)  $ are
probability measures, it follows from Proposition \ref{p:decomposition} that
$\widetilde{\mu}^{\Lambda_{\xi_{0}}}=0$ and~:
\[
\mu\left(  t,\cdot\right)  =\int_{\left\langle
\Lambda_{\xi_{0}}\right\rangle }\widetilde{\mu}_{\Lambda_{\xi_{0}}}\left(
t,\cdot,d\eta\right)  =\mu_{\left(  x_{0}+td^{2}H\left(  \xi_{0}\right)
\eta_{0},\xi_{0}\right)  }.
\]
Let us now prove the claim \eqref{e:claim}. Set
\[
\tilde{u}_{h}(x)=v_{h}\left(  x\right)  e^{i\eta_{0}/\left(  h\tau_{h}\right)
}.
\]
Consider $h_{0}>0$ and $\chi\in{\mathcal{C}}_{0}^{\infty}(\R^{d})$ such that
$\chi\tilde{u}_{h}=\tilde{u}_{h}$ for all $h\in(0,h_{0})$ and $\mathbf{P}
\chi^{2}\equiv1$. We now take $a\in{\mathcal{S}}_{\Lambda}^{1}$ and denote by
$\tilde{a}$ the smooth compactly supported function defined on $\R^{d}$ by
$\tilde{a}=\chi^{2}a$. Using the fact that the two-scale quantization admits
the gain $h\tau_{h}$ (in view of~(\ref{estforgain})),
\begin{align*}
\langle u_{h}\;,\;\Op_{h}^{\Lambda_{\xi_{0}}}(a)u_{h}\rangle_{L^{2}(\T^{d})}
&  =\langle u_{h}\;,\;\Op_{h}^{\Lambda_{\xi_{0}}}(\tilde{a})u_{h}
\rangle_{L^{2}(\R^{d})}\\
&  =\langle\tilde{u}_{h}\;,\;\Op_{h}^{\Lambda_{\xi_{0}}}(a)\tilde{u}
_{h}\rangle_{L^{2}(\R^{d})}+O(h\tau_{h}).
\end{align*}
Therefore, it is possible to lift the computation of the limit of
$w_{I_{\Lambda_{\xi_{0}}},h,R}\left(  0\right)  $ to $T^{\ast}\R^{d}
\times\left\langle \Lambda_{\xi_{0}}\right\rangle $ and, in consequence,
replace sums by integrals. A direct computation gives:
$$\displaylines{\qquad
\langle\tilde{u}_{h},\Op_{h}^{\Lambda_{\xi_{0}}}(a)\tilde{u}_{h}\rangle
_{L^{2}(\R^{d})}   =(2\pi)^{-d}\int_{\mathbb{R}^{3d}}
\mathrm{e}^{i\xi\cdot\left(  x-y\right)  }\overline{\rho}
(x)\rho(y)\hfill\cr\hfill\times
a\left(  x_{0}
+\varepsilon_{h}\frac{x+y}{2},\xi_{0}+\frac{1}{\tau_{h}}\eta_{0}+{\frac
{h}{\varepsilon_{h}}}\xi,\tau_{h}\eta(\xi_{0}+\frac{1}{\tau_{h}}\eta
_{0}+{\frac{h}{\varepsilon_{h}}}\xi)\right)
dxdyd\xi.\qquad
\cr}$$
Note that if $F(\xi)=(\sigma,\eta)$, then
\[
\forall k\in\Lambda,\;\;F(\xi+k)=(\sigma,\eta+k)=F(\xi)+(0,k),
\]
which implies that $dF(\xi)k=(0,k)$ and $d\eta(\xi)k=k$ for all $k\in
\Lambda_{\xi_{0}}$. We deduce $d\eta(\xi_{0})\eta_{0}=\eta_{0}$ since
$\eta_{0}\in\langle\Lambda_{\xi_{0}}\rangle$ and, in view of $\eta(\xi_{0}
)=0$, a Taylor expansion of $\eta(\xi)$ around $\xi_{0}$ gives
\[
\tau_{h}\eta\left(  \xi_{0}+\frac{1}{\tau_{h}}\eta_{0}+{\frac{h}
{\varepsilon_{h}}}\xi\right)  =\eta_{0}+o(1).
\]
Therefore, as $h$ goes to $0$,
\[
\langle\tilde{u}_{h},\Op_{h}^{\Lambda_{{\xi_{0}}}}(a)\tilde{u}_{h}
\rangle\rightarrow a(x_{0},\xi_{0},\eta_{0})=\langle\widetilde{\mu}
_{\Lambda_{\xi_{0}}},a\rangle.
\]

\end{proof}

%%%%%%%%%%%%%%%%%%%%%%%%%%%%%%%%%%%%%%%%%%%%%%%%%%%%%%%%%

\subsection{Singular concentration for Hamiltonians with critical points}

We next show by a quasimode construction that for Hamiltonians having a
degenerate critical point (of order $k>2$) and for time scales $\tau_{h}\ll1/h^{k-1}$, the
set $\widetilde{\mathcal{M}}\left(  \tau\right)  $ always contains singular measures.

\medskip

Suppose $\xi_{0}\in\mathbb{R}^{d}$ is such that:%
\[
dH\left(  \xi_{0}\right)  ,d^{2}H\left(  \xi_{0}\right)  ,...,d^{k-1}H\left(
\xi_{0}\right)  \quad\text{vanish identically.}%
\]
The Hamiltonian $H(\xi)=|\xi|^k$ ($k$ an even integer -- corresponding to the operator $(-\Delta)^{\frac{k}{2}}$) provides such an example
(with $\xi_{0}=0$). Let $u_{h}=\mathbf{P}v_{h}$, where $v_{h}$ is defined in
(\ref{e:defv_h}). If $\varepsilon_{h}\gg h$ it is not hard to see that
\[
\left\Vert H\left(  hD_{x}\right)  u_{h}-H\left(  \xi_{0}\right)
u_{h}\right\Vert _{L^{2}\left(  \mathbb{T}^{d}\right)  }=O\left(
h^{k}/\left(  \varepsilon_{h}\right)  ^{k}\right)  .
\]
Therefore,%
\[
\left\Vert S_{h}^{t}u_{h}-e^{-i\frac{t}{h}H\left(  \xi_{0}\right)  }%
u_{h}\right\Vert _{L^{2}\left(  \mathbb{T}^{d}\right)  }=t\,O\left(
h^{k-1}/\left(  \varepsilon_{h}\right)  ^{k}\right)  ,
\]
and, it follows that, for compactly supported $\varphi\in L^{1}(\R)$ and $a\in
\mathcal{C}_{c}^{\infty}\left(  T^{\ast}\mathbb{T}^{d}\right)  $,
\[
\int_{\mathbb{R}}\varphi(t)\langle w_{h}(t)\;,\;a\rangle dt=\int_{\mathbb{R}%
}\varphi(t)\left\langle u_{h},\Op_{h}(a)u_{h}\right\rangle _{L^{2}%
(\mathbb{T}^{d})}dt+O\left(  \tau_{h}h^{k-1}/\left(  \varepsilon_{h}\right)
^{k}\right)  .
\]
Choosing $\left(  \varepsilon_{h}\right)  $ tending to zero and such that
$\varepsilon_{h}\gg\left(  \tau_{h}h^{k-1}\right)  ^{1/k },$
the latter quantity converges to $a(x_{0},\xi_{0})\Vert\varphi\Vert
_{L^{1}(\R)}$ as $h\To0^{+}$. In other words,
\[
dt\otimes\delta_{x_{0}}\otimes\delta_{\xi_{0}}\in\widetilde{\mathcal{M}}%
(\tau),
\]
whence $dt\otimes\delta_{x_{0}}\in{\mathcal{M}}(\tau).$

\medskip

In the special case of $H(\xi)=|\xi|^k$ ($k$ an even integer), we know that the threshold $\tau_h^H$ is precisely $h^{1-k}$. From the discussion
of \S \ref{s:hierarchy} and previously known results about eigenfunctions of the laplacian, we know that the elements of ${\mathcal{M}}(\tau)$ are absolutely continuous for scales $\tau_h\gg 1/h^{k-1}$. In the case of $\tau_h= 1/h^{k-1}$, one can still show that elements of ${\mathcal{M}}(\tau)$ are absolutely continuous. This requires some extra work which consists in checking that all our proofs still work in this case for $\tau_h= 1/h^{k-1}$ and $\xi$ in a neighbourhood of $\xi_0=0$, replacing the Hessian $d^2H(\xi_0)$ by $d^k H(\xi_0)$,
and the assumption that the Hessian is definite by the assumption that $\left[d^k H(\xi_0).\xi^k=0 \Longrightarrow \xi=0\right]$.

\medskip

In the general case of a Hamiltonian having a degenerate critical point, the existence of such a threshold, and its explicit determination, is by no means obvious.

\subsection{The effect of the presence of a subprincipal symbol of lower order in $h$}

Here we present some remarks concerning how the preceding results may change
when the Hamiltonian $H(hD_{x})$ is perturbed by a potential $h^{\beta
}{\mathbf V}_h(t)$ with $\beta\in(0,2)$ and ${\mathbf V}_h(t)$ is a multiplication operator by some smooth function $V(t,x)$.
In this case, it is possible to find potentials $V(t,x)$ for which
Theorem \ref{t:main}(2) fails, \emph{i.e.} such that there exists $\mu
\in\widetilde{\mathcal{M}}\left(  1/h\right)  $, the
projection of which on $x$ is not absolutely continuous with respect to
$dtdx$. The following example has been communicated to us by Jared Wunsch. On the 2-dimensional
torus, take $H\left(  \xi\right)  =\left\vert \xi\right\vert ^{2}$ and
$V(x_{1},x_{2}):=W(x_{2})$ such that $W(x_{2})=\left(  x_{2}\right)  ^{2}/2$
in the set~$\{|x_{2}|<1/2\}$. Take $\varepsilon\in(0,1)$ and
\[
u_{h}(x,y):=\frac{1}{{\pi}^{1/4}h^{\varepsilon/4}}e^{i\frac{x_{1}}{h}%
}e^{-\frac{\left(  x_{2}\right)  ^{2}}{2h^{\varepsilon}}}\chi(y),
\]
where $\chi$ is a smooth function that is equal to one in $\{|x_{2}|<1/4\}$
and identically equal to~$0$ in $\{|x_{2}|>1/2\}$. One checks that
\[
\left(  -h^{2}\Delta+h^{2(1-\varepsilon)}V-1\right)  u_{h}=h^{2-\varepsilon
}u_{h}+O(h^{\infty}).
\]
It follows that for $\varphi\in L^{1}(\R)$ and $a\in\mathcal{C}_{c}^{\infty
}\left(  T^{\ast}\mathbb{T}^{2}\right)  $,
$$\displaylines{
\qquad\qquad
\lim_{h\rightarrow0^{+}}\int_{\mathbb{R}}\varphi(t)\left\langle
S_{2(1-\epsilon),h}^{t/h}u_{h},\Op_{h}(a)S_{2(1-\epsilon),h}^{t/h}%
u_{h}\right\rangle _{L^{2}(\mathbb{T}^{2})}dt
\hfill\cr\hfill
=\lim_{h\rightarrow0^{+}}\int_{\mathbb{R}}\varphi(t)\left\langle u_{h}%
,\Op_{h}(a)u_{h}\right\rangle _{L^{2}(\mathbb{T}^{2})}dt\qquad\qquad\cr\hfill
=
\left(
\int_{\mathbb{R}}\varphi\left(  t\right)  dt\right) \lim_{h\rightarrow0^{+}} \left\langle u_{h}%
,\Op_{h}(a)u_{h}\right\rangle _{L^{2}(\mathbb{T}^{2})}dt\qquad\qquad\cr\hfill
=
\left(
\int_{\mathbb{R}}\varphi\left(  t\right)  dt\right)  \int_{T^{\ast}%
\mathbb{T}^{2}}a\left(  x,\xi\right)  \mu\left(  dx,d\xi\right)  ,\qquad\qquad\cr}$$
and it is not hard to see that $\mu$ is concentrated on $\{x_{2}=0,\xi
_{1}=1,\xi_{2}=0\}$. In particular the image of $\mu$ by the projection to
$\T^{2}$ is supported on $\{x_{2}=0\}$.

%%%%%%%%%%%%%%%%%%%%%%%%%%%%%%%%%%%%%%%%%%%%%%%%%%%%%%%%%%%%%

 %%%%%%%%%%%%%%%%%%%%%%%%%%%%%%%%%%%%%%%%%%%%%%%%%%%%%%%%%%%%%
%%%%%%%%%%%%%%%%%%%%%%%%%%%%%%%%%%%%%%%%%%%%%%%%%%%%%%%%%%%%%%

\section{Hierarchies of time scales\label{s:hierarchy}}
Here we prove the results announced in \S \ref{subsec:hierarchy} of the introduction.
These results make explicit the relation between the sets
$\widetilde{\mathcal{M}}\left(  \tau\right)  $ as the time scale $\left(
\tau_{h}\right)  $ varies.

\begin{proposition}
\label{p:convsm}Let $\left(  \tau_{h}\right)  $ and $\left( \sigma
_{h}\right)  $ be time scales tending to infinity as
$h\longrightarrow0^{+}$ such that
$\lim_{h\rightarrow0^{+}}\sigma_{h}/\tau_{h}=0$. Then for every
$\mu\in\widetilde{\mathcal{M}}\left(  \tau\right)  $ and almost
every $t\in\mathbb{R}$ there exist $\mu^{t}\in\operatorname{Conv}
\widetilde{\mathcal{M}}\left(  \sigma\right)  $ such that
\begin{equation}
\mu\left(  t,\cdot\right)  =\int_{0}^{1}\mu^{t}\left(  s,\cdot\right)  ds.
\label{e:mutcc}
\end{equation}

\end{proposition}

Before presenting the proof of this result, we shall need two auxiliary lemmas.

\begin{lemma}
\label{l:cconv}Let
$\left(  \sigma_{h}\right)  $ be a time scale tending to infinity as
$h\longrightarrow0^{+}$. Let $\left(v_{h}^{\left(  n\right)  }\right)_{h>0, n\in \N}$ be a normalised family in $L^{2}\left(
\mathbb{T}^{d}\right)$ and define:
\[
w_{h}^{\left(  n\right)  }\left(  t,\cdot\right)  :=w_{S_{h}^{\sigma_{h}
t}v_{h}^{\left(  n\right)  }}^{h}.
\]
Let $c_{h}^{\left(  n\right)  }\geq0$, $n\in \N$, be such that $\sum_{n\in
\N}c_{h}^{\left(  n\right)  }=1$.Then, every weak-$\ast$ accumulation point
in $L^{\infty}\left(  \mathbb{R};\mathcal{D}^{\prime}\left(  T^{\ast
}\mathbb{T}^{d}\right)  \right)  $ of
\begin{equation}
\sum_{n\in I_{h}}c_{h}^{\left(  n\right)  }w_{h}^{\left(  n\right)  }\left(
t,\cdot\right)  \label{e:cconv}
\end{equation}
belongs to $\operatorname{Conv}\widetilde{\mathcal{M}}\left(
\sigma\right)  $.
\end{lemma}

\begin{proof}
Suppose (\ref{e:cconv}) possesses an accumulation point
$\tilde{\mu}\in L^{\infty}\left(  \mathbb{R};\mathcal{M}_{+}\left(
T^{\ast}\mathbb{T} ^{d}\right)  \right)  $ that does not belong to
$\operatorname{Conv} \widetilde{\mathcal{M}}\left(  \sigma\right)
$. By the Hahn-Banach theorem applied to the compact convex sets $\left\{
\tilde{\mu}\right\}  $ and $
\operatorname{Conv}\widetilde{\mathcal{M}}\left(  \sigma\right) $
we can ensure the existence of $\varepsilon>0$,
$a\in\mathcal{C}_{c}^{\infty}\left(  T^{\ast}\mathbb{T}
^{d}\right)  $ and $\theta\in L^{1}\left(  \mathbb{R}\right)  $
such that:
\[
\int_{\mathbb{R}}\theta\left(  t\right)  \left\langle \tilde{\mu}\left(
t,\cdot\right)  ,a\right\rangle dt<-\varepsilon<0,
\]
and,
\begin{equation}
\int_{\mathbb{R}}\theta\left(  t\right)  \left\langle \mu\left(
t,\cdot\right)  ,a\right\rangle dt\geq-{\varepsilon\over
3},\quad\forall\mu
\in\operatorname{Conv}\widetilde{\mathcal{M}}\left(  \sigma\right)
.\label{e:pos}
\end{equation}
Suppose that $\tilde{\mu}$ is attained through a sequence $\left(
h_{k}\right)  $ tending to zero. For $k>k_{0}$ big enough,
\[
\int_{\mathbb{R}}\theta\left(  t\right)  \sum_{n\in I_{h_{k}}}c_{h_{k}
}^{\left(  n\right)  }\left\langle w_{h_{k}}^{\left(  n\right)  }\left(
t,\cdot\right)  ,a\right\rangle dt\leq-{3\over 2}\varepsilon,
\]
which implies that there exists $n_{k}\in \N$ such that:
\begin{equation}
\int_{\mathbb{R}}\theta\left(  t\right)  \left\langle w_{h_{k}}^{\left(
n_{k}\right)  }\left(  t,\cdot\right)  ,a\right\rangle dt\leq-{3\over 2}\varepsilon
.\label{e:neg}
\end{equation}
Therefore, every accumulation point of $\left(  w_{h_{k}}^{\left(
n_{k}\right)  }\right)  $ also satisfies (\ref{e:neg}) which contradicts
(\ref{e:pos}).
\end{proof}

\begin{lemma}
\label{l:mesint}Let $\tau$, $\sigma$ and $\mu$ be as in
Proposition \ref{p:convsm}. For every $\alpha<\beta$ there exists
$\mu_{\alpha,\beta}
\in\operatorname{Conv}\widetilde{\mathcal{M}}\left(  \sigma\right)
$ such that
\[
\frac{1}{\beta-\alpha}\int_{\alpha}^{\beta}\mu\left(  t,\cdot\right)
dt=\int_{0}^{1}\mu_{\alpha,\beta}\left(  t,\cdot\right)  dt.
\]

\end{lemma}

\begin{proof}
Let $\mu\in\widetilde{\mathcal{M}}\left(  \tau\right)  $. Then there exists an
$h$-oscillating, normalised sequence $\left(  u_{h}\right)  $ such that, for
every $\theta\in L^{1}\left(  \mathbb{R}\right)  $ and every $a\in
C_{c}^{\infty}\left(  T^{\ast}\mathbb{T}^{d}\right)  $:
\[
\lim_{h\rightarrow0^{+}}\int_{\mathbb{R}}\theta\left(  t\right)  \left\langle
S_{h}^{\tau_{h}t}u_{h},\Op_{h}(a)S_{h}^{\tau_{h}t}u_{h}\right\rangle
dt=\int_{\mathbb{R}}\theta\left(  t\right)  \left\langle \mu\left(
t,\cdot\right)  ,a\right\rangle dt.
\]
Write $N_{h}:=\tau_{h}/\sigma_{h}$; by hypothesis $N_{h}\longrightarrow\infty$
as $h\longrightarrow0^{+}$. Let $\alpha<\beta$, define $L:=\beta-\alpha$ and
put:
\[
\delta_{h}:=\frac{LN_{h}}{\left\lfloor LN_{h}\right\rfloor },\quad t_{n}
^{h}:=\alpha N_{h}+n\delta_{h},
\]
where $\left\lfloor LN_{h}\right\rfloor $ is the integer part of $LN_{h}$.
Then,
\begin{align*}
\frac{1}{L}\int_{\alpha}^{\beta}\left\langle S_{h}^{\tau_{h}t}u_{h}
,\Op_{h}(a)S_{h}^{\tau_{h}t}u_{h}\right\rangle _{L^{2}(\mathbb{T}^{d})}dt  &
=\frac{1}{LN_{h}}\int_{\alpha N_{h}}^{\beta N_{h}}\left\langle S_{h}
^{\sigma_{h}t}u_{h},\Op_{h}(a)S_{h}^{\sigma_{h}t}u_{h}\right\rangle
_{L^{2}(\mathbb{T}^{d})}dt\\
&  =\frac{1}{LN_{h}}\sum_{n=1}^{\left\lfloor LN_{h}\right\rfloor }
\int_{t_{n-1}^{h}}^{t_{n}^{h}}\left\langle S_{h}^{\sigma_{h}t}u_{h}
,\Op_{h}(a)S_{h}^{\sigma_{h}t}u_{h}\right\rangle _{L^{2}(\mathbb{T}^{d})}dt\\
&  =\frac{1}{LN_{h}}\sum_{n=1}^{\left\lfloor LN_{h}\right\rfloor }\int
_{0}^{\delta_{h}}\left\langle S_{h}^{\sigma_{h}t}v_{h}^{\left(  n\right)
},\Op_{h}(a)S_{h}^{\sigma_{h}t}v_{h}^{\left(  n\right)  }\right\rangle
_{L^{2}(\mathbb{T}^{d})}dt,
\end{align*}
where the functions $v_{h}^{\left(  n\right)  }:=S_{h}^{\sigma_{h}t_{n}^{h}
}u_{h}$ form, for each $n\in\mathbb{Z}$, a normalised sequence indexed by
$h>0$. The result then follows by Lemma \ref{l:cconv} and using the fact that
$\delta_{h}\longrightarrow1$ as $h\longrightarrow0^{+}$.
\end{proof}

\begin{proof}
[Proof of Proposition \ref{p:convsm}]Let
$\mu\in\widetilde{\mathcal{M}}\left( \tau\right)  $; an
application of the Lebesgue differentiation theorem gives the
existence of a countable dense set
$S\subset\mathcal{C}_{c}^{\infty}\left(
T^{\ast}\mathbb{T}^{d}\right)  $ and a set $N\subset\mathbb{R}$ of
measure
zero such that, for $a\in S$ and $t\in\mathbb{R}\setminus N$,
\begin{equation}
\lim_{\varepsilon\rightarrow0^{+}}\frac{1}{2\varepsilon}\int_{t-\varepsilon
}^{t+\varepsilon}\int_{T^{\ast}\mathbb{T}^{d}}a\left(  x,\xi\right)
\mu\left(  s,dx,d\xi\right)  ds=\int_{T^{\ast}\mathbb{T}^{d}}a\left(
x,\xi\right)  \mu\left(  t,dx,d\xi\right)  . \label{e:c1}
\end{equation}
Fix $t\in\mathbb{R}\setminus N$; then, for any $\varepsilon>0$
there exist
$\mu_{\varepsilon}^{t}\in\operatorname{Conv}\widetilde{\mathcal{M}
}\left(  \sigma\right)  $ such that, for every
$a\in\mathcal{C}_{c}^{\infty }\left( T^{\ast}\mathbb{T}^{d}\right)
$,
\begin{equation}
\frac{1}{2\varepsilon}\int_{t-\varepsilon}^{t+\varepsilon}\int_{T^{\ast
}\mathbb{T}^{d}}a\left(  x,\xi\right)  \mu\left(  s,dx,d\xi\right)
ds=\int_{0}^{1}\int_{T^{\ast}\mathbb{T}^{d}}a\left(  x,\xi\right)
\mu_{\varepsilon}^{t}\left(  s,dx,d\xi\right)  ds. \label{e:c2}
\end{equation}
Note that $\operatorname{Conv}\widetilde{\mathcal{M}}\left(
\sigma\right)  $ is sequentially compact for the weak-$\ast$
topology, therefore, there exist a sequence $\left(
\varepsilon_{n}\right)  $ tending to zero and a measure
$\mu^{t}\in\operatorname{Conv}\widetilde{\mathcal{M} }\left(
\sigma\right)  $ such that $\mu_{\varepsilon_{n}}^{t}$ converges
weakly-$\ast$ to $\mu^{t}$. Identities (\ref{e:c1}) and
(\ref{e:c2}) ensure that $\mu\left(  t,\cdot\right)
=\int_{0}^{1}\mu^{t}\left(  s,\cdot\right) ds$.
\end{proof}

\begin{remark}
\label{r:cchosc}
Projecting on $x$ in identity (\ref{e:mutcc}) we deduce that given $\nu
\in\mathcal{M}\left(  \tau\right)  $ there exist $\nu^{t}\in\mathcal{M}\left(
\sigma\right)  $ such that:
\[
\nu\left(  t,\cdot\right)  =\int_{0}^{1}\nu^{t}\left(  s,\cdot\right)  ds.
\]
This, together with the fact that elements of $\mathcal{M}\left(  1/h\right)
$ are absolutely continuous imply the conclusion of Theorem \ref{t:main}(2)
when $\tau_{h}\gg1/h$.
\end{remark}

We now assume that ${\mathbf V}_h(t)=0$. Denote by $\widetilde{\mathcal{M}}\left(  \infty\right)  $ the set of
weak-$\ast$ limit points of sequences of Wigner distributions $\left(
w_{u_{h}}\right)  $ corresponding to sequences $\left(  u_{h}\right)  $
consisting of normalised eigenfunctions of $H\left(  hD_{x}\right)  $. We now
focus on a family of time scales $\tau$ for which the structure of
$\widetilde{\mathcal{M}}\left(  \tau\right)  $ can be described in terms of
the closed convex hull of $\widetilde{\mathcal{M}}\left(  \infty\right)  $.
Given a measurable subset $O\subseteq\mathbb{R}^{d}$, we define:
\[
\tau_{h}^{H}\left(  O\right)  :=h\sup\left\{  \left\vert H\left(  hk\right)
-H\left(  hj\right)  \right\vert ^{-1}\;:\;H\left(  hk\right)  \neq H\left(
hj\right)  ,\;hk,hj\in h\mathbb{Z}^{d}\cap O\right\}  .
\]
Note that the scale $\tau_{h}^{H}$ defined in the introduction coincides with
$\tau_{h}^{H}\left(  \mathbb{R}^{d}\right)  $. The following holds.

\begin{proposition}
\label{p:muconv}Let $O\subseteq\mathbb{R}^{d}$ be an open set such
that $\tau_{h}^{H}\left(  O\right)  $ tends to infinity as
$h\longrightarrow0^{+}$. Suppose $\left(  \tau_{h}\right)  $ is a
time scale such that $\lim _{h\rightarrow0^{+}}\tau_{h}^{H}\left(
O\right)  /\tau_{h}=0$. If $V=0$ and if $\mu
\in\widetilde{\mathcal{M}}\left(  \tau\right)  $ is obtained
through a sequence whose semiclassical measure satisfies
$\mu_{0}\left(  \mathbb{T}^{d} \times\left(
\mathbb{R}^{d}\setminus O\right)  \right)  =0$ then $\mu
\in\operatorname{Conv}\widetilde{\mathcal{M}}\left(  \infty\right)
$.
\end{proposition}

\begin{proof} Since the Fourier coefficient of $S_{h}^{\tau_{h}t}u_{h}$ are ${\rm e}^{-it{\tau_h\over h}H(hk)}\widehat{u_h}(k)$ and in view of~(\ref{e:defWD}) and~(\ref{e:Weylq}), we can write
for $a\in\mathcal{C}_{c}^{\infty}\left(  T^{\ast}\mathbb{T}
^{d}\right)  $ and $\theta\in L^{1}\left(  \mathbb{R}\right)  $, we write:
\begin{multline}\label{e:wignerbigt}
\int_{\mathbb{R}}\theta\left(  t\right)  \left\langle w_{h}\left(  t\right)
,a\right\rangle dt
=\int_{\mathbb{R}}\theta(t)  \langle u_{h},S_{h}^{\tau_{h}t *}\ \Op_{h}
(a)S_{h}^{\tau_{h}t}\ u_{h}\rangle_{L^{2}(\T^{d})}dt\\
\\
={1\over (2\pi)^{d/2} }
\sum_{k,j\in\mathbb{Z}^d}  \widehat {u_{h}}\left(  k\right)
\overline{\widehat{u_{h}}\left(  j\right)
}\,\widehat{a}_{j-k}\left(  \frac{h}{2}(k+j)\right)
\int_\R \theta(t) {\rm e}^{-it{\tau_h\over h}(H(hk)-H(hj)}dt\\
=\frac{1}{(2\pi)^{d/2}}\sum_{h,j\in\mathbb{Z}^{d}}\widehat{\theta}\left(
\tau_{h}\frac{H\left(  hk\right)  -H\left(  hj\right)  }{h}\right)
\widehat{u}_{h}(k)\overline{\widehat{u}_{h}(j)}\widehat{a}_{j-k}\left(
\frac{h}{2}(k+j)\right)  .\nonumber
\end{multline}
Our assumptions on the semiclassical measure of the initial data implies that,
for a.e. $t\in\mathbb{R}$:
\[
\mu\left(  t,\mathbb{T}^{d}\times\left(  \mathbb{R}^{d}\setminus O\right)
\right)  =0.
\]
Suppose that $\mu$ is obtained through the normalised sequence $\left(
u_{h}\right)  $. Suppose that $a\in\mathcal{C}_{c}^{\infty}\left(
\mathbb{T}^{d}\times O\right)  $ and that $\operatorname*{supp}\widehat
{\theta}$ is compact. For $0<h<h_{0}$ small enough,
\[
\tau_{h}\frac{H\left(  hk\right)  -H\left(  hj\right)  }{h}\notin
\operatorname*{supp}\widehat{\theta},\quad\forall hk,hj\in O\text{ such that
}H\left(  hk\right)  \neq H\left(  hj\right)  .
\]
Therefore, for such $h$, $a$ and $\theta$,
\begin{align*}
\int_{\mathbb{R}}\theta\left(  t\right)  \left\langle w_{h}\left(  t\right)
,a\right\rangle dt  &  =\frac{\widehat{\theta}\left(  0\right)  }{(2\pi
)^{d/2}}\sum_{\substack{kh,hj\in O\\H\left(  hk\right)  =H\left(  hj\right)
}}\widehat{u}_{h}(k)\overline{\widehat{u}_{h}(j)}\widehat{a}_{j-k}\left(
\frac{h}{2}(k+j)\right) \\
&  =\widehat{\theta}\left(  0\right)  \sum_{E_{h}\in H\left(  h\mathbb{Z}
^{d}\right)  \cap H\left(  O\right)  }\left\langle P_{E_{h}}u_{h}
,\Op_{h}(a)P_{E_{h}}u_{h}\right\rangle _{L^{2}(\mathbb{T}^{d})},
\end{align*}
where $P_{E_{h}}$ stands for the orthogonal projector onto the eigenspace
associated to the eigenvalue $E_{h}$. This can be rewritten as:
\[
\int_{\mathbb{R}}\theta\left(  t\right)  \left\langle w_{h}\left(  t\right)
,a\right\rangle dt=\widehat{\theta}\left(  0\right)  \sum_{E_{h}\in H\left(
h\mathbb{Z}^{d}\right)  \cap H\left(  O\right)  }c_{h}^{E_{h}}\left\langle
w_{v_{h}^{E_{h}}}^{h},a\right\rangle ,
\]
where
\[
v_{h}^{E_{h}}:=\frac{P_{E_{h}}u_{h}}{\left\Vert P_{E_{h}}u_{h}\right\Vert
_{L^{2}\left(  \mathbb{T}^{d}\right)  }},\quad\text{and\quad}c_{h}^{E_{h}
}:=\left\Vert P_{E_{h}}u_{h}\right\Vert _{L^{2}\left(  \mathbb{T}^{d}\right)
}^{2}.
\]
Note that $v_{h}^{E_{h}}$ are eigenfunctions of $H\left(  hD_{x}\right)  $ and
the fact that $\left(  u_{h}\right)  $ is normalised implies:
\[
\sum_{E_{h}\in H\left(  h\mathbb{Z}^{d}\right)  \cap H\left(  O\right)  }
c_{h}^{E_{h}}=1.
\]
We conclude by applying (a straightforward adaptation of) Lemma \ref{l:cconv}
to $v_{h}^{E_{h}}$ and $c_{h}^{E_{h}}$.
\end{proof}

\begin{corollary}
Suppose $\tau_{h}^{H}:=\tau_{h}^{H}\left(  \mathbb{R}^{d}\right)
\longrightarrow\infty$ as $h\longrightarrow0^{+}$ and that $\left(  \tau
_{h}\right)  $ is a time scale such that $\tau_{h}^{H}\ll\tau_{h}$. Then
\[
\widetilde{\mathcal{M}}\left(  \tau\right)  =\operatorname{Conv}
\widetilde{\mathcal{M}}\left(  \infty\right)  .
\]

\end{corollary}

\begin{proof}
The inclusion $\widetilde{\mathcal{M}}\left(  \tau\right)
\subseteq \operatorname{Conv}\widetilde{\mathcal{M}}\left(
\infty\right)  $ is a consequence of the previous result with
$O=\mathbb{R}^{d}$. The converse inclusion can be proved by
reversing the steps of the proof of Proposition \ref{p:muconv}.
\end{proof}

\begin{remark}
Proposition \ref{p:conv} is a direct consequence of this result.
\end{remark}

\section{Observability and unique continuation.}\label{sec:obs}

In this section we prove Theorem \ref{t:semiclassicalObs}. Start
noticing that the fact that (\ref{e:scobs}) does not hold is
equivalent to the existence of a sequence $\left( u_{h}\right) $
in $L^{2}(\mathbb{T}^{d})$
such that:%
\begin{equation*}
\left\Vert \chi \left( hD_{x}\right) u_{h}\right\Vert _{L^{2}(\mathbb{T}%
^{d})}=1,
\end{equation*}%
and%
\begin{equation*}
\lim_{h\rightarrow 0^{+}}\int_{0}^{T}\int_{U}\left\vert
S_{h}^{t/h}\chi \left( hD_{x}\right) u\left( x\right) \right\vert
^{2}dxdt=0.
\end{equation*}%
This in turn, is equivalent to the existence of an element $\mu \in \mathcal{%
\widetilde{M}}\left( 1/h\right) $ such that:%
\begin{equation}
\overline{\mu }(\supp\chi )=1,\quad \overline{\mu }(C_{H})=0,\quad
\int_{0}^{T}\mu \left( t,U\times \supp\chi \right) dt=0,
\label{e:mnul}
\end{equation}%
(recall that $\overline{\mu }$ is the projection on $\mu$ on the $\xi$-coordinate).
This establishes the equivalence between statements (i) and (ii) in Theorem %
\ref{t:semiclassicalObs}. \smallskip

Let $\mu \in \mathcal{\widetilde{M}}\left( 1/h\right) $ such that $\overline{%
\mu }(C_{H})=0$. Theorem \ref{t:precise} implies that $\mu $
decomposes as a
sum of positive measures:%
\begin{equation*}
\mu =\sum_{\Lambda \in \mathcal{L}}\mu _{\Lambda }^{{\rm final}},
\end{equation*}%
such that, see Remark \ref{r:rhoL} and Theorem \ref{prop:opvame}, for any $%
b\in \mathcal{C}(T^{\ast }\mathbb{T}^{d})$,%
\begin{equation*}
\int_{T^{\ast }\mathbb{T}^{d}}b\left( x,\xi \right) \mu _{\Lambda
}^{{\rm final}}\left(
t,dx,d\xi \right) =\int_{(\la\Lambda\ra/\Lambda)\times I_{\Lambda}}\operatorname{Tr}
\left(
m_{\left\langle b\right\rangle_{\Lambda}}\left(  \sigma\right)
 N_{\Lambda}(t, \omega, \sigma)
 \right)
\bar{\mu}_{\Lambda}( d\omega, d\sigma),
\end{equation*}%
for some $\bar{\mu}_{\Lambda}\in \mathcal{M}_{+}\left( (\la\Lambda
\ra/\Lambda )\times \mathbb{R}^{d}\right) $ and where $N_{\Lambda
}\left( t,\omega
,\sigma \right) $ is given by:%
\begin{equation}
N_{\Lambda }\left( t,\omega ,\sigma \right) =U_{\Lambda ,\omega
,\sigma }\left( t\right) N_{\Lambda }^{0}\left( \omega ,\sigma
\right) U_{\Lambda ,\omega ,\sigma }^{\ast }\left( t\right) ,
\label{e:mprop}
\end{equation}%
for some positive, self-adjoint trace-class operator $N_{\Lambda
}^{0}\left( \omega ,\sigma \right) $ acting on $L_{\omega
}^{2}(\mathbb{R}^{d},\Lambda )$ with $\mathrm{Tr}_{L_{\omega
}^{2}(\mathbb{R}^{d},\Lambda )}N_{\Lambda }^{0}(\omega ,\sigma )=1$ and where $U_{\Lambda,\omega,\sigma}(t)$ is the unitary propagator of the equation (S$_{\Lambda, \omega, \sigma}$).

\medskip 

Therefore, the measure $\bar{\mu}_{\Lambda}$ only charges those $\sigma \in
\mathbb{R}^{d}$
satisfying $\Lambda \subseteq dH(\sigma)^\bot$
(see Remark~\ref{r:rhoL}) and we also have:%
\begin{equation*}
\int_{\mathbb{T}^{d}}\mu _{\Lambda }^{{\rm final}}\left( dx,\cdot \right) =\int_{\la \Lambda \ra/\Lambda}\bar{\mu}_{\Lambda}(d\omega ,\cdot ).
\end{equation*}%
If $\left( \varphi _{j}^{0}\left( \omega ,\sigma \right) \right)
_{j\in
\mathbb{N}}$ is an orthonormal basis in $L_{\omega }^{2}(\mathbb{R}%
^{d},\Lambda )$ consisting of eigenfunctions of the operator $N_{\Lambda
}^{0}\left( \omega ,\sigma \right) $ then
\begin{equation*}
N_{\Lambda }^{0}\left( \omega ,\sigma \right) =\sum_{j=1}^{\infty
}\lambda
_{j}(\omega ,\sigma )|\varphi _{j}^{0}\left( \omega ,\sigma \right) \ra\la%
\varphi _{j}^{0}\left( \omega ,\sigma \right) |,
\end{equation*}%
where $\sum_{j=1}^{\infty }\lambda _{j}=1$ and $\lambda _{j}\geq 0$. Now (%
\ref{e:mprop}) implies that:%
\begin{equation}
N_{\Lambda }\left( t,\omega ,\sigma \right) =\sum_{j=1}^{\infty
}\lambda
_{j}(\omega ,\sigma )|\varphi _{j}\left( t,\omega ,\sigma \right) \ra\la%
\varphi _{j}\left( t,\omega ,\sigma \right) |  \label{e:eigenfm}
\end{equation}%
where $\varphi _{j}\left( t,\omega ,\sigma \right) \in L_{\omega }^{2}(%
\mathbb{R}^{d},\Lambda )$ is the solution to :%
\begin{equation*}
i\partial _{t}\varphi _{j}\left( t,\omega ,\sigma \right) =\left( \frac{1}{2}%
d^{2}H(\sigma )D_{y}\cdot D_{y}+\la V(t,\sigma )\ra_{\Lambda
}\right) \varphi _{j}\left( t,\omega ,\sigma \right)
\end{equation*}%
with $\varphi _{j}|_{t=0}=\varphi _{j}^{0}$.\smallskip\

Now, suppose that Theorem \ref{t:semiclassicalObs} (ii) fails.
Therefore there exists $\mu \in \mathcal{\widetilde{M}}\left(
1/h\right) $ which satisfies condition (\ref{e:mnul}). Then there exists
$\Lambda \in \mathcal{L}$ such that
\begin{equation*}
\mu _{\Lambda }^{{\rm final}}\left( t,U\times \supp\chi \right) dt=0,
\end{equation*}%
for every $t\in \left( 0,T\right) $, but such that $\mu _{\Lambda
}^{{\rm final}}\not=0$.
This implies that $\bar{\mu}_{\Lambda}\neq 0$ and that, for $\bar{\mu}_{\Lambda}$-a.e. $%
\left( \omega ,\sigma \right) $ with $\Lambda \subseteq \Lambda _{\sigma }$:%
\begin{equation}
\mathrm{Tr}_{L_{\omega }^{2}(\mathbb{R}^{d},\Lambda )}\left(
\left\langle \mathbf{1}_{U}\right\rangle _{\Lambda }N_{\Lambda
}\left( t,\omega ,\sigma \right) \right) =0.  \label{e:heis}
\end{equation}%
Comparing with (\ref{e:eigenfm}), we obtain%
\begin{equation*}
\int_{0}^{T}\int_{U}\left\vert \varphi _{j}\left( t,\omega ,\sigma
\right) \right\vert ^{2}\left( y\right) dydt=0,
\end{equation*}%
for every $j$ such that $\lambda _{j}\neq 0$ $\bar{\mu}_{\Lambda}$-a.e.. Since $%
\bar{\mu}_{\Lambda}\neq 0$, $N_{\Lambda }\left( \cdot ,\omega ,\sigma
\right) \neq 0 $ on a set of positive $\bar{\mu}_{\Lambda}$-measure. This
implies that at least for one $j$, $\varphi _{j}\left( \cdot
,\omega ,\sigma \right) \neq 0$ and
therefore, the unique continuation property of Theorem \ref%
{t:semiclassicalObs} iii) fails for that choice of $\Lambda $,
$\omega $ and $\sigma $. This shows that iii) implies ii).

%%%%%%%%%%%%%%%%%%%%%%%%%%%%%%%%%%%%%%%%%%%%%%%%%%%%%%%%%%%%
%%%%%%%%%%%%%%%%%%%%%%%%%%%%%%%%%%%%%%%%%%%%%%%%%%%%%%%%%%%%%%

%%%%%%%%%%%%%%%%%%%%%%%%%%%%%%%%%%%%%%%%%%%
%%%%%%%%%%%%%%%%%%%%%%%%%%%%%%%%%%%%%%%%%%%%
\section{Appendix: Basic properties of Wigner distributions and semi-classical
measures}

\label{sec:Ap1}

In this Appendix, we review basic properties of Wigner distributions and
semiclassical measures. Recall that we have defined $w_{u_{h}}^{h}$ for
$u_{h}\in L^{2}\left(  \mathbb{T}^{d}\right)  $ as:
\begin{equation}
\int_{T^{\ast}\mathbb{T}^{d}}a(x,\xi)w_{u_{h}}^{h}(dx,d\xi)=\left\langle
u_{h},\Op_{h}(a)u_{h}\right\rangle _{L^{2}(\mathbb{T}^{d})},\qquad
\mbox{ for all }a\in\mathcal{C}_{c}^{\infty}(T^{\ast}\mathbb{T}^{d}
),\label{e:defWD}
\end{equation}
Start noticing that (\ref{e:defWD}) admits the more explicit expression:
\begin{equation}
\int_{T^{\ast}\mathbb{T}^{d}}a(x,\xi)w_{u_{h}}^{h}(dx,d\xi)=\frac{1}
{(2\pi)^{d/2}}\sum_{k,j\in\Z^{d}}\widehat{u_{h}}(k)\overline{\widehat{u_{h}
}(j)}\widehat{a}_{j-k}\left(  \frac{h}{2}(k+j)\right)  ,\label{e:Weylq}
\end{equation}
where $\widehat{u}_{h}(k):=\int_{\mathbb{T}^{d}}u_{h}(x)\frac{e^{-ik.x}}
{(2\pi)^{d/2}}dx$ and $\widehat{a}_{k}(\xi):=\int_{\mathbb{T}^{d}}
a(x,\xi)\frac{e^{-ik.x}}{(2\pi)^{d/2}}dx$ denote the respective Fourier
coefficients of $u_{h}$ and $a$, with respect to the variable $x\in
\mathbb{T}^{d}$.

\medskip

By the Calder\'{o}n-Vaillancourt theorem \cite{CV}, the norm of $\Op_{h}(a)$
is uniformly bounded in~$h$: indeed, there exists an integer $K_{d}$, and a
constant $C_{d}>0$ (depending on the dimension$d$) such that, if $a$ is a
smooth function on $T^{\ast}\IT^{d}$, with uniformly bounded derivatives,
then
\[
\norm{\Op_1(a)}_{{\mathcal L}(L^{2}(\IT^{d}))}\leq C_{d}\sum_{\alpha
\in\IN^{2d},|\alpha|\leq K_{d}}\sup_{T^{\ast}\IT^{d}}|\partial^{\alpha
}a|=:C_{d}M\left(  a\right)  .
\]
A proof in the case of $L^{2}(\IR^{d})$ can be found in
\cite{DimassiSjostrand}. As a consequence of this, equation (\ref{e:defWD})
gives:%
\[
\left\vert \int_{T^{\ast}\mathbb{T}^{d}}a(x,\xi)w_{u_{h}}^{h}(dx,d\xi
)\right\vert \leq C_{d}\left\Vert u_{h}\right\Vert _{L^{2}\left(
\mathbb{T}^{d}\right)  }^{2}M\left(  a\right)  ,\qquad\mbox{ for all }a\in
\mathcal{C}_{c}^{\infty}(T^{\ast}\mathbb{T}^{d}).
\]
Therefore, if $w_{h}(t,\cdot):=w_{S_{h}^{\tau_{h}t}u_{h}}^{h}$ for some
function $h\longmapsto\tau_{h}\in\mathbb{R}_{+}$ and if the family $\left(  u_{h}\right)  $
is bounded in $L^{2}\left(  \mathbb{T}^{d}\right)  $ one has that $\left(
w_{h}\right)  $ is uniformly bounded in $L^{\infty}\left(  \mathbb{R}
;\mathcal{D}^{\prime}\left(  T^{\ast}\mathbb{T}^{d}\right)  \right)  $. Let us consider
$\mu\in L^{\infty}\left(  \mathbb{R};\mathcal{D}^{\prime}\left(  T^{\ast
}\mathbb{T}^{d}\right)  \right)  $  an accumulation point of $\left(
w_{h}\right)  $ for the weak-$\ast$ topology.

\medskip

It follows from standard results on the Weyl quantization that $\mu$ enjoys
the following properties~:

\begin{enumerate}
\item[(a)] $\mu\in L^{\infty}(\R;\cM_{+}(T^{\ast}\mathbb{T}^{d}))$, meaning that
for almost all $t$, $\mu(t,\cdot)$ is a positive measure on $T^{\ast
}\mathbb{T}^{d}$.

\item[(b)] The unitary character of $S_{h}^{t}$ implies that $\int_{T^{\ast
}\mathbb{T}^{d}}\mu(t,dx,d\xi)$ does not depend on $t$; from the normalization
of $u_{h}$, we have $\int_{T^{\ast}\mathbb{T}^{d}}\mu(\tau,dx,d\xi)\leq1$, the
inequality coming from the fact that $T^{\ast}\mathbb{T}^{d}$ is not compact,
and that there may be an escape of mass to infinity. Such escape does not
occur if and only if $\left(  u_{h}\right)  $ is $h$-oscillating, in which
case $\mu\in L^{\infty}\left(  \mathbb{R};\mathcal{P}\left(  T^{\ast
}\mathbb{T}^{d}\right)  \right)  $.

\item[(c)] If $\tau_{h}\To\infty$ as $h\To0^{+}$ then the measures $\mu(t,\cdot)$
are invariant under $\phi_{s}$, for almost all $t$ and all $s$.

\item[(d)] Let $\bar{\mu}$ be the measure on $\mathbb{R}^{d}$ given by the  image of $\mu
(t,\cdot)$ under the projection map $(x,\xi)\longmapsto\xi$. If ${\mathbf V}_h(t)=\Op_h(V(t, x, \xi))$ is a pseudodifferential operator and if $\tau_h \ll h^{-2}$ then $\bar{\mu}$
does not depend on $t$. Moreover, if $\overline{\mu_{0}}$ stands for the image
under the same projection of any semiclassical measure corresponding to the
sequence of initial data $\left(  u_{h}\right)  $ then $\bar{\mu}
=\overline{\mu_{0}}$.
\end{enumerate}

For the reader's convenience, we next prove statements (c) and (d)  (see also
\cite{MaciaAv} for a proof of these results in the context of the
Schr\"{o}dinger flow $e^{iht\Delta}$ on a general Riemannian manifold). Let us
begin with the invariance through the Hamiltonian flow. We set
\[
a_{s}(x,\xi):=a(x+sdH(\xi),\xi)=a\circ\phi_{s}(x,\xi).
\]
The symbolic calculus for Weyl's quantization implies:
\begin{align*}
{\frac{d}{ds}}S_{h}^{s}\Op_{h}(a_{s})S_{h}^{s *} &  =S_{h}^{s}\Op_{h}
(\partial_{s}a_{s})S_{h}^{s *}-{\frac{i}{h}}S_{h}^{s}\left[  H(hD)+h^2 {\mathbf V}_h(t)\;,\;\Op_{h}
(a_{s})\right]  S_{h}^{s *}
&  =O(h).
\end{align*}
Therefore, for fixed $s$, $S_{h}^{s}\Op_{h}(a_{s})S_{h}^{s *}=\Op_{h}(a)+O(h)$ (note that we have only used here the boundedness of the operator ${\mathbf V}_h(t)$)  and for
$\theta\in L^{1}(\R)$,
\begin{align*}
\int_{\mathbb{R}}\theta(t)\left\langle w_{h}\left(  t\right)  ,a\right\rangle
dt &  =\int_{\mathbb{R}}\theta(t)\la u_{h}\;,\;S_{h}^{\tau_{h}t *}
\Op_{h}(a)S_{h}^{\tau_{h}t}u_{h}\ra dt\\
&  =\int_{\mathbb{R}}\theta(t)\la u_{h}\;,\;S_{h}^{\tau_{h}(t-s/\tau_{h}
) *}\Op_{h}(a\circ\phi_{s})S_{h}^{\tau_{h}(t-s/\tau_{h})}u_{h}\ra dt+O(h)\\
&  =\int_{\mathbb{R}}\theta(t+s/\tau_{h})\la u_{h}\;,\;S^{\tau_{h}t *}
\Op_{h}(a\circ\phi_{s})S^{\tau_{h}t}u_{h}\ra dt+O(h)\\
&  =\int_{\mathbb{R}}\theta(t+s/\tau_{h})\la w_{h}\left(  t\right)
,a\circ\phi_{s}\ra dt+O(h).
\end{align*}
Since $\Vert\theta(\cdot+s/\tau_{h})-\theta\Vert_{L^{1}}\To0$ (recall that we
have assumed that  $\tau_{h}\To\infty$ as $h\To0^{+}$) we obtain
\[
\int_{\mathbb{R}}\theta(t)\left\langle w_{h}\left(  t\right)  ,a\right\rangle
dt-\int_{\mathbb{R}}\theta(t)\la w_{h}\left(  t\right)  ,a\circ\phi_{s}\ra
dt\To0,\text{ as }h\To0^{+},
\]
whence the invariance under $\phi_{s}$.

\medskip

Let us now prove property (d). Consider $\overline{\mu}$ the image of $\mu$ by
the projection $(x,\xi)\longmapsto\xi$, we have for $a\in{\mathcal{C}}
_{0}^{\infty}(\R^{d})$ :
\begin{align*}
\left\langle w_{h}\left(  t\right)  ,a\left(  \xi\right)  \right\rangle
-\left\langle w_{u_{h}}^{h},a\left(  \xi\right)  \right\rangle  &  =\int
_{0}^{t}\frac{d}{ds}\langle w_{h}(s)\;,\;a(\xi)\rangle ds\\
&  =\int_{0}^{t}\la u_{h}\;,\;\frac{d}{ds}\left(  S_{h}^{\tau_{h}s *}
\Op_{h}(a)S_{h}^{\tau_{h}s}\right)  u_{h}\ra ds\\
&  = O\left(\tau_h h \norm{[ {\mathbf V}_h(t), \Op_h(a)]}_{{\mathcal L}(L^2(\T^d))}\right),
\end{align*}
(for $a$ only
depending on $\xi$ we have $\Op_{h}(a)=a(hD_{x})$, which commutes with
$H(hD_{x})$).
If ${\mathbf V}_h(t)=\Op_h(V(t, x, \xi))$ then
$$\norm{[ {\mathbf V}_h(t), \Op_h(a)]}_{{\mathcal L}(L^2(\T^d))}=O(h) \;\;{\rm and} \;\;\tau_h h \left\| \left[ {\mathbf V}_h(t), \Op_h(a)\right ]\right\|_{{\mathcal L}(L^2(\T^d)}=O(\tau_h h^2).$$ Therefore, if $\tau_h\ll h^{-2}$, we find, for every $\theta\in
L^{1}\left(  \mathbb{R}\right)  $:
\[
\int_{\mathbb{R}}\theta\left(  t\right)  \int_{T^{\ast}\mathbb{T}^{d}}a\left(
\xi\right)  \mu\left(  t,dx,d\xi\right)  =\left(  \int_{\mathbb{R}}
\theta\left(  t\right)  dt\right)  \int_{T^{\ast}\mathbb{T}^{d}}a\left(
\xi\right)  \mu_{0}\left(  dx,d\xi\right)  ,
\]
where $\mu_{0}$ is any accumulation point of $\left(  w_{u_{h}}^{h}\right)  $.
As a consequence of this, we find that $\overline{\mu}$ does not depend on $t$
and:
\[
\overline{\mu}\left(  \xi\right)  =\int_{\mathbb{T}^{d}}\mu_{0}\left(
dy,\xi\right)  .
\]

%%%%%%%%%%%%%%%%%%%%%%%%%%%%%%%%%%%%%%%%%%%%%%%%%%%%%%%%

\def\cprime{$'$} \def\cprime{$'$}

\end{document}